%% file: main.tex
\documentclass[11pt]{article}

\usepackage[ansinew, utf8]{inputenc}
\usepackage{amsmath,amsthm}
\usepackage{amssymb}
\usepackage{graphicx}
\usepackage{mathtools}
\usepackage{hyperref}
\usepackage{nicefrac}
\usepackage{bbm}
\usepackage{natbib}
\usepackage{anyfontsize}
\usepackage{xcolor}
\usepackage{graphicx}
\usepackage[total={6.5in,8.75in},
top=1.2in, left=0.9in, includefoot]{geometry}

\numberwithin{equation}{section}
\DeclareMathOperator*{\argmin}{arg\,min}
\newcommand{\Aug}{\mathrm{Aug}}
\newcommand{\R}{\mathbb{R}}
\newcommand{\PR}{\mathbb{P}}
\newcommand{\K}{\mathcal{K}}
\newcommand{\supp}{\mathrm{supp}}
\newcommand{\Pos}{\mathrm{Pos}}
\newcommand{\TZ}{\mathrm{TZ}}
\newcommand{\DZ}{\mathrm{DZ}}
\newcommand{\inte}{\mathrm{int}}
\newcommand{\Z}{\mathcal Z}
\newcommand{\BL}{\mathrm{BL}}
\newcommand{\Lip}{\mathrm{Lip}}
\newcommand{\E}{\mathbb{E}}
\newcommand{\Conv}{\mathrm{Conv}}
\newcommand{\tr}{\mathrm{tr}}
\newcommand{\proj}{\mathrm{proj}}

\newtheorem{thm}{Theorem}
\numberwithin{thm}{section}
\newtheorem{cor}[thm]{Corollary}
\newtheorem{lem}[thm]{Lemma}
\newtheorem{prop}[thm]{Proposition}
\newtheorem{defn}[thm]{Definition}
\newtheorem{rem}[thm]{Remark}
\newtheorem{exm}[thm]{Example}

\providecommand{\keywords}[1]{\textbf{\textit{Keywords}} #1}
\providecommand{\keywordsMSC}[1]{\textbf{\textit{MSC 2010 subject classification}} #1}
\newcommand{\footremember}[2]{\footnote{#2}\newcounter{#1}\setcounter{#1}{\value{footnote}}}
\newcommand{\footrecall}[1]{\footnotemark[\value{#1}]}

\begin{document}
\setlength{\parindent}{0pt}

\author{Marcel Klatt \footremember{ims}{\scriptsize Institute for Mathematical
		Stochastics, University of G\"ottingen,
		Goldschmidtstra{\ss}e 7, 37077 G\"ottingen} 
	\and 
	Axel Munk \footrecall{ims} \footnote{\scriptsize Max Planck Institute for Biophysical
		Chemistry, Am Fa{\ss}berg 11, 37077 G\"ottingen}
		\and 
	Yoav Zemel \footremember{cis}{\scriptsize Centre for Mathematical Sciences, University of Cambridge, Cambridge CB3 0WB}}

	\title{Limit Laws for Empirical Optimal Solutions \\in Stochastic Linear Programs}

\maketitle
\begin{abstract}
We consider a general linear program in standard form whose right-hand side constraint vector is subject to random perturbations. This defines a stochastic linear program for which, under general conditions, we characterize the fluctuations of the corresponding empirical optimal solution by a central limit-type theorem. Our approach relies on the combinatorial nature and the concept of degeneracy inherent in linear programming, in strong contrast to well-known results for smooth stochastic optimization programs. In particular, if the corresponding dual linear program is degenerate the asymptotic limit law might not be unique and is determined from the way the empirical optimal solution is chosen. Furthermore, we establish consistency and convergence rates of the Hausdorff distance between the empirical and the true optimality sets. As a consequence, we deduce a limit law for the empirical optimal value characterized by the set of all dual optimal solutions which turns out to be a simple consequence of our general proof techniques. \\
Our analysis is motivated from recent findings in statistical optimal transport that will be of special focus here.  In addition to the asymptotic limit laws for optimal transport solutions, we obtain results linking degeneracy of the dual transport problem to geometric properties of the underlying ground space, and prove almost sure uniqueness statements that may be of independent interest.
\end{abstract}

\keywords{Limit law, Linear programming, Optimal transport, Sensitivity analysis}

\keywordsMSC{Primary: 62E20, 90C15, 90C05 Secondary: 90C31, 49N15}

\input{Sections/Introduction}

\input{Sections/Preliminaries}

\input{Sections/MainResults}

\input{Sections/Proofs}

\input{Sections/Assumptions}

\input{Sections/OptimalTransport}

\section*{Acknowledgments}
M.\ Klatt and A.\ Munk gratefully acknowledge support from the DFG Research Training Group 2088 \textit{Discovering structure in complex data: Statistics meets Optimization and Inverse Problems}. Y.\ Zemel was supported in part by Swiss National Science Foundation Grant 178220, an in part by a U.K.\ Engineering and Physical Sciences Research Council programme grant.

\bibliographystyle{plainnat.bst}
\bibliography{Sections/main.bib}{}

\appendix

\input{Sections/AppendixA}

\input{Sections/AppendixB}

\input{Sections/AppendixC}

\input{Sections/AppendixD}

\end{document}

%% file: Sections/Introduction.tex
\section{Introduction}

Linear programs are constrained optimization problems where the objective function and the constraints are given by linear functions on a Euclidean space. Arising naturally in many applications, they have become ubiquitous in topics such as operations research, control theory, economics, physics, mathematics and statistics (see, e.g., the textbooks by \cite{bertsimas1997introduction}, \cite{luenberger2008linear},  \cite{galichon2018optimal} and the references therein). Their solid mathematical foundation dates back to the mid-twentieth century, to mention the seminal works of \cite{kantorovich1960mathematical}, \cite{hitchcock1941distribution} and \cite{dantzig1948programming} and its algorithmic computation is an active topic of research until today. In mathematical terms, a linear program in standard form writes
\begin{equation}\label{eq:standardLP}
\begin{aligned}
\min_{x\in \mathbb{R}^d} \quad
c^T x \qquad
\text{s.t.} \qquad
Ax = \,b,\quad
x \geq 0,
\end{aligned}\tag{$\text{P}_b$}
\end{equation}
with $(A,b,c)\in \mathbb{R}^{m\times d}\times \mathbb{R}^m\times \mathbb{R}^d$ and matrix $A$ of full rank $m\leq d$, and where, for the purpose of the paper, the lower subscript $b$ in \eqref{eq:standardLP} emphasizes the dependence on the vector $b$. At the heart of linear programming is the observation that the optimum must be attained on a finite set of feasible points termed basic feasible solutions. Each such point is identified by a basis $I\subset\{1,\ldots,d\}$ indexing $m$ linearly independent columns of the constraint matrix $A$. In fact, the simplex algorithm \citep{dantzig1948programming,dantzig1951maximization}, among the most well-known algorithm to solve \eqref{eq:standardLP}, is specifically tailored to move from one basic feasible solution to another whilst improving the objective value at each step.

Shortly after first algorithmic approaches and theoretical results became available, the need to incorporate uncertainity in the parameters has become apparent (see \cite{dantzig1955uncertainty,beale1955minimizing,ferguson1956allocation} for early contributions). In fact, apart from its relevance in numerical stability issues, in many practical applications certain parameters are not known exactly and instead estimated empirically. Hence, accounting for randomness in linear programs is critical for many problems and encouraged the development of \emph{stochastic linear programming} in which some parameters in \eqref{eq:standardLP} are subject to (possibly random) perturbations (see, e.g., \cite{kall1976stochastic,ruszczynski2003stochastic}. Nevertheless, studying the behavior of the estimated optimal values and the corresponding estimated optimal solutions appears more common in stochastic \emph{nonlinear} programs, where the objective function and constraints are nonlinear functions of $x\in\mathbb R^d$ (see \cite{dupavcova1987stochastic,dupacova1988asymptotic,shapiro1989asymptotic,shapiro1991asymptotic,shapiro1993asymptotic,king1993asymptotic} and references therein). Regularity assumptions such as second order growth conditions for the functions defining the optimization problem allow for either explicit asymptotic expansions of optimal values and optimal solutions or applications of implicit function theorems and generalizations thereof. More recent studies include analytical properties of the optimal solution or its objective value such as continuity and differentiability, as well as statistical implications (see \cite{romisch2003stability,eichhorn2007stochastic,klatt2020empirical} and references therein).

In this paper, we focus on statistical aspects for the standard \emph{linear} program \eqref{eq:standardLP}, where these regularity assumptions fail. Specifically, we consider the case where the right-hand side constraint vector $b\in\mathbb{R}^m$ is replaced by a (consistent) estimator $b_n$ indexed in $n\in \mathbb{N}$, e.g., based on $n$ observations. Such a randomness in the parameter $b$ reflects practical needs, as it usually models budget, prices or capacities that are often not known in advance. The main goal of this paper is to characterize the statistical fluctuation of the empirical optimal solution 
\begin{equation}\label{eq:introSolution}
x^\star(b_n) = \argmin_{Ax=b_n,\, x\geq 0} c^Tx
\end{equation}
by a central limit-type theorem. Our approach is based on a careful study of possible changes of bases depending on small (random) perturbations of $b$. This is achieved by considerations of the corresponding \textit{dual} linear program to \eqref{eq:standardLP} given by
\begin{equation}\label{eq:standardDP}
\begin{aligned}
\max_{\lambda \in \mathbb{R}^m} \quad
b^T \lambda \qquad
\text{s.t.} \quad
\lambda^T A \leq\, c^T.
\end{aligned}\tag{$\text{D}_b$}
\end{equation}
Most notably, a basis $I\subset \{1,\ldots,d\}$, i.e., a collection of $m$ independent columns of the constraint matrix $A$, does not only define a basic solution for the primal program \eqref{eq:standardLP} but (possibly) also for the dual \eqref{eq:standardDP}. In fact, our results show that the stochastic behavior of $x^\star(b_n)$ in \eqref{eq:introSolution} is inextricably linked to the collection of all bases $I$ that induce optimal solutions to both the primal and the dual problem. The collection of such bases depends on $b$ and this dependence can be rather complex. Being the key notion behind the simplex method, the understanding of the behavior of the collection of optimal bases has been studied under different names in the literature. A first contribution is \cite{wets1966programming} and \cite{walkup1969lifting} \emph{basis decomposition theorem} that describes the behavior for the optimal value for \eqref{eq:standardLP} as a function on the parameter $b$. In algebraic geometry such statement is closely related to a \emph{cone-triangulation} \citep{sturmfels1997variation,de2010triangulations} of the primal feasible optimization region. Feasible basis changes are also fundamental in analyzing linear programming algorithms such as the dual simplex method \citep{bertsimas1997introduction}.  Lastly, dealing with changes in certain parameters for \eqref{eq:standardLP} is nowadays subsumed as \emph{sensitivity analysis} for linear programming or in the special case of our bases driven approach, \emph{basis invariancy} \citep{greenberg1986analysis,ward1990approaches,hadigheh2006sensitivity}.

The main results of this paper stated in Theorem \ref{thm:limitGeneral} and Theorem \ref{thm:limit} describe the statistical fluctuation of the estimated optimal solution $x^\star(b_n)$ in \eqref{eq:introSolution} around its population version $x^\star(b)$ (after proper standardization) by a central limit-type theorem. Under suitable assumptions (see Section \ref{sec:Preliminaries} for details) we find that 
\begin{equation}\label{eq:introTHM}
r_n\left(x^{\star}(b_n) - x^{\star}(b)\right)
\xrightarrow[]{D} M(G)\, ,
\end{equation}
where \begin{small}$\xrightarrow[]{D}$\end{small} denotes weak convergence \citep{billingsley1999convergence} and $G$ is the weak limit random variable of $r_n(b_n-b)$ with $r_n\to\infty$ as $n$ tends to infinity. A prototypical example is the central limit theorem, whereby $r_n=\sqrt{n}$ and $G$ is a Gaussian random vector on $\mathbb R^m$. The limit law in \eqref{eq:introTHM} is then given by a (possible random) function $M$ evaluated at $G$ and its explicit form is captured on whether and to what extent degeneracy is present in the primal and dual optimal solutions. More precisely, one distinguishes three cases of increasing complexity for $M$. The first and simplest case occurs if the primal optimal solution $x^\star(b)$ is nondegenerate.  In this case the function $M$ is a linear transformation depending on the unique optimal basis for \eqref{eq:standardLP}. In the central limit theorem case where $G$ is Gaussian, the limit law in \eqref{eq:introTHM} will consequently be Gaussian, too (see Theorem \ref{thm:limitGeneral}). When $x^\star(b)$ is degenerate but the dual optimal basic solutions for \eqref{eq:standardDP} are not, then $M$ is a sum of (deterministic) linear transformations defined on cones that are indexed by the collection of dual optimal bases (see Theorem \ref{thm:limit}). Specifically, the number of summands in $M$ is equal the number of dual optimal basic solutions for \eqref{eq:standardDP}. The last and most complicated case arises when both $x^\star(b)$ and some dual optimal basic solutions exhibit degeneracy. In this setting the function $M$ is still a sum of linear transformations defined on cones. However, these transformations are potentially random and indexed by certain \textit{subsets} of the set of optimal bases. In comparison with the previous case, these subsets do not only consist of singletons and indeed the collection of those subsets reflect the complicated combinatorial nature in linear programming under degeneracy (see Theorem \ref{thm:limitGeneral}).  Note that, as a consequence of usual sensitivity analysis for the nondegenerate case \citep[see for example][Section 4.4]{luenberger2008linear}, the limit law is a linear function of $G$.  This is no longer true when denegeracy is present, and, to the best of our knowledge, this paper is the first that covers limit laws for general linear programs under degeneracy.

The central limit theorem in \eqref{eq:introTHM} is, of course, only reasonable if the primal optimal solution $x^\star(b)$ for \eqref{eq:standardLP} is unique. Still, interesting results can be established by our bases driven approach when such uniqueness fails. First, we establish consistency and convergence rates of the Hausdorff distance between the empirical and the true optimality sets (see Theorem \ref{thm:Hausdorff}). Second, denoting by $c(b)$ the optimal objective value for \eqref{eq:standardLP}, we provide a general distributional result for the empirical optimal value
\begin{equation}\label{eq:introValue}
r_n(c(b_n) - c(b))
\xrightarrow[]{D}
\max_{\substack{\lambda(I) \text{ dual optimal} \\ \text{basic solution for \eqref{eq:standardDP}}}}G^T\lambda(I)
\end{equation}
(see Proposition \ref{prop:cltoptimalvalue}). The limit law \eqref{eq:introValue} depends on the set of all dual optimal basic solutions and this again turns out to be a simple consequence of our bases driven approach. 

One of the most important instances of linear programming is \textit{optimal transport}, i.e., how to transport goods in the most efficient (e.g., economically or physically) manner. With a rich history in economics and  mathematics \citep{vershik2013long}, recent computational progress paved the way to explore novel fields of application and optimal transport achieved great interest in imaging \citep{rubner2000earth,solomon2015convolutional}, machine learning \citep{frogner2015learning,arjovsky2017wasserstein}, and statistical data analysis \citep{chernozhukov2017monge,sommerfeld2018inference,delBarrio2019robust,peyre2019computational,panaretos2019statistical}. In fact, our statistical analysis for general stochastic linear programs in standard form is motivated by recent findings in statistical optimal transport. In particular, while central limit theorems for the empirical optimal transport cost are well investigated \citep[see e.g.,][]{delBarrio1999tests,tameling2019empirical,del2019central}, the statistical behaviour of their corresponding empirical optimal transport solutions remains largely open. An exception is \cite{klatt2020empirical}, who provide limit theorems for \emph{(entropy) regularized} optimal transport solutions, thus modifying the underlying linear program to be strictly convex, nonlinear and most importantly nondegenerate in the sense that every regularized optimal transport solution is strictly positive in each coordinate. Hence, an implicit function theorem approach in conjunction with a delta method allows concluding for Gaussian limits in this case. This stands in stark contrast to the unregularized optimal transport considered in this paper, where the degenerate case is generic rather than the exception for most practical situations. More precisely, only if the optimal transport solution is nondegenerate then we observe a Gaussian fluctuation on the support set, i.e., on all entries with positive values. If the optimal transport solution is degenerate, then the asymptotic limit law \eqref{eq:introTHM} is not Gaussian anymore. Degeneracy in optimal transport problems easily occurs as soon as certain subsets of demand and supply sum up to the same quantity. In particular, we encounter the highest possible degree of degeneracy if individual demand is equal to individual supply. Additionally, we obtain necessary and sufficient conditions on the cost function in order for the dual optimal transport to be nondegenerate. This allows to prove almost sure uniqueness statements that may be of interest in their own. 

Our distributional results can be viewed as a basis for uncertainty quantification and other statistical inference procedures concerning solutions to linear programs. This is illustrated by the distributional laws of various regular functionals of the optimal solution, which follow easily from our theory (see Section \ref{subsec:twosample} for first examples). A detailed study of their statistical consequences such as in testing theory or for confidence statements presents an important avenue for future research.

The outline of the paper is as follows. We first recall basics for linear programming in Section \ref{sec:Preliminaries} and introduce deterministic and stochastic assumptions on the linear program \eqref{eq:standardLP} and the random fluctuation $b_n$ of the constraint vector $b$, respectively. Our main results are summarized in Section \ref{sec:MainResults}, followed by their proofs in Section \ref{sec:Proofs}. The assumptions are discussed in more detail in Section \ref{sec:Assumptions}. Section \ref{sec:OT} focuses on the specific case of the optimal transport problem. Apart from the self-contained proofs of the main results in Section \ref{sec:Proofs}, for the sake of readability most proofs from Sections \ref{sec:Preliminaries}, \ref{sec:Assumptions} and \ref{sec:OT} are given in Appendices \ref{Appendix:LinearProgramProofs}, \ref{Appendix:Assumptions} and \ref{Appendix:OT}, respectively. 

%% file: Sections/Preliminaries.tex
\section{Preliminaries and Assumptions}\label{sec:Preliminaries}

This section recalls basics of linear programming and introduces notation and assumptions required to state the main results of the paper. Proofs for statements in this section are either contained in Appendix \ref{Appendix:LinearProgramProofs} or a reference is given. For further details we encourage the reader to consult standard textbooks on linear programs such as \cite{bertsimas1997introduction}, \cite{sierksma2001linear} and \cite{luenberger2008linear}.\\
\textbf{Linear programs and duality.} At the heart of linear programming is the correspondence between the primal \eqref{eq:standardLP} and its \emph{dual} linear program \eqref{eq:standardDP}. To see this, let the columns of the matrix $A$ be indexed by the set $[d]\coloneqq \{1,\ldots,d\}$. For an index set $I\subseteq [d]$ let $A_I\in \mathbb{R}^{m\times \vert I \vert}$ be the sub-matrix of $A$ formed by the corresponding columns indexed by $I$. Similarly, a vector $x_I\in \mathbb{R}^{\vert I \vert}$ is a sub-vector of $x\in\mathbb{R}^d$ that only consists of coordinates indexed by $I$. By full rank of $A$ there exists at least one index set $I$ with cardinality $m$ such that $A_I\in \mathbb{R}^{m\times m}$ is one-to-one. An index set with that property is said to be a \emph{basis}. The fact that $A_I$ is one-to-one means that the linear equation $\lambda^T A_I=c_I$ has a unique solution
\begin{equation*}
\lambda(I)\coloneqq (A_I)^{-T}c_I \in \mathbb{R}^m
\end{equation*}
referred to as a \emph{dual basic solution}. Notice that $\lambda(I)$ is not necessarily feasible for \eqref{eq:standardDP} as we only enforced the subset $I\subseteq [d]$ of constraints to be satisfied by equalities. If dual feasibility holds $\lambda(I)^TA\leq c^T$ then $\lambda(I)$ is said to be a \textit{dual basic feasible solution} with \emph{dual feasible basis} $I$. If $\lambda(I)$ is an optimal solution, i.e., feasible and maximizes the objective in \eqref{eq:standardDP} among all dual feasible solutions, then it is referred to as a \emph{dual optimal basic solution}. Similarly, for the primal program \eqref{eq:standardLP}, for each basis $I$ consider the linear equation $A_I x_I = b$ with unique solution $x_I \in \mathbb{R}^{m}$. In order to match dimensions (a solution for the primal has dimension $d$ instead of $m\leq d$), we augment the solution with coordinates indexed by $I^c\coloneqq [d]\setminus I$ whose values are set to zero. Hence, for each basis $I$ this yields the vector 
\begin{equation*}
x(I,b)\coloneqq \text{Aug}_I\left[\left(A_I\right)^{-1}b \right]\in \mathbb{R}^d
\end{equation*}
denoted as a \emph{primal basic solution}, where $\Aug_I\colon\R^m\to\R^d$ is the operator that sets zeroes in the coordinates that are not in $I$. Notice that $\Aug_I$ is a linear operator, i.e., for $b_1,b_2\in \mathbb{R}^m$ it holds that $x(I,b_1-b_2)=x(I,b_1)-x(I,b_2)$. Again the primal basic solution $x(I,b)$ is not necessarily feasible for \eqref{eq:standardLP} as some coordinates might be negative. If $x(I,b)\geq 0$ and hence primal feasibility holds then $x(I,b)$ is said to be a \emph{primal basic feasible solution} with \textit{primal feasible basis} $I$. If additionally $x(I,b)$ is optimal, i.e., feasible and minimizes the objective function in \eqref{eq:standardLP} among all primal feasible solutions, it is said to be a \emph{primal optimal basic solution}. The \emph{fundamental theorem of linear programming} \cite[Section 2.4]{luenberger2008linear} addresses the existence of such an (optimal) feasible basis $I$. We again emphasize that throughout $A$ is assumed to have full rank.

\begin{thm}\label{thm:optFeas}
Consider the primal linear program \eqref{eq:standardLP}.
\begin{itemize}
\item[(i)] If there exists a feasible solution, there exists a primal feasible basis $I\subseteq [d]$ such that $x(I,b)$ is a primal basic feasible solution.
\item[(ii)] If there exists an optimal solution, there exists a primal feasible basis $I\subseteq [d]$ such that $x(I,b)$ is a primal optimal basic solution.
\end{itemize}
Moreover, the same statement holds for the dual linear program \eqref{eq:standardDP}.
\end{thm}
In view of Theorem \ref{thm:optFeas}, feasible and optimal solutions for primal \eqref{eq:standardLP} and dual program \eqref{eq:standardDP} can be found by considering the collection of all possible bases. Each basis $I$ contains those column indices from the coefficient matrix $A\in \mathbb{R}^{m\times d}$ such that the sub-matrix $A_I\in\mathbb{R}^{m\times \vert I\vert}$ is invertible. Hence, a trivial upper bound on the number of bases is $\binom{d}{m}=\frac{d!}{m!(d-m)!}$. This is a finite quantity but grows exponentially fast in $m$ and $d$. In general, a basis $I$ might be dual feasible while on the contrary it does not constitute a primal feasible basis and vice versa. However, we have the following statement known as \emph{strong duality} \cite[Section 4.2]{luenberger2008linear}.

\begin{thm}\label{thm:strongduality}
Consider the primal linear program \eqref{eq:standardLP} and its dual \eqref{eq:standardDP}.
\begin{itemize}
\item[(i)] If either of the linear programs \eqref{eq:standardLP} or \eqref{eq:standardDP} has a finite optimal solution, so does the other and the corresponding optimal values are the same.
\item[(ii)] If for a basis $I\subseteq [d]$ the vector $\lambda(I)$ is dual feasible and $x(I,b)$ is primal feasible, then both are primal and dual optimal basic solutions, respectively.
\end{itemize}
\end{thm}
In view of the preceding two theorems, we see that to each linear program \eqref{eq:standardLP} and \eqref{eq:standardDP} is associated a collection of feasible bases that are possibly but not necessarily linked by strong duality. The question arises if there always exists a common basis $I$ such that $x(I,b)$ and $\lambda(I)$ are both primal and dual basic feasible  solutions and hence also optimal, respectively. To answer that question we introduce
\begin{equation}\label{eq:sets}
\begin{split}
P(b)&\coloneqq \left\{x\in\R^d \, \mid\, Ax=b,\, x\geq 0\right\},
\\
OPT(b)
&\coloneqq \left\{x^{\star}\in P(b)\, \mid \, c^Tx^{\star}=\inf_{x\in P(b)}c^Tx\right\}
\end{split}
\end{equation}
the feasibility and optimality set for the linear program \eqref{eq:standardLP}, respectively. To alleviate the notation, and since $A$ and $c$ will generally be fixed, the dependence of $P(b)$ and $OPT(b)$ on $A$ and $c$ is suppressed. We state our first assumption.
\begin{equation}\label{ass:optBounded}\tag{\textbf{A1}}
\textit{The set } OPT(b) \textit{ is non-empty and bounded.}
\end{equation}
Recall that the convex hull of a collection of vectors $\{x_1,\ldots,x_K\}\subset \mathbb{R}^d$ is the set of all possible convex combinations of them.

\begin{lem}\label{lem:rulesOnly}
Consider the primal linear program \eqref{eq:standardLP} and assume \eqref{ass:optBounded} holds. Then for any $\tilde{b}\in \mathbb{R}^m$ either one of the following statements is correct.
\begin{itemize}
\item[(i)] The feasible set $P(\tilde{b})=\{x\in\R^d\, \mid \,Ax=\tilde{b},x\geq 0\}$ is empty.
\item[(ii)] The set of minimizers $OPT(\tilde{b})$ is non-empty and bounded. Moreover, it is equal to the convex hull of the set
\begin{equation*}
\left\lbrace x(I,\tilde{b})\, \mid\, I \text{ primal and dual feasible basis for $(\text{P}_{\tilde{b}})$ and $(\text{D}_{\tilde{b}})$ }\right\rbrace.
\end{equation*}
\end{itemize}
\end{lem}
Notice that the preceding lemma proves that in order to find a primal optimal basic solution it suffices to consider all dual feasible bases $I$ and check whether $x(I,b)$ is primal feasible. This observation leads to our bases driven approach underlying the analysis for limit laws of empirical optimal solutions.

\begin{rem}[Splitting of the Bases Collection]\label{rem:basessplitting}
Suppose that $I_1,\ldots,I_N$ are all the dual feasible bases for the dual linear program \eqref{eq:standardDP}, i.e., $\lambda(I_j)^T=c_{I_j}^TA_{I_j}^{-1}$ constitutes a dual basic feasible solution for $1\leq j \leq N$. We can partition this collection of bases into two subsets containing those bases that also induce a primal feasible basic solution for \eqref{eq:standardLP}, i.e., 
\begin{equation*}
I_1,\ldots,I_K \text{ induce primal and dual basic feasible solution } (K\leq N)
\end{equation*}
and those that lead to primal basic infeasible solutions
\begin{equation*}
I_{K+1},\ldots,I_N \text{ induce dual basic feasible but primal basic infeasible solution}.
\end{equation*}
Notice by Theorem \ref{thm:strongduality} that $x(I_k,b)$ is a primal basic optimal solution for all $k\leq K$.
\end{rem}
We also use the abbreviation $x^{\star}(b)$ to denote any optimal solution for the primal program \eqref{eq:standardLP}. An important assumption for our central limit theorem will be the following.
\begin{equation}\label{ass:uniquexb}\tag{\textbf{A2}}
 \textit{An optimal solution }x^\star(b) \textit{ for \eqref{eq:standardLP} exists and is unique.}
\end{equation}  
Clearly, assumption \eqref{ass:uniquexb} implies assumption \eqref{ass:optBounded}. Finally, we recall the definition for degeneracy of primal and dual basic feasible solutions. A primal basic feasible solution $x(I,b)$ is \emph{degenerate} if less than $m$ of its coordinates are nonzero. Similarly, a dual basic feasible solution $\lambda(I)$ is degenerate if more than $m$ of the $d$ inequalities $\lambda(I)^TA\leq c$ hold as equalities. The following proposition links the concept of degeneracy of optimal solutions for a linear program to uniqueness of optimal solutions for its related dual linear program and vice versa.

\begin{prop}\label{prop:slackness}
Consider the linear program \eqref{eq:standardLP} and its dual \eqref{eq:standardDP}.
\begin{enumerate}
\item[(i)] If \eqref{eq:standardLP} (resp. \eqref{eq:standardDP}) has a nondegenerate optimal basic solution, then \eqref{eq:standardDP} (resp. \eqref{eq:standardLP}) has a unique solution.
\item[(ii)] If \eqref{eq:standardLP} (resp. \eqref{eq:standardDP}) has a unique nondegenerate optimal basic solution, then \eqref{eq:standardDP} (resp. \eqref{eq:standardLP}) has a unique nondegenerate optimal solution.
\item[(iii)] If \eqref{eq:standardLP} (resp. \eqref{eq:standardDP}) has a unique degenerate optimal basic solution, then \eqref{eq:standardDP} (resp. \eqref{eq:standardLP}) has multiple solutions.
\end{enumerate}
\end{prop} 
Many fundamental results in linear programming simplify when the optimal solutions are nondegenerate. This effect is even more remarkable in our stochastic analysis, as will be seen below. We introduce the assumption
\begin{equation}\label{ass:dualnondegen}\tag{\textbf{A3}}
 \lambda(I_j)\neq \lambda(I_k), \, 1\leq j<k\leq K,
\end{equation} 
where again $I_1,\ldots,I_K$ enumerate all bases $I$ such that $\lambda(I)$ is a dual basic feasible and $x(I,b)$ is a primal optimal basic solution (see Remark \ref{rem:basessplitting}). Assumption  \eqref{ass:dualnondegen} is weaker than nondegeneracy of all optimal dual basic solutions for \eqref{eq:standardDP}.
\begin{lem}\label{lem:dualnondegenerate}
Suppose \eqref{ass:optBounded} holds. Then assumption \eqref{ass:dualnondegen} is equivalent to nondegeneracy of all dual optimal basic solutions.
\end{lem}
Assumptions \eqref{ass:optBounded}, \eqref{ass:uniquexb} and \eqref{ass:dualnondegen} are purely deterministic and only depend on the parameters $(A,b,c)\in \mathbb{R}^{m\times d}\times \mathbb{R}^m\times \mathbb{R}^d$ defining the primal linear program.\\
\textbf{Stochastic setting.} Introducing randomness in problems \eqref{eq:standardLP} and \eqref{eq:standardDP}, we suppose to have incomplete knowledge of the vector $b\in\R^m$ and replace it by a (consistent) estimator $b_n$, e.g., based on a sample of size $n$ independently drawn from a distribution with mean $b$. This defines empirical primal and dual counterparts (\hyperref[eq:standardLP]{$\text{P}_{b_n}$}) and (\hyperref[eq:standardDP]{$\text{D}_{b_n}$}), respectively. We allow the more general case that only the first $m_0\in\{1,\dots,m\}$ coordinates\footnote{One may assume at first reading that $m_0=m$; the additional generality will turn useful for the one-sample case naturally arising in optimal transport in Section \ref{sec:OT}.} of $b$ are unknown and assume the existence of a sequence of random vectors $b_n=(b_n^{m_0},[b]_{m-m_0})\in\R^{m_0}\times\R^{m-m_0}$ converging to $b$ at rate $\frac{1}{r_n}\to 0$ as $n$ tends to infinity:
\begin{equation}\label{ass:cltforb}\tag{\textbf{B1}}
\begin{split}
&G_n^{m_0}\coloneqq r_n(b_n^{m_0}-b)\xrightarrow[]{D} G=(G^{m_0},0_{m-m_0})\\
 \textit{ with } &G^{m_0} \textit{ absolutely continuous w.r.t. Lebesgue measure on } \mathbb{R}^{m_0},
\end{split}
\end{equation} 
where \begin{small}$\xrightarrow[]{D}$\end{small} denotes convergence is distribution. In a typical central limit theorem type scenario (see Section \ref{sec:OT}), $r_n=\sqrt{n}$ and $G^{m_0}$ is a centred Gaussian random vector in $\mathbb{R}^{m_0}$, assumed to have a nonsingular covariance matrix. Whenever $m=m_0$, we suppress dependency of $m_0$ and write $b_n^{m_0}=b_n$ and $G^{m_0}=G$. Notice that assumption \eqref{ass:cltforb} implies $b_n$ to be a (weakly) consistent estimator for $b$ meaning that for any $\epsilon>0$ the probability $\PR(\Vert b_n-b\Vert>\epsilon)$ converges to zero as $n$ tends to infinity. To avoid pathological cases we impose the last assumption that asymptotically an optimal solution $x^\star(b_n)$ for the primal $(\text{P}_{b_n})$ exists.

\begin{equation}\label{ass:feasforn}\tag{\textbf{B2}}
\lim_{n\to\infty}\PR\left(x^{\star}(b_n) \textrm{ exists}\right)= 1.
\end{equation}
We discuss all stated assumptions and their implications in more detail in Section \ref{sec:Assumptions}.  

%% file: Sections/MainResults.tex
\section{Main Results}\label{sec:MainResults}

According to Remark \ref{rem:basessplitting} we may split the collection of all dual feasible bases $I_1,\ldots,I_N$ for \eqref{eq:standardDP} by those that are also primal feasible bases $I_1,\ldots,I_K$ for \eqref{eq:standardLP} and those that are only dual feasible bases $I_{K+1},\ldots,I_N$. In view of Lemma \ref{lem:rulesOnly}, whenever an optimal solution $x^{\star}(b_n)$ for (\hyperref[eq:standardLP]{$\text{P}_{b_n}$}) exists, it takes the form
\begin{equation*}
x^{\star}(b_n)
=\sum_{k\in\mathcal K}
(\alpha_n^\K)_k x(I_k,b_n)
\coloneqq \alpha_n^\K\otimes x(I_\K,b_n),
\end{equation*}
where $\K$ is a non-empty subset of $[N]\coloneqq \{1,\dots,N\}$ and $\alpha_n^\K$ is a random vector in the unit simplex $\Delta_{|\K|}\coloneqq \left\lbrace\alpha\in\R_+^{|\K|}\, \mid \,\|\alpha\|_1=1\right\rbrace$.

\subsection{Distributional Limits}

\begin{thm}\label{thm:limitGeneral}
Suppose assumptions \eqref{ass:uniquexb}, \eqref{ass:cltforb}, and \eqref{ass:feasforn} hold, and let $x^{\star}(b_n)$ be any (measurable) choice of an optimal solution for (\hyperref[eq:standardLP]{$\text{P}_{b_n}$}). Further, assume that for all $\K$, the random vector $\left(\alpha_n^\K, G_n\right)$ converges jointly in distribution as $n$ tends to infinity to $(\alpha^\K,G)$ on $\Delta_{\vert \K\vert}\times \mathbb{R}^m$. Then there exist closed convex cones $H_1,\dots,H_K\subseteq\R^{m_0}$, each of which is an intersection of $m-\vert\supp\, x^{\star}(b)\vert$ half-spaces in $\R^{m_0}$, passing through the origin such that
\begin{equation*}
r_n\left(x^{\star}(b_n) - x^{\star}(b)\right)
\xrightarrow[]{D} M(G^{m_0})\coloneqq
\sum_\K \mathbbm{1}_{G^{m_0}\in H_\K\setminus \cup_{k\notin\K}H_k}\,\alpha^\K \otimes x(I_\K,G)
\, \in \R^d.
\end{equation*}
The sum runs over non-empty subsets $\K$ of $[K]=\{1,\dots,K\}$ and $H_\K\coloneqq\cap_{k\in\K}H_k$. \\
In particular, if the primal optimal solution $x^{\star}(b)$ is nondegenerate, then $K=1$, there exists a unique basis $I_1$, and the limit reads as 
\begin{equation*}
r_n\left(x^{\star}(b_n) - x^{\star}(b)\right)
\xrightarrow[]{D} x(I_1,G)
\, \in\R^d.
\end{equation*}
Hence, it is an (invertible) linear function of $G$ and Gaussian if $G$ is Gaussian.
\end{thm}
In Section \ref{sec:Assumptions} we discuss sufficient conditions for the joint distributional convergence of the random vector $\left( \alpha_n^\K,G_n\right)$. In short, if we use any linear program solver, such joint distributional convergence appears to be reasonable.\\
Notice that the structure of the limit law depends on the degree of degeneracy of the primal optimal solution $x^\star(b)$ for \eqref{eq:standardLP}. If $x^\star(b)$ is degenerate, then the sum in the limit law can consist of several summands. In contrast, for a nondegenerate optimal solution $x^\star(b)$, the limit law is simple and might even be a $d$-dimensional Gaussian random variable (with effective dimension $m$). In between these two cases is the situation that assumption \eqref{ass:dualnondegen} holds, which is related to the case that all dual optimal basic solutions for \eqref{eq:standardDP} are nondegenerate (see Lemma \ref{lem:dualnondegenerate}). The limit law can then be simplified, as the subsets $\K$ have to be singletons. In fact, the number of summands is exactly equal the number of dual optimal basic solutions.
\begin{thm}\label{thm:limit}
Suppose assumptions \eqref{ass:uniquexb}, \eqref{ass:dualnondegen}, \eqref{ass:cltforb}, and  \eqref{ass:feasforn} hold. Then any\footnote{There is no need to assume joint distributional convergence of $(\alpha_n^\K,G_n)$ as in Theorem \ref{thm:limitGeneral}.} (measurable) choice of $x^{\star}(b_n)$ satisfies
\begin{equation*}
r_n\left(x^{\star}(b_n) - x^{\star}(b)\right)
\xrightarrow[]{D} \sum_{k=1}^K \mathbbm{1}_{G\in H_k\setminus \cup_{j<k}H_j}\,x\left(I_k,G\right)
\, \in\R^d
\end{equation*}
with the $H_k$'s as given in Theorem \ref{thm:limitGeneral}.
\end{thm}
\begin{rem}
In Theorem \ref{thm:limit}, absolute continuity of the limiting random variable $G$ (or $G^{m_0}$) is not required. Indeed, $G$ can be an arbitrary random vector, and Theorem \ref{thm:limit} thus accommodates, e.g., Poisson limit distributions. 
If $G$ is absolutely continuous then the indicator functions simplify to $\mathbbm{1}_{G\in H_k}$ instead of $\mathbbm{1}_{G\in H_k\setminus \cup_{j<k}H_j}$, as intersections $H_k\cap H_j$ have Lebesgue measure zero (see Section \ref{sec:Proofs}).
\end{rem}

\subsection{Proximity of the Empirical Optimality Sets}
When multiple primal optimal solutions exist, we can still obtain proximity of the empirical optimality set $Opt(b_n)$ to $Opt(b)$ in Hausdorff distance
\begin{equation*}
d_H\left(Opt(b_n),Opt(b)\right)\coloneqq \max \left\lbrace\sup_{x\in OPT(b_n)}\inf_{y\in OPT(b)}\Vert x-y\Vert, \sup_{x\in OPT(b)}\inf_{y\in OPT(b_n)}\Vert x-y\Vert\right\rbrace.
\end{equation*}
The next theorem proves that the Hausdorff distance converges to zero.  The rate of convergence is precisely the same as that of $b_n$, namely $r_n$.

\begin{thm}\label{thm:Hausdorff}
Suppose assumptions \eqref{ass:optBounded} and \eqref{ass:feasforn} hold. Further, let $b_n$ be a (weakly) consistent estimator for $b$ such that $\Vert b_n-b\Vert= O_\PR(r_n^{-1})$. Then the Hausdorff distance between $OPT(b_n)$ and $OPT(b)$ is bounded in probability and in particular
\begin{equation*}
d_H\left(Opt(b_n),Opt(b)\right)=O_\PR (r_n^{-1}).
\end{equation*}
\end{thm}

\subsection{The Empirical Optimal Objective Value}\label{subsec:Limitvalue}

Limit laws for the empirical optimal objective value in the special case of optimal transport (see Section \ref{sec:OT}) have recently been of particular interest in \cite{sommerfeld2018inference,tameling2019empirical}. Our theory extends their findings to more general standard linear programs. Furthermore, compared to their approach based on a functional delta method for Hadamard directional differentiable functionals, our proof is elementary and makes obvious the dependency on all dual optimal basic solutions for \eqref{eq:standardDP}. Recall that we denote by $c(b)$ the primal optimal value for the linear program \eqref{eq:standardLP}.

\begin{prop}\label{prop:cltoptimalvalue}
Suppose assumption \eqref{ass:optBounded}, \eqref{ass:cltforb} and \eqref{ass:feasforn} hold. Then, as $n$ tends to infinity
\begin{equation*}
r_n(c(b_n) - c(b))
\xrightarrow[]{D}
\max_{1\leq k \leq K}G^T\lambda(I_k),
\end{equation*}
where $r_n(b_n-b)\xrightarrow[]{D} G$.
\end{prop}
Proposition \ref{prop:cltoptimalvalue} shows that the limit law for the empirical optimal objective value only depends on the maximum of all dual optimal basic solutions for \eqref{eq:standardDP}. However, the convex hull spanned by these solutions defines the dual optimality set and since we are minimizing a (random) linear function we can rewrite the limit law in terms of all dual optimal solutions
\begin{equation*}
\max_{1\leq k \leq K}G^T\lambda(I_k) = \max_{\substack{\lambda \text{ optimal solution} \\ \text{ for }\eqref{eq:standardDP}}}G^T\lambda.
\end{equation*}

\subsection{Properties of the Support}
For the primal linear program \eqref{eq:standardLP} and some primal feasible bases $I\subseteq [d]$ denote by $x=x(I,b)$ the induced primal basic feasible solution. Any such solution can be considered to be relatively sparse as there are at most $m\leq d$ nonnegative entries (usually $m<<d$). Sparsity is even more prominent if $x$ turns out to be degenerate. Further, the coordinates of the corresponding primal basic feasible solution $x$ can be partitioned as follows. Define 
\begin{equation*}
\Pos(x)
\coloneqq \{i\, \mid \,x_i>0\}
\subseteq[d]=\{1,\dots,d\}
\end{equation*}
as the nonzero entries of $x$. The true zeroes are the entries that vanish for \emph{any} dual-feasible $I$ that induces $x$, i.e., 
\begin{equation*}
\TZ(x) \coloneqq [d] \setminus 
\left(\bigcup_{I\, \mid \,A^T\lambda(I)\leq c;\,x=x(I,b)}I\right).
\end{equation*}
The degenerate zeroes are the entries that vanish for \emph{some} (but not all) $I$ that induces $x$, i.e.,
\begin{equation*}
\DZ(x)\coloneqq \left(\bigcup_{I\, \mid \,A^T\lambda(I)\le c;\,x=x(I,b)}I\right)
\setminus \Pos(x).
\end{equation*}
We clearly have that $\DZ(x)$, $\TZ(x)$ and $\Pos(x)$ form a partition of $[d]$. With this notation we can prove that asymptotically an empirical optimal basic solution $x^\star(b_n)$ includes the set $\Pos(x)$ and $\TZ(x)$ with high probability.

\begin{thm}\label{thm:support}
Suppose assumptions \eqref{ass:uniquexb}, \eqref{ass:cltforb}, and \eqref{ass:feasforn} hold. Then we find that
\begin{equation*} 
\lim_{n\to \infty}\PR\left(\bigcap_{i\in \TZ(x)} x_i^{\star}(b_n)=0\right) = 1, \quad \lim_{n\to\infty}\PR\left(\bigcap_{i\in \Pos(x)} x_i^{\star}(b_n)>0\right) = 1. 
\end{equation*}
Assume further that $m_0=m$ and the density of $G$ is positive in a neighborhood of the origin, and that for each $\K$ the limiting distribution of $\alpha^\K$ is not concentrated on a lower-dimensional simplex (this is trivially true if \eqref{ass:dualnondegen} holds). Then for any degenerate zero $i\in \DZ(x)$ it holds that 
\begin{equation*}
\PR\left( [M(G)]_i>0\right)>0.
\end{equation*}
\end{thm}

%% file: Sections/Proofs.tex
\section{Proofs for the Main Results}\label{sec:Proofs}

In this section, we prove our main theorems. The approach is based on a careful decomposition of the ground probability space $\Omega$ into events (subsets of the underlying probability space\footnote{To simplify the notation, we assume that all random vectors in the paper are defined on a common probability space $(\Omega,\mathcal F,\PR)$. This is no loss of generality by Skrokhod representation theorem.}) depending on the random fluctuation $b_n-b$. 

\subsection{Preliminary Steps}

This subsection introduces the aforementioned events. In particular, the main step here is to rewrite them in a convenient way to conclude about their $G_n^{m_0}$ probability content as $n$ tends to infinity. At this stage we suppose that assumptions \eqref{ass:uniquexb}, \eqref{ass:cltforb}, and \eqref{ass:feasforn} hold. 

\subsubsection{Indexing Dual Solutions by Bases}
According to assumption \eqref{ass:cltforb}, the estimator $b_n$ is (weakly) consistent, i.e., for any $\epsilon>0$ the probability $\PR(\Vert b_n-b\Vert>\epsilon)$ converges to zero as $n$ tends to infinity. Consider the empirical counterpart (\hyperref[eq:standardDP]{$\text{D}_{b_n}$}) and its corresponding empirical primal linear program (\hyperref[eq:standardLP]{$\text{P}_{b_n}$}). Notice that the feasible dual bases for (\hyperref[eq:standardDP]{$\text{D}_{b_n}$}) are precisely the same as for \eqref{eq:standardDP} since they do not depend on $b$. Let $I_1,\dots,I_N$ be all dual feasible bases, where without loss of generality the first $K\leq N$ are such that they also induce optimal primal solutions for \eqref{eq:standardLP}, i.e.,
\begin{equation*}
x(I_k,b)\in OPT(b) \Leftrightarrow k\leq K.
\end{equation*}
In general, the primal basic solution $x(I_k,b_n)$ for $k\leq K$ may fail to be feasible for (\hyperref[eq:standardLP]{$\text{P}_{b_n}$}) as we might encounter negative entries; this is true even if $b_n$ is close to $b$. Define for any subset $\K\subseteq[N]$ the event 
\begin{equation*}
A_n^{\K} \coloneqq\left\lbrace\omega\in\Omega \,\vert\, x(I_k,b_n(\omega))\geq 0 \Leftrightarrow k\in \K\right\rbrace \subseteq\Omega
\end{equation*}
that the dual feasible bases indexed by $\K$ are precisely those that induce a primal optimal basic solution for (\hyperref[eq:standardLP]{$\text{P}_{b_n}$}). Since $\lambda(I_k)$ for $k\in \K$ is a dual basic feasible solution for (\hyperref[eq:standardDP]{$\text{D}_{b_n}$}) we deduce by strong duality (see Theorem \ref{thm:strongduality}) that the set $A_n^{\K}$ is the event that the dual feasible bases indexed by $\K$ are precisely those that induce a primal optimal basic solution for (\hyperref[eq:standardLP]{$\text{P}_{b_n}$}). If $\K=\emptyset$ this is the event that the primal is infeasible and consequently the optimality set $P(b_n)$ is empty. For notational simplicity we set $x^{\star}(b_n)=(\infty,\dots,\infty)\in(\R\cup\{\infty\})^d$ when $A_n^\emptyset$ occurs. However, this event has vanishing probability since
\begin{equation*}
\lim_{n\to\infty}\PR\left(A_n^\emptyset\right)=\lim_{n\to\infty}\PR\left(x^\star(b_n)\textit{ does not exist}\right) = 0,
\end{equation*}
by assumption \eqref{ass:feasforn}. The $A_n^\K$ are disjoint by definition and in view of Theorem \ref{lem:rulesOnly}, they form a partition of the underlying probability space
\begin{equation*}
\bigcup_{\K\subseteq[N]}A_n^{\K} =\Omega,
\quad A_n^{\K} \cap A_n^{\K'}=\emptyset,\quad \K\ne \K'.
\end{equation*}
According to the law of total probability
\begin{equation}\label{eq:limitrewritten}
r_n\left(x^{\star}(b_n(\omega)) - x^{\star}(b)\right)
=\sum_{\emptyset\subset\K\subseteq[N]} \mathbbm{1}_{A_n^\K}(\omega) \,r_n\left(x^{\star}(b_n(\omega)) - x^{\star}(b)\right)
+o_\PR (1),
\end{equation}
where $\mathbbm{1}_{A}(\omega)$ denotes the usual indicator function of the set $A$.

\subsubsection{Neglecting the Infeasible Bases}
As a next step, we can simplify the sum of the right-hand side in \eqref{eq:limitrewritten} as some subsets $\K\subseteq [N]$ have asymptotically probability zero. We define the events
\begin{equation*}
B_n^k\coloneqq \left\lbrace\omega\in\Omega\, \vert \,x(I_k,b_n(\omega))\geq 0\right\rbrace \subseteq \Omega
\end{equation*}
that dual feasible basis $I_k$ also induces a primal feasible basic solution for (\hyperref[eq:standardLP]{$\text{P}_{b_n}$}). By strong duality (see Theorem \ref{thm:strongduality}) the set $B_n^k$ entails that $x(I_k,b_n)$ is a primal optimal basic solution for (\hyperref[eq:standardLP]{$\text{P}_{b_n}$}), and moreover
\begin{equation}\label{eq:Bnkdual}
B_n^k
\subseteq
\left\{b_n^T\lambda(I_k)
=\max_{\lambda:A^T\lambda\le c} b_n^T\lambda\right\} \subseteq \Omega.
\end{equation}
The right-hand side is the event that $\lambda(I_k)$ is optimal for the dual problem (\hyperref[eq:standardDP]{$\text{D}_{b_n}$}). This does not necessarily imply that $I_k$ induces a primal optimal basic solution for (\hyperref[eq:standardLP]{$\text{P}_{b_n}$}) as some coordinates of $x(I_k,b_n)$ might be negative. Hence, the inclusion in the above display can be strict. We conclude the following probabilistic statement for $n$ tending to infinity. Recall that $I_1,\dots,I_K$ enumerate all bases that induce primal and dual feasible basic solutions (see Remark \ref{rem:basessplitting}), where $1\le K\le N$.
\begin{lem}\label{lem:feas}
For any index $k>K$ it holds that
\begin{equation*}
\lim_{n\to \infty}\PR \left(B_n^k\right)= 0.
\end{equation*}
\end{lem}
\begin{proof}
Let $k>K$. Then $I_k$ that yields a dual basic feasible solution for \eqref{eq:standardDP}, has the property that for at least one coordinate $i\in I_k$ it holds that $x_i(I_k,b)<0$ (else $x(I_k,b)\geq 0$ and hence $k\leq K$). For this particular index $i$ we find
\begin{equation*}
\begin{split}
\PR\left(B_n^k\right)&=\PR\left( \{\omega\in\Omega\, \vert \,x(I_k,b_n(\omega))\geq 0\}\right)\\
&\leq\PR(x_i(I_k,b_n)\geq 0)=\PR\left(x_i(I_k,b)-x_i(I_k,b_n)\leq x_i(I_k,b)\right).
\end{split}
\end{equation*}
As $b_n$ converges in probability to $b$, the real value $x_i(I_k,b_n)$ converges in probability to $x_i(I_k,b)$ by the continuous mapping theorem. Hence, we conclude that the last event in the above display has probability converging to zero for $n$ to infinity as $x_i(I_k,b)<0$. This yields the claim.
\end{proof}
An immediate consequence of Lemma \ref{lem:feas} is that the sum in \eqref{eq:limitrewritten} can be rewritten as
\begin{equation*}
\begin{split}
&\sum_{\emptyset\subset\K\subseteq[N]} \mathbbm{1}_{A_n^\K}(\omega)\, r_n\left(x^{\star}(b_n(\omega)) - x^{\star}(b)\right) + o_\PR(1) \\
= &\sum_{\emptyset\subset\K\subseteq[K]} \mathbbm{1}_{A_n^\K}(\omega)\, r_n\left(x^{\star}(b_n(\omega)) - x^{\star}(b)\right) + o_\PR(1)\\
=&\sum_{\emptyset\subset\K\subseteq[K]} \mathbbm{1}_{A_n^\K}(\omega)\, r_n\left(\alpha_n^\K(\omega) \otimes x(I_\K,b_n(\omega)) - x^{\star}(b)\right) + o_\PR(1).
\end{split}
\end{equation*}
Moreover, by assumption \eqref{ass:uniquexb} and as $\K\subseteq[K]$ is non-empty, all the basic solutions $x(I_\K,b)$ induce the same primal optimal basic solution $x^{\star}(b)$, and then $\alpha_n^\K\otimes x(I_\K,b)=x^{\star}(b)$. Thus, the last sum equals
\begin{equation}\label{eq:proofsumm}
\sum_{\emptyset\subset\K\subseteq[K]} \mathbbm{1}_{A_n^\K}(\omega)\, \alpha_n^\K(\omega) \otimes x(I_k,G_n(\omega)) + o_\PR(1),
\end{equation}
where we recall that $G_n=r_n(b_n-b)$. 

\subsubsection{The Limiting Convex Cones}\label{subsubsec:convexcones}
We next investigate the indicator functions $\mathbbm{1}_{A_n^\K}(\omega)$ appearing in \eqref{eq:proofsumm}. The idea is to rewrite $A_n^\K$ such that the underlying process $G_n$ appears. Since for any basis $I_k$ it holds that $$x(I_k,b_n)-x(I_k,b)=x(I_k,b_n-b)$$ and $r_n\geq 0$ we rewrite the event $A_n^{\K} \coloneqq\left\lbrace x(I_k,b_n)\geq 0 \Leftrightarrow k\in \K\right\rbrace$ as 
\begin{equation*}
\begin{split}
A_n^\K&=\bigcap_{k\in \K}\bigcap_{i\in I_k}\left\{x_i(I_k,b_n)\geq 0\right\}
\cap\bigcap_{k\notin \K}\bigcup_{i\in I_k}\{x_i(I_k,b_n)< 0\}\\
&=\bigcap_{k\in \K}\bigcap_{i\in I_k}\{x_i(I_k,G_n)\geq -r_nx_i(I_k,b)\}
\cap\bigcap_{k\notin \K}\bigcup_{i\in I_k}\{x_i(I_k,G_n)< -r_nx_i(I_k,b)\}.
\end{split}
\end{equation*}
To streamline the presentation we assume momentarily that $m_0=m$ (see Remark \ref{rem:m0} for the general case) and investigate each of the intersection separately.\\
\textbf{$\mathbf{\bigcap_{i\in I_k}\{x_i(I_k,G_n)\geq -r_nx_i(I_k,b)\}}$ :} Denote by $\Pos\coloneqq \Pos[x^{\star}(b)]$ the positive indices of the primal optimal basic solution $x^\star(b)$ and set $\epsilon\coloneqq\min_{i\in\Pos}x_i^{\star}(b)$ (equal to 1 if $\Pos$ is empty). For a non-empty subset $\K\subseteq[K]$ notice that $x(I_k,b)=x^{\star}(b)$ and write the intersection in $i$ as
\begin{equation*}
\bigcap_{i\in \Pos}\left\{x_i(I_k,G_n)\ge -r_nx_i^{\star}(b)\right\} \cap 
\bigcap_{i\in I_k\setminus\Pos}\left\lbrace x_i(I_k,G_n)\ge -r_nx_i^{\star}(b)\right\rbrace.
\end{equation*}
The first intersection includes the intersection over $i\in\Pos$ of the events that $x_i(I_k,b_n-b)\ge-\epsilon$, which occurs with high probability as $x_i(I_k,b_n)$ is close to $x_i(I_k,b)$ in probability by the continuous mapping theorem. The second intersection is the event $\left\lbrace x_{I_k\setminus \Pos}(I_k,G_n)\geq 0\right\rbrace$ (recall $x_i^\star(b)=0$ for $i\in I_k\setminus\Pos$), which is the nonnegativity of the coordinates of $(A_{I_k})^{-1}G_n$ corresponding to $I_k\setminus\Pos$. More precisely, write
\begin{align*}
I_k=\{i_1^k<\dots<i_m^k\} \subseteq[d],\,
J_k=\{j\in[m]\, \vert\, i_j^k\notin \Pos[x^{\star}(b)]\}\subseteq[m],
\, (|J_k|=m-|\Pos(x^{\star}(b))|=D),
\end{align*}
and then $\{x_{I_k\setminus \Pos}(I_k,G_n)\ge0\}$ is the event $\left\lbrace G_n\in H_k\right\rbrace$ with
\begin{equation}\label{eq:Hk}
H_k
\coloneqq\bigcap_{j\in J_k}\{v\in\R^m\, \vert \,[(A_{I_k})^{-1}v]_j\geq0\}
=\bigcap_{j\in J_k}\{A_{I_k}u\,\vert\,u_j\geq0\}
=\{A_{I_k}u\,\vert\,u_{J_K}\geq0\}
\end{equation}
a closed convex cone in $\mathbb{R}^m$. In total, we obtain that 
\begin{equation}\label{eq:firstset}
\bigcap_{i\in I_k}\{x_i(I_k,G_n)\geq -r_nx_i(I_k,b)\} = \left\lbrace  G_n\in H_k\right\rbrace + o_\PR(1).
\end{equation}

\begin{rem}[Degeneracy in Linear Programs]\label{rem:degeneracy}
Notice that the index set $J_k$ defined above depends on the amount of degeneracy of the primal optimal solution $x^{\star}(b)$. In particular, if $x^{\star}(b)$ is nondegenerate then there exists only one corresponding basis $I$ such that $x^{\star}(b)=x(I,b)$, i.e., we have that $k=K=1$. Moreover, the set $J_1$ is empty and the closed convex cone $H_1^m=\R^m$.
\end{rem}
\textbf{$\mathbf{\bigcup_{i\in I_k}\{x_i(I_k,G_n)< -r_nx_i(I_k,b)\}}$ :} Dealing with the union requires some care as $k$ is possibly larger than $K$. If $k>K$, then by definition $x(I_k,b)$ is not a primal basic feasible solution for \eqref{eq:standardLP}, which means that there exists an $i\in I_k$ such that $x_i(I_k,b)<0$. Then the probability $\PR \left(\{x_i(I_k,b_n-b)<-x_i(I_k,b)\}\right)$ tends to one with increasing $n$ since again by the continuous mapping theorem $x_i(I_k,b_n)$ is close to $x_i(I_k,b)$ in probability. Therefore,
\begin{equation*}
\lim_{n\to\infty}\PR\left(\bigcap_{k>K}\bigcup_{i\in I_k}\{x_i(I_k,G_n)< -r_nx(I_k,b)\}\right)=1
\end{equation*}
and we can focus on $k\le K$. As before we decompose the union in $i\in I_k$ by
\begin{equation*}
\bigcup_{i\in \Pos}\{x_i(I_k,G_n)< -r_nx_i^{\star}(b)\}\cup
\bigcup_{i\in I_k\setminus\Pos}\{x_i(I_k,G_n)< -r_nx_i^{\star}(b)\}.
\end{equation*}
By identical arguments as before the first union has asymptotic probability zero. With the previously introduced notation of the closed convex cone $H_k$, a straightforward calculation shows that the second union is the event $\lbrace G_n\notin H_k\rbrace$. This yields that 
\begin{equation}\label{eq:secondset}
\bigcup_{i\in I_k}\{x_i(I_k,G_n)< -r_nx_i(I_k,b)\} = \left\lbrace G_n\notin H_k \right\rbrace + o_\PR(1).
\end{equation}
\textbf{The Limit for $\mathbf{\mathbbm{1}_{A_n^\K}}$ :} Combining \eqref{eq:firstset} and \eqref{eq:secondset} we conclude that asymptotically
\begin{equation*}
\begin{split}
A_n^\K
&=\bigcap_{k\in\K}\{G_n\in H_k\}\cap \bigcap_{k\notin \K}\{G_n\notin H_k\} + o_\PR(1)\\
&=\bigcap_{k\in\K}\{G_n\in H_k\}\setminus \bigcup_{k\notin \K} \{G_n\in H_k\} + o_\PR(1).
\end{split}
\end{equation*}
Denote for $\emptyset\subset\K\subseteq[K]$ the closed convex cone $H_\K=\cap_{k\in\K}H_k\subseteq\R^{m}$ and set $H_\emptyset=\emptyset$. Then we find that
\begin{equation*}
\PR (A_n^\K)
=\PR\left(G_n\in H_\K\setminus \cup_{k\notin \K}H_k\right) + o(1),
\qquad \emptyset\subset\K\subseteq[K]
,
\end{equation*}
where the union is empty if $\K=[K]$. In order to take the limit as $n$ approaches infinity, we employ the Portmanteau theorem \citep[][Theorem 2.1]{billingsley1999convergence}. For this it suffices for the boundary\footnote{We denote by $\partial A=\overline A\setminus \inte A$ the boundary of a set $A$; $\inte(A)$ and $\overline A$ denote the interior and closure of $A$, respectively.} of these sets to have Lebesgue measure zero as by assumption \eqref{ass:cltforb} the limit $G$ is absolutely continuous with respect to Lebesgue measure. Notice that\footnote{In general and as $\partial(\R^{m}\setminus A)=\partial A$, it holds $\partial(A\cap B_1\dots\cap B_l)\subseteq (\partial A)\cup \bigcup_{i=1}^l \partial (\R^{m}\setminus B_i)$.}
\begin{equation*}
\partial(H_\K\setminus \cup_{k\notin \K}H_k)
\subseteq (\partial H_\K)\cup \bigcup_{k\notin\K}\partial H_k
\end{equation*}
has Lebesgue measure zero, since $H_\K$ and $H_k$ are convex sets. Applying the Portmanteau theorem this allows to conclude for $n$ tending to infinity 
\begin{equation*}
\mathbbm{1}_{A_n^\K}
\xrightarrow[]{D} \mathbbm{1}_{G\in H_\K\setminus \cup_{k\notin \K}H_k}.
\end{equation*}

\begin{rem}[The case $m_0<m$]\label{rem:m0}
The same line of reasoning works if we only consider the first $m_0<m$ coordinates of $b_n$ to be random. Then the last $m-m_0$ coordinates of $G_n$ are zero and one can replace the closed convex cone $H_k$ by an $m_0$-dimensional closed convex cone
\begin{equation*}
H_k^{m_0}
=\bigcap_{j\in J_k}\left\{v_{m_0}\in\R^{m_0} \, \vert\,\left[(A_{I_k})^{-1}\begin{pmatrix}v_{m_0}\\0_{m-m_0}\end{pmatrix}\right]_j\ge0\right\}\, .
\end{equation*}
In particular, it holds as $n$ approaches infinity that
\begin{equation*}
\mathbbm{1}_{A_n^\K}
\xrightarrow[]{D} \mathbbm{1}_{G^{m_0}\in H_\K^{m_0}\setminus \cup_{k\notin \K}H_k^{m_0}}.
\end{equation*}
\end{rem}

\subsection{Proofs for Main Results}

With the preliminary steps from the previous subsection, we are now able to prove our main statements.

\begin{proof}[Theorem \ref{thm:limitGeneral}]\label{subsubsec:proofmain}
For the closed convex cones $H_1^{m_0},\ldots,H_K^{m_0}\subseteq \mathbb{R}^{m_0}$ and a nonempty $\K\subseteq[K]$ define the function $T^\K:\R^{|\K|}\times \R^m\to \R^d$ by
\begin{equation*}
T^\K(\alpha,v)
=\sum_{k\in\K} \mathbbm{1}_{v_{[m_0]}\in H_\K^{m_0}\setminus \cup_{k\notin \K}H^{m_0}_k}\alpha_k x(I_k,v).
\end{equation*}
This function is continuous for all $\alpha\in\R^{\K}$ and all vectors $v\in\R^m$ such that $v_{[m_0]}\notin\partial (H_\K^{m_0}\setminus\cup_{k\notin \K}H^{m_0}_k)$. In particular, the continuity set is of full measure with respect to $(\alpha^{\K},G)$, because based on our previous discussion in Section \ref{subsubsec:convexcones} the probability that $G^{m_0}$ is in the boundary vanishes for all non-empty $\K\subseteq[K]$. As there are finitely many possible subsets $\K$ denoted by $\K_1,\dots,\K_B$, the function $T=\left(T^{\K_1},\dots,T^{\K_B}\right)\colon \R^{\sum_{i=1}^B|\K_i|}\times \R^m\to (\R^d)^B$ defined by
\begin{equation*}
T\left(\alpha^{\K_1},\dots,\alpha^{\K_B},v\right)
=\left(T^{\K_1}(\alpha^{\K_1},v),\dots,T^{\K_B}(\alpha^{\K_B},v)\right)
\end{equation*}
is continuous $G$-almost surely. The continuous mapping theorem together with the assumed joint distributional convergence of the random vector $(\alpha_n^\K,G_n)$ yield that
\begin{equation*}
\begin{split}
&\sum_{\emptyset\subset\K\subseteq[K]}  \mathbbm{1}_{G_n^{m_0}\in H_\K^{m_0}\setminus \cup_{k\notin \K}H_k^{m_0}}\,\alpha_n^\K\otimes x\left(I_\K,G_n\right)
= \sum_{\emptyset\subset\K\subseteq[K]}   T^\K\left(\alpha^{\K}_n,G_n\right)\\
&\xrightarrow[]{D} \sum_{\emptyset\subset\K\subseteq[K]}  T^\K\left(\alpha^{\K},G\right)=
\sum_\K \mathbbm{1}_{G^{m_0}\in H_\K^{m_0}\setminus \cup_{k\notin\K}H_k^{m_0}}\,\alpha^\K \otimes x\left(I_\K,G\right),
\end{split}
\end{equation*}
which finishes the proof for Theorem \ref{thm:limitGeneral}. Notice that according to Remark \ref{rem:degeneracy} if the primal and dual feasible basis $I_1$ is unique the above sum collapses to a single summand and the limit law takes the form $r_n\left(x^{\star}(b_n) - x^{\star}(b)\right)
\xrightarrow[]{D} M(G) \coloneqq x(I_1,G)$.
\end{proof}

\begin{proof}[Theorem \ref{thm:limit}]
Since the result holds without any regularity conditions on $G$ or $G^{m_0}$, we can assume that $m_0=m$.  We proceed by showing that, in presence of \eqref{ass:dualnondegen}, the function
\begin{align*}
T\left(\alpha^{\K_1},\dots,\alpha^{\K_B},v\right)
=\sum_{\emptyset\subset\K\subseteq[K]}  \mathbbm{1}_{v\in H_\K\setminus \cup_{k\notin \K}H_k}\,\alpha^\K\otimes x\left(I_\K,v\right),
\end{align*}
as defined in the previous proof, does not depend on $\alpha=\left(\alpha^{\K_1},\dots,\alpha^{\K_B}\right)$. We rewrite
\begin{align*}
H_k&= \left\{v\in\R^{m} \, \vert\,\left[(A_{I_k})^{-1}v\right]_j\ge0,\, \forall j\in J_k\right\}\\
&= \left\{v\in\R^{m} \, \vert\,\left[(A_{I_k})^{-1}v\right]_j\geq -\frac{1}{\eta}\left[x^\star(b)\right]_j,\, \text{ for some }\eta>0,\, \forall j\in I_k\right\}\\
&=\left\{v\in\R^{m} \, \vert\,(A_{I_k})^{-1}v\geq -\frac{1}{\eta}A_{I_k}^{-1}b,\, \text{ for some }\eta>0\right\}\\
&=\left\{v\in\R^{m} \, \vert\,(A_{I_k})^{-1}\left(b+\eta v\right)\geq 0  \text{ for some }\eta>0\right\}.
\end{align*}
In particular, the cone $H_k$ is the set of all directions $v\in\mathbb{R}^{m}$ such that for some $\eta>0$ the perturbed linear program (\hyperref[eq:standardLP]{$\text{P}_{b+\eta v}$}) has $I_k$ as optimal basis. According to Lemma \ref{lem:dualnondegenerate}, $\lambda(I_k)$ is nondegenerate and optimal for (\hyperref[eq:standardLP]{$\text{P}_{b}$}); it remains so for problem (\hyperref[eq:standardDP]{$\text{D}_{b+\eta v}$}), and therefore the corresponding primal solution for (\hyperref[eq:standardLP]{$\text{P}_{b+\eta v}$}) is unique (see Proposition \ref{prop:slackness}). Hence, for any $v\in H_j^{m}\cap H_k^{m}$ with $1\leq j<k\leq K$ we deduce that $x(I_j,b+\eta v)=x(I_k,b+\eta v)$ and consequently $x(I_j,v)=x(I_k,v)$.  For $v\in\mathcal{U}\coloneqq\cup_{k=1}^K H_k$, define the index set $\K(v)\coloneqq \left\lbrace k\in [K]\mid v\in H_k\right\rbrace \neq \emptyset$ and denote its minimal index by $K(v)\coloneqq \min \K(v)$. The sets defining the indicator functions for $T$ are by construction disjoint and their union is $\mathcal{U}$. We therefore conclude that
\begin{align*}
T\left(\alpha^{\K_1},\dots,\alpha^{\K_B},v\right)= \alpha^{\K(v)}\otimes x(I_{\K(v)},v)=x(I_{K(v)},v).
\end{align*}
Thus, $T$ does not depend on $\alpha$ and simplifies to
\begin{align*}
T(v)=
\sum_{k=1}^K \mathbbm{1}_{K(v)=k}\,x\left(I_k,v\right)=
\sum_{k=1}^K \mathbbm{1}_{v\in H_k\setminus \cup_{j<k}H_j}\,x\left(I_k,v\right).
\end{align*}
To prove continuity of $T$, let us now fix $v_0\in \mathcal{U}$. Since all cones $H_k$ are closed, there exists $\epsilon>0$ such that for $v\in\mathcal{U}$ with $\Vert v-v_0\Vert<\epsilon$ it holds $\K(v)\subseteq\K(v_0)$. This yields $T(v)=x(I_{K(v_0)},v_0)=x(I_{K(v)},v_0)$ and we deduce
\begin{align*}
\Vert T(v)-T(v_0)\Vert=\Vert x(I_{K(v)},v-v_0)\Vert\leq \max_{k\le K} \Vert A_{I_k}^{-1}\Vert_\infty \Vert v-v_0\Vert.
\end{align*}
Hence, $T$ is Lipschitz continuous on $\mathcal U$ with Lipschitz constant $\max_{k\le K} \Vert A_{I_k}^{-1}\Vert_\infty$. The statement now follows by the continuous mapping theorem and the limit reads as
\begin{align*}
T(G)=\sum_{k=1}^K \mathbbm{1}_{G\in H_k\setminus \cup_{j<k}H_j}\,x\left(I_k,G\right).
\end{align*}
If $G$ is absolutely continuous, this can be further simplified to $T(G)=\sum_{k=1}^K \mathbbm{1}_{G\in H_k}\,x\left(I_k,G\right)$ as any intersection $H_j\cap H_k$ has Lebesgue measure zero. To see this notice for $v\in H_j\cap H_k$ and some $\eta>0$, it holds
\begin{align*}
c^Tx(I_j,b+\eta v)= c^Tx(I_k,b+\eta v)\quad
\Leftrightarrow\quad &c^T A_{I_j}^{-1}(b+\eta v)=c^T A_{I_k}^{-1}(b+\eta v)\\
\Leftrightarrow\quad &\left(\lambda(I_k)-\lambda(I_j) \right)^T v =0,
\end{align*} 
where we use the definition for corresponding optimal dual solution $\lambda(I_k)=c^T A_{I_k}^{-1}$. By \eqref{ass:dualnondegen} the vector $\lambda(I_k)-\lambda(I_j)$ is nonzero and hence $v$ is contained in its orthogonal complement, which indeed has Lebesgue measure zero.
\end{proof}

\begin{proof}[Theorem \ref{thm:Hausdorff}]\label{proof:Hausdorff}
Recall the definition of the Hausdorff distance
\begin{equation*}
d_H\left(Opt(b_n),Opt(b)\right)\coloneqq \max \left\lbrace\sup_{x\in OPT(b_n)}\inf_{y\in OPT(b)}\Vert x-y\Vert, \sup_{x\in OPT(b)}\inf_{y\in OPT(b_n)}\Vert x-y\Vert\right\rbrace.
\end{equation*}
We prove that it is bounded in probability $O_\PR (r_n^{-1})$ by considering separately each of the two random quantities over which the maximum is defined. Proving boundedness in probability for the first random quantity
\begin{equation*}
\sup_{x\in OPT(b_n)}\inf_{y\in OPT(b)}\|x-y\|
\end{equation*}
relies on the observation from Lemma \ref{lem:feas} that if $n$ approaches infinity the optimal primal and dual bases for (\hyperref[eq:standardLP]{$\text{P}_{b_n}$}) are contained in the optimal primal and dual bases for \eqref{eq:standardLP} with probability tending to one. According to Lemma \ref{lem:rulesOnly}, we deduce that with high probability 
\begin{equation*}
\emptyset \subset OPT(b_n) \subseteq \Conv\{x(I_{[K]},b_n)\}.
\end{equation*}
Since by assumption $b_n-b\in O_\PR (r_n^{-1})$, we conclude
\begin{equation*}
\sup_{x\in OPT(b_n)}\inf_{y\in OPT(b)}\|x-y\|
\leq
\sup_{\alpha\in \Delta_{K}} \left\|\alpha\otimes x(I_{[K]},b_n-b)\right\|
\leq \max_{k\le K}\|x(I_k,b_n-b)\|
=O_\PR (r_n^{-1}).
\end{equation*}
In fact, if $OPT(b)=\{x^\star(b)\}$ is a singleton the proof for the Hausdorff distance is completed as
\begin{equation*}
d_H\left(Opt(b_n),Opt(b)\right)=\sup_{x\in OPT(b_n)}\Vert x-x^\star(b)\Vert,
\end{equation*}
which is already known to be $O_\PR (r_n^{-1})$. In particular, by \eqref{ass:optBounded} this holds for $b=0$ and hence without loss of generality we may assume from now on that $b\neq 0$. However, if $OPT(b)$ is not a singleton then proving the boundedness for the second random quantity 
\begin{equation*}
\sup_{x\in OPT(b)}\inf_{y\in OPT(b_n)}\Vert x-y\Vert
\end{equation*}
requires some more care. More precisely, in presence of degeneracy for any primal optimal basic solutions for \eqref{eq:standardLP} the primal and dual optimal bases for (\hyperref[eq:standardLP]{$\text{P}_{b_n}$}) can be a strict subset of primal and dual optimal bases for \eqref{eq:standardLP}. Hence, our previous argument does not apply. We define an equivalence relation on the set of all primal and dual optimal bases $I_1,\ldots,I_K$ for \eqref{eq:standardLP} as follows:
\begin{equation*}
j\sim k :\Longleftrightarrow x(I_j,b)=x(I_k,b), \, \, j,k\in [K].
\end{equation*}
We denote by $\mathcal{B}_1,\ldots,\mathcal{B}_T$ the induced equivalence classes.

\begin{lem}\label{lem:basisequivalenceclasses}
There exists $\epsilon>0$ such that for all $\tilde{b}$ with $(\text{P}_{\tilde{b}})$ feasible and $\Vert \tilde{b}-b\Vert <\epsilon$ it holds that for any $1\leq t\leq T$ there exists $k\in \mathcal{B}_t$ such that $x(I_k,\tilde{b})\in OPT(\tilde{b})$.
\end{lem}

\begin{proof}
For any $1\leq t\leq T$ denote by $Pos_t\subseteq[d]$ the positivity set of the primal optimal basic solution defined by all bases in equivalence class $\mathcal{B}_t$. Since $b\ne0$, the set $Pos_t$ is nonempty. Consider the pair of primal and dual linear programs
\begin{center}
\begin{minipage}{0.45\textwidth}
\begin{equation}\label{eq:dualvalue1}
\begin{aligned}
\max_{\lambda\in\mathbb{R}^m} \quad
\lambda^Tb & \\
\text{s.t.} \quad
[A^T\lambda]_{[d]\setminus Pos_t} \leq& c_{[d]\setminus Pos_t},\\
[A^T\lambda]_{Pos_t} = & c_{Pos_t},\\
\end{aligned}\tag{$\text{D}_{b,Pos_t}$}
\end{equation}
\end{minipage}
\begin{minipage}{0.45\textwidth}
\begin{equation}\label{eq:primalvalue1}
\begin{aligned}
\min\limits_{x\in \mathbb{R}^d}{} \quad
c^T x& \\
\text{s.t.} \quad
Ax = b,  \\
x_{[d]\setminus Pos_t}\geq 0.
\end{aligned}\tag{$\text{P}_{b,Pos_t}$} 
\end{equation}
\end{minipage}
\end{center}
By construction the basic solutions for \eqref{eq:dualvalue1} are $\lambda(I_k)$ for each $k\in \mathcal{B}_t$. By continuity of $x(I_k,b)$ in $b$ choose an $\epsilon>0$ such that $\Vert \tilde{b}-b\Vert<\epsilon$ and $Pos_t\subset Pos(x(I_k,\tilde{b}))$. By assumption $(\text{P}_{\tilde{b}})$ is feasible and its optimality set is bounded. Hence, $(\text{D}_{\tilde{b},Pos_t})$ is bounded and feasible and strong duality holds. This yields a representative $k\in \mathcal{B}_t$ such that $x(I_k,\tilde{b})\geq 0$ is optimal for $(\text{P}_{\tilde{b}})$, an hence contained in $OPT(\tilde{b})$. As $1\leq t\leq T$ is chosen arbitrarily the proof is finished.
\end{proof}

We are now able to prove that even if the primal \eqref{eq:standardLP} does not admit a unique optimal solution, we still have
\begin{equation*}
\sup_{x\in OPT(b)}\inf_{y\in OPT(b_n)}\Vert x-y\Vert \in O_\PR (r_n^{-1}).
\end{equation*}
Specifically, we apply Lemma \ref{lem:basisequivalenceclasses} to all $1\leq t\leq T$. Then as $n$ tends to infinity we find with probability $1-o(1)$ that $\Vert b_n-b\Vert<\epsilon$ and together with assumption \eqref{ass:feasforn} and Lemma \ref{lem:basisequivalenceclasses} there exists $k_t(n)\in \mathcal{B}_t$ such that $x(I_{k_t(n)},b_n)\in OPT(b_n)$. Observe that if we pick for any equivalence class $\mathcal{B}_t$ an arbitrary representative $k_t$ to obtain a sequence $\{k_t\}_{t\in T}$ of full representation then we can rewrite $OPT(b)=Conv(\{x(I_{k_t},b)\, \mid \, 1\le t\le T\})$. In particular, for $n$ large enough $\{k_t(n)\}_{t\in T}$ is a random full representation, i.e., $OPT(b)=Conv(\{x(I_{k_t(n)},b)\, \mid \, 1\le t\le T\})$. We deduce that with high probability
\begin{equation*}
\begin{split}
\sup_{y\in OPT(b)}\inf_{x\in OPT(b_n)}\Vert x-y\Vert &= \sup_{\alpha\in\Delta_T}\inf_{x\in OPT(b_n)}\left\Vert x - \sum_t \alpha_tx(I_{k_t}(n),b)\right\Vert\\
&\leq \sup_{\alpha\in\Delta_T}\left\Vert\sum_{t=1}^T \alpha_tx(I_{k_t}(n),b_n-b)\right\Vert\\
&\leq \max_{1\leq t\leq T} \left\Vert x(I_{k_t}(n),b_n-b)\right\Vert =O_p(r_n^{-1}).
\end{split}
\end{equation*}
\end{proof}

\begin{proof}[Proposition \ref{prop:cltoptimalvalue}]
In fact, the statement can be deduced from our main Theorem \ref{thm:limitGeneral}. However, the observation in Lemma \ref{lem:feas} paves the way for a more direct proof.

Denote again by $I_1,\ldots,I_N$ all dual optimal basic solutions for \eqref{eq:standardDP} among those $I_1,\ldots,I_K$ for $K\leq N$ induce primal and dual optimal solutions for \eqref{eq:standardLP} and \eqref{eq:standardDP}, respectively. According to Lemma \ref{lem:feas}, as $n$ tends to infinity the probability that at least one of the bases $I_1,\ldots,I_K$ induce also primal and dual optimal solutions for (\hyperref[eq:standardLP]{$\text{P}_{b_n}$}) and (\hyperref[eq:standardDP]{$\text{D}_{b_n}$}), respectively, tends to one. More precisely, and with the notation for the sets $B_n^k$ as defined on page \pageref{eq:Bnkdual} this means for the event 
\begin{equation*}
F_n\coloneqq \bigcup_{1\leq k \leq K} B_n^k\subseteq \Omega
\end{equation*}
that $\lim_{n\to\infty}\PR(F_n)=1$. We now decompose the random quantity $r_n(c(b_n)-c(b))$ conditioned on the event $F_n$ and apply strong duality. Notice that if event $F_n$ occurs then at least one of $I_1,\ldots,I_K$ induces a dual optimal basic solution for (\hyperref[eq:standardDP]{$\text{D}_{b_n}$}). This choice of basis depends on $n$, however, the optimal value for (\hyperref[eq:standardDP]{$\text{D}_{b_n}$}) is equal to $\max_{1\leq k\leq K}b_n^T\lambda(I_k)$ since all $\lambda(I_k)$ remain basic feasible solutions for (\hyperref[eq:standardDP]{$\text{D}_{b_n}$}). In particular, we find
\begin{equation*}
\begin{split}
r_n(c(b_n)-c(b))&=\mathbbm{1}_{F_n} \,r_n(c(b_n)-c(b)) + \mathbbm{1}_{F_n^c}\,r_n(c(b_n)-c(b))\\
&=\mathbbm{1}_{F_n} \,r_n\left(\max_{1\leq k\leq K}b_n^T\lambda(I_k)-\max_{1\leq k\leq K}b^T\lambda(I_k)\right) + o_\PR(1).
\end{split}
\end{equation*}
Finally, the value $\max_{1\leq k\leq K}b^T\lambda(I_k)$ does not depend on $k$ since all $\lambda(I_k)$ are by definition dual optimal basic solutions for \eqref{eq:standardDP}. Therefore, by continuity, the central limit law for $r_n(b_n-b)\xrightarrow[]{D} G$ and an application of Slutzky we conclude
\begin{equation*}
\begin{split}
&\mathbbm{1}_{F_n} \,r_n\left(\max_{1\leq k\leq K}b_n^T\lambda(I_k)-\max_{1\leq k\leq K}b^T\lambda(I_k)\right) + o_\PR(1) \\
= &\mathbbm{1}_{F_n} \, \left(\max_{1\leq k\leq K}r_n\left(b_n^T-b^T\right)\lambda(I_k)\right) + o_\PR(1) \xrightarrow[]{D}\max_{1\leq k\leq K}G^T\lambda(I_k).
\end{split}
\end{equation*}
\end{proof}

\begin{proof}[Theorem \ref{thm:support}]
Define the event
\begin{equation*}
GoodOpt
=\{P(b_n)\neq\emptyset\}\setminus \cup_{k>K}B_n^k
\end{equation*}
that the primal (\hyperref[eq:standardLP]{$\text{P}_{b_n}$}) is feasible and the only bases that induce a primal optimal basic solution $x^{\star}(b_n)$ are $I_1,\ldots,I_K$ that also induce an optimal $x^{\star}(b)\in OPT(b)$ by definition. An application of Lemma \ref{lem:feas} to $I_k$ for $k>K$ shows that
\begin{equation}\label{eq:PgoodOptexp}
\begin{split}
\PR\left(goodOpt\right)=1 - \PR \left(goodOpt^c\right)= 1- \PR\left(\cup_{k>K}B_n^k \cup \{P(b_n)=\emptyset\}\right)\\
\geq 1- \left(\sum_{l=K+1}^N\PR\left(B_n^l\right) + \PR \left(\left\{P(b_n)=\emptyset\right\}\right)\right)
\to 1,
\end{split}
\end{equation}
where we invoke the (asymptotic) existence assumption \eqref{ass:feasforn}. Recall that $\Pos[x]$ is the set of entries where $x$ is not zero.
\begin{lem}\label{lem:support}
Suppose assumptions \eqref{ass:uniquexb} and \eqref{ass:feasforn} hold. Then any (measurable) choice of $x^{\star}(b_n)$ satisfies
\begin{equation*}
\lim_{n\to\infty}\PR\left(\Pos[x^{\star}(b_n)] \supseteq \Pos[x^{\star}(b)]\right)=1.
\end{equation*}
\end{lem}
\begin{proof}
According to Lemma \ref{lem:rulesOnly}, whenever $goodOpt$ occurs, any choice of $x^{\star}(b_n)$ must belong to the convex hull of $\{x(I_k,b_n) \, \vert \, k\le K\}$. Therefore, the probability of the event in the statement of the lemma is not smaller than
\begin{equation*}
\PR \left(goodOpt\right)
-\sum_{i\in \Pos[x^{\star}(b)]}\sum_{k=1}^K \PR \left(x_i(I_k,b_n)=0\right).
\end{equation*}
Because of \eqref{eq:PgoodOptexp} we only need to bound the probabilities in the last sum.  But they clearly converge to zero since by the continuous mapping theorem $x_i(I_k,b_n)\to x_i(I_k,b)=x_i^{\star}(b)>0$ in probability.
\end{proof}
Continuing the proof, the first assertion follows from \eqref{eq:PgoodOptexp} and Lemma \ref{lem:support}. Now let $i$ be an index corresponding to a degenerate zero. Then there exists $k_0\le K$ such that $i\in I_{k_0}$ and $x_i(I_{k_0},b)=0$. The closed convex cone $H_{k_0}$ defined in \eqref{eq:Hk} contains an open neighborhood around $A_{I_{k_0}}\mathbf 1_m$. Hence $\PR (G\in H_{k_0})>0$ and consequently there exists $\K\supseteq\{k_0\}$ such that $2\delta=\PR (G\in H_\K\setminus H_{[K]\setminus \K})>0$ and thus $\PR (A_n^\K)>\delta>0$ for $n$ large. Since $\alpha_n^\K$ converges in distribution to $\alpha^\K$, which has positive probability of being positive on the coordinate corresponding to $k_0$, for $n$ large there is positive probability that
\begin{equation*}
x^{\star}(b_n)
=\sum_{k\in\K} (\alpha_n^\K)_k x(I_k,b_n)
\end{equation*}
with $(\alpha_n^\K)_{k_0}>0$. Moreover the probability that $x_i(I_{k_0},b_n)=0$ vanishes as $n$ tends to infinity, because $r_nx_i(I_{k_0},b_n)=x_i(I_{k_0},G_n)\to x_i(I_{k_0},G)$, which is a continuous random variable because $G$ is absolutely continuous. All this combined yields that $\PR (x_i^{\star}(b_n)>0)$ is positive for all $n$ large enough.
\end{proof}

%% file: Sections/Assumptions.tex
\section{On the Assumptions}\label{sec:Assumptions}
This section elaborates on the assumptions stated in Section \ref{sec:Preliminaries}. The deterministic assumptions \eqref{ass:optBounded}, \eqref{ass:uniquexb} and \eqref{ass:dualnondegen} are related by
\begin{equation*}
\{\eqref{ass:optBounded} \ \&\  \eqref{ass:dualnondegen}\}
\implies \eqref{ass:uniquexb} \implies \eqref{ass:optBounded},
\qquad
\eqref{ass:dualnondegen}   \mathrel{{\ooalign{\hidewidth$\not\phantom{=}$\hidewidth\cr$\implies$}}}
 \eqref{ass:optBounded},
\end{equation*}
where the first implication is shown in Lemma \ref{lem:dualnondegenerate} and the second is obvious. Regarding the stochastic assumptions \eqref{ass:cltforb}--\eqref{ass:feasforn}, we have (see Lemma \ref{lem:feasibility})
\begin{equation*}
\{Slater(b) \ \&\  \eqref{ass:cltforb}\} \implies \eqref{ass:feasforn},
\end{equation*}
where $Slater(b)$ is the condition that the feasible set $P(b)=\{x\in\mathbb R^d:Ax=b,x\ge0\}$ contains a positive element $x\in(0,\infty)^d$ (Slater's constraint qualification \citep{boyd2004convex}). The assumption \eqref{ass:cltforb} has to be checked for each particular case and can usually be verified by an application of the central limit theorem.
Further statements regarding the assumptions are stated below, and proven in Appendix \ref{Appendix:Assumptions}.\\
\textbf{Feasibility assumption.} The assumption \eqref{ass:feasforn} is obviously necessary for the limiting distribution to exist. It is verified if the convex polytope $P(b)=\{x\in\mathbb{R}^m\,\vert\,Ax=b,x\geq0\}$ is feasible and bounded whenever $\|b-b_0\|<\epsilon$ for some $\epsilon>0$.

\begin{lem}[Boundedness]\label{lem:boundedness1}
Suppose there exists some $b_0\in\mathbb{R}^m$ such that the polytope $P(b_0)=\{x\in\mathbb{R}^m\,\vert\,Ax=b_0,x\geq 0\}$ is feasible and bounded. Then for all $b\in\mathbb{R}^m$ the polytope $P(b)$ is bounded.
\end{lem}
A sufficient condition for boundedness is positivity of the constraint matrix.

\begin{lem}[Boundedness]\label{lem:boundedness2}
Suppose that $A$ has nonnegative entries and no column of $A$ is $0\in\R^m$. Then the polytope $P(b)$ is bounded (possibly nonfeasible) for all $b\in\R^m$.
\end{lem}
Feasibility for $P(b)$ is a more delicate question and usually fails to hold uniformly in $b$. However, we have the following statements that suffice for our purposes.

\begin{lem}[Feasibility]\label{lem:feasibility}
Consider the polytope $P(b_0)=\{x\in\mathbb{R}^m\,\vert\,Ax=b_0,x\geq0\}$. Then $P(b)$ is feasible for all $b$ sufficiently close to $b_0$ if either of the following two conditions holds:
\begin{itemize}
\item [(i)] The polytope $P(b_0)$ admits a nondegenerate basic solution.
\item [(ii)] Slater's constraint qualification holds.
\end{itemize}
\end{lem}
Let us finally give a sufficient condition for assumption \eqref{ass:feasforn}. We obviously have that 
\begin{equation*}
\PR\left(x^{\star}(b_n) \textrm{ exists}\right)\geq \PR\left( P(b_n) \text{ feasible \& bounded}\right).
\end{equation*}
If assumption \eqref{ass:cltforb} holds, we have that $b_n$ converges in probability to $b$. Consequently, the latter event in the above display has asymptotically probability one if $P(b)$ is feasible, bounded and fulfills Slater's condition. These conditions are satisfied for many linear programs, such as the optimal transport problem considered in more detail in Section \ref{sec:OT}.\\
\textbf{Uniqueness assumption.} The uniqueness assumption \eqref{ass:uniquexb} is obviously necessary for our distributional limit stated in Theorem \ref{thm:limitGeneral}. Further, recall Theorem \ref{prop:slackness} for primal uniqueness statements in terms of dual nondegeneracy. From a probabilistic point of view uniqueness is satisfied for almost every cost.

\begin{lem}\label{lem:uniqueness}
Let $P=\{Ax=b,x\ge0\}$ and let $C$ be a $d$-dimensional random vector with distribution absolutely continuous with respect to Lebesgue measure on $\mathbb{R}^d$. Then almost surely the standard linear program $\min_{x\in P} C^Tx$ either admits a unique optimal solution or no solution at all.
\end{lem}
In fact, the proof for Lemma \ref{lem:uniqueness} shows that all dual basic feasible solutions are nondegenerate almost surely. In particular, assumption \eqref{ass:dualnondegen} holds almost surely. We further discuss uniqueness for optimal transport problems in Section \ref{sec:OT}.\\
\textbf{Duality assumption.} Lemma \ref{lem:dualnondegenerate} proves that in presence of \eqref{ass:optBounded} the assumption \eqref{ass:dualnondegen} is equivalent to nondegeneracy of all optimal dual solutions. Thus when these assumptions hold, we have
\begin{equation*}
\lim_{n\to\infty}\PR(\{x^{\star}(b_n) \textit{ is unique}\}) = 1.
\end{equation*}
Under some conditions on the limiting distribution $G$ we can also establish a converse result.
\begin{lem}\label{lem:primalNotUnique}
Suppose that the limit distribution $G$ has a positive density in a neighborhood of the origin\footnote{More generally, it suffices to assume that the origin is a Lebesgue point of the support of $G$.}. If \eqref{ass:dualnondegen} does not hold, then
\begin{equation*}
\liminf_{n\to \infty} \PR\left(x^{\star}(b_n) \text{ is non-unique}\right)>0.
\end{equation*}
\end{lem}
\textbf{Joint convergence.} Our goal here is to state useful conditions such that the random vector $\left(\alpha_n^\K, G_n\right)$\footnote{Recall that the $\alpha_n^\K$ represent random weights (summing up to one) for each optimal basis $I_k$, $k\in\K$ for the case that $A_n^\K$ occurs, i.e., that several bases yield primal optimal solutions and hence any convex combination is also optimal.} jointly converges in distribution to some limit random variable $\left(\alpha^\K,G\right)$ on the space $ \Delta_{\vert \K\vert}\times \mathbb{R}^m$. By assumption \eqref{ass:cltforb}, $G_n\to G$ in disribution, and a necessary condition for the joint distributional convergence of $(\alpha_n^\K,G_n)$ is that $\alpha_n^\K$ has a distributional limit $\alpha^\K$.  There is no reason to expect $\alpha_n^\K$ and $G_n$ to be independent, as discussed at the end of this section.  We give a weaker condition than independence that is formulated in terms of the conditional distribution of $\alpha_n^\K$ given $G_n$ (or, equivalently, given $b_n=b+G_n/r_n$).  These conditions are natural in the sense that if $b_n=g$, then the choice of solution $x^\star(g)$, as encapsulated by the $\alpha_n^\K$'s, is determined by the specific linear program solver in use.\\
Treating conditional distributions rigorously requires some care and machinery.  Let $\Z=\Z^\K=\Delta_{|\K|}\times \R^m$ and for $\varphi:\Z\to \R$ denote
\[
\|\varphi\|_\infty = \sup_{z}|\varphi(z)|,
\qquad \|\varphi\|_{\Lip}=\sup_{z_1\ne z_2}\frac{|\varphi(z_1)-\varphi(z_2)}{\|z_1-z_2\|},
\qquad \|\varphi\|_{\BL}=\|\varphi\|_\infty + \|\varphi\|_{\Lip}.
\]
We say that $\varphi$ is bounded Lipschitz if it belongs to $\BL(\Z)=\{\varphi:\Z\to\R\,\vert \,\|\varphi\|_{\BL}\le1\}$. The bounded Lipschitz metric
\begin{equation}\label{eq:BLnorm}
\BL(\mu_1,\mu_2)
\coloneqq\sup_{\varphi\in\BL(\Z)}\left|\int_\Z \varphi(z)\, d\left(\mu_1 - \mu_2\right)(z)\right|
\end{equation}
is well-known to metrize convergence in distribution of (probability) measures on $\Z$ \cite[Theorem 11.3.3]{dudley2002real}. According to the disintegration theorem (see \citet[Theorem 5.4]{kallenberg1997foundations}, \citet[Section 10.2]{dudley2002real} or \cite{chang1997conditioning} for details), we may write the joint distribution of $\left(\alpha_n^\K,b_n\right)$ as an integral of conditional distributions $\mu_{n,g}^\K$ that represent the distribution of $\alpha_n^\K$ given that $b_n=g$.  More precisely, $g\mapsto \mu_{n,g}^\K$ is measurable from $\R^m$ to the metric space of probability measures on $\Delta_{|\K|}$ with the bounded Lipschitz metric, so that for any $\varphi\in \BL(\Z)$ it holds that
\begin{equation*}
\E\varphi(\alpha_n^\K,b_n)
=\E\psi_n(b_n),
\qquad \psi_n(g)=\int_{\Delta_{|\K|}}\varphi(\alpha,g)d\mu^\K_{n,g}(\alpha),
\end{equation*}
where $\psi_n:\R^m\to\R$ is a measurable function. The joint distribution of $(\alpha_n^\K,G_n)$ is determined by the collection of expectations
\begin{equation*}
\E \varphi(\alpha_n^\K,G_n)
=\E\psi_n(G_n)
=\E\psi_n\left(r_n(b_n-b)\right)
,\qquad \varphi\in \BL(\Z).
\end{equation*}
Our sufficient condition for joint convergence is given by the following lemma.  It is noteworthy that the spaces $\R^m$ and $\Delta_{|\K|}$ can be replaced with arbitrary Polish spaces, and even more general spaces, as long as the disintegration theorem is valid.
\begin{lem}\label{lem:jointConv}
Let $\{\mu^\K_g\}_{g\in\R^m}$ be a collection of probability measures on $\Delta_{|\K|}$ such that the map $g\mapsto \mu^\K_g$ is continuous at $G$-almost any $g\in\R^m$, and suppose that $\mu^\K_{n,g}\to\mu^\K_g$ uniformly with respect to the bounded Lipschitz metric $\BL$. Then $(\alpha_n^\K,G_n)$ converges in distribution to a random vector $(\alpha^\K,G)$ satisfying
\begin{equation*}
\E \varphi(\alpha^\K,G)
=\E_{G}\int_{\Delta_{|\K|}}\varphi(\alpha,G)d\mu_{G}^\K(\alpha)
:=\E \psi(G)
\end{equation*}
for any continuous bounded function $\varphi\in\BL(\Z)$ (this determines the distribution for the random vector $(\alpha^\K,G)$ completely). Moreover, if $\mathcal L$ denotes the distribution of a random vector, then the rate of convergence can be quantified as
\[
\BL(\mathcal L[(\alpha_n^\K,G_n)],\mathcal L[(\alpha^\K,G)])
\le \sup_g \BL(\mu^\K_{n,g},\mu^\K_g)+(1+L)\BL(\mathcal L[G_n],\mathcal L[G]),
\]
where $L\coloneqq\sup_{g_1\ne g_2}\BL(\mu^\K_{g_1},\mu^\K_{g_2})/\|g_1-g_2\|\in[0,\infty]$. The supremum with respect to $g$ can be replaced by an essential supremum.
\end{lem}

The conditions of Lemma~\ref{lem:jointConv} (and hence the joint convergence in Theorem~\ref{thm:limitGeneral}) will be satisfied in many practical situations.  For example, given $b_n$ and an initial basis for the simplex method, its output is determined by the pivoting rule (for a general overview see \cite{terlaky1993pivot} and references therein). Deterministic pivoting rules lead to degenerate conditional distributions of $\alpha_n^\K$ given $b_n=g$, whereas random pivoting rules may lead to nondegenerate conditional distributions. In both cases these conditional distributions do not depend on $n$ at all, but only on the input vector $g$. In particular, the uniform convergence in Lemma \ref{lem:jointConv} is trivially fulfilled (the supremum is equal to zero). It is reasonable to assume that these conditional distributions depend continuously on $g$ except for some boundary values that are contained in a lower-dimensional space (which will have measure zero under the absolutely continuous random vector $G$).

%% file: Sections/OptimalTransport.tex
\section{Optimal Transport}\label{sec:OT}
In this section, we focus on the optimal transport problem (see \cite{villani2008optimal,peyre2019computational,panaretos2020invitation} for further details) on finite spaces, and consider a space $\mathcal{X}=\{x_1,\ldots,x_N\}$ equipped with some underlying cost $c\colon \mathcal{X}\times\mathcal{X}\to \mathbb{R}^{N^2}$ usually represented as a matrix $c\in\mathbb{R}^{N\times N}$ with entries $c_{ij}=c(x_i,x_j)$. Further, denote by 
$\Delta_N\coloneqq\{ r\in \mathbb{R}^N\, \vert \, \mathbf{1}_N^Tr=1,\, r_i\geq 0 \}$ the set of all probability measures on the space $\mathcal{X}$ and let $\text{ri}(\Delta_N)\coloneqq \{ r\in \mathbb{R}^N\, \vert \, \mathbf{1}_N^Tr=1,\, r_i> 0 \}$ be its relative interior. More precisely, we identify the measures with their densities with respect to the counting measure on $\mathcal X$. Two probability measures $r,s\in\Delta_N$ define the set of all couplings between them
\begin{equation*}
\Pi(r,s)\coloneqq \left\lbrace \pi \in \mathbb{R}^{N\times N}\, \mid \, \sum_{j=1}^N \pi_{ij}=r_i,\, \sum_{i=1}^N \pi_{ij}=s_j,\, \pi_{ij}\geq 0\right\rbrace,
\end{equation*}
i.e., probability measures on the product space $\mathcal{X}\times\mathcal{X}$ whose row-marginal coincide with $r$ and column-marginal with $s$. The optimal transport problem 
\begin{equation}\tag{$\text{OT}$}\label{eq:OT}
\min_{\pi\in\Pi(r,s)} \sum_{i,j=1}^N c_{ij}\pi_{ij}
\end{equation}
seeks to find an \textit{optimal transport coupling} $\pi^\star$ between $r$ and $s$ such that the integrated cost is minimal among all possible couplings. Its dual problem reads
\begin{equation}\tag{$\text{DOT}$}\label{eq:DOT}
\begin{aligned}
\max_{\alpha,\beta \in \mathbb{R}^{N}} \quad
r^T\alpha+s^T\beta \quad
\text{s.t.} \quad
\alpha_i+\beta_j \leq\, c_{ij}, \, \forall\, i,j\in[N].
\end{aligned}
\end{equation}


\subsection{Assumptions in View of Optimal Transport}\label{subsec:assumptionsOT}

\textbf{Deterministic assumptions.} We start discussing the deterministic assumptions \eqref{ass:optBounded}, \eqref{ass:uniquexb} and \eqref{ass:dualnondegen} in the context of optimal transport. For \eqref{ass:optBounded} and \eqref{ass:uniquexb} recall that the set of couplings $\Pi(r,s)$ is never empty since it always contains the independence coupling $rs^T$. By Lemma \ref{lem:boundedness1}, the feasibility set $\Pi(r,s)$ is a compact subset of $\mathbb{R}^{N\times N}$. We conclude that optimal transport always attains at least one optimal solution
\begin{equation*}
\pi^\star \in \argmin_{\pi \in \Pi(r,s)} \, \sum_{i,j=1}^N c_{ij}\pi_{ij}
\end{equation*}
all of which are usually termed optimal transport couplings. However, uniqueness of an optimal transport coupling is in general not guaranteed, and we discuss this in more detail below. In total, assumptions \eqref{ass:optBounded} and \eqref{ass:uniquexb} hold if and only if 
\begin{equation}\tag{\textbf{AOT}}\label{ass:uniqueOT}
\textit{the optimal transport coupling }\pi^\star \textit{ is unique.}
\end{equation}
Under \eqref{ass:uniqueOT} the optimal transport coupling is of course a primal optimal basic solution, i.e., induced by an optimal basis.
\bigskip

Sufficient conditions for uniqueness and nondegeneracy in optimal transport problems can be considered to be of interest in their own. The proofs of the statements contained in the following are deferred to Appendix \ref{Appendix:OT}. In fact, various sufficient conditions ensuring uniqueness of an optimal transport coupling are known. Among the most prominent is the \emph{strict Monge} condition that the cost $c$ satisfies
\begin{equation}\label{eq:strictmonge}
c_{ij}+c_{i^{'}j^{'}}<c_{ij^{'}}+c_{i^{'}j}\,, \quad \forall \, i<i^{'}, j<j^{'},
\end{equation}
also taking into account possible relabelling of the indices \cite[Theorem 7]{dubuc1999note}. This translates to easily interpretable statements on the real line.

\begin{lem}\label{cor:uniquerealline}
Let $\mathcal{X}\coloneqq\{x_1<\ldots< x_N\}$ be a set of $N$ distinct ordered points on the real line. Suppose that the cost takes the form $c(x,y)=f(\vert x-y\vert)$ with $f:\R_+\to\R_+$ such that $f(0)=0$ and that $f$ fulfils either one of the following two conditions:
\begin{itemize}
\item[(i)] $f$ is strictly convex,
\item[(ii)] $f$ is strictly concave.
\end{itemize}
Then for any marginals $r,s\in\Delta_N$ the optimal transport problem \eqref{eq:OT} attains a unique solution.
\end{lem}
The first statement follows by employing the Monge condition (see also \citet[Proposition A2]{mccann1999exact} for an alternative approach). The second case is more delicate, and indeed, the description of the unique optimal solution is more complicated. 
Moreover, for both cases the unique transport coupling can be computed by the simple Northwest corner algorithm \citep{hoffman1963simple}. Typical costs covered by Corollary \ref{cor:uniquerealline} are $d^p(x,y)=\vert x-y \vert^p$ for any $p\geq 0$ such that $p\notin \{0,1\}$. 

\begin{rem}[The real line for $p\in \{0,1\}$]\label{rem:realline}
In general, uniqueness statements on the real line for cost $c(x,y)=\vert x-y \vert^p$ with $p\in \{0,1\}$ do not hold. Consider the case that the probability measure $r$ is supported on $\{x_1<\ldots<x_k\}$ while $s$ has support $\{x_{k+1}<\ldots<x_N\}$ for some $1<k< N$. Then every coupling is optimal since for any $\pi\in\Pi(r,s)$ we find that 
\begin{equation*}
\sum_{i,j=1}^N \vert x_i-x_j\vert^p\pi_{ij}=
\begin{cases}
\sum_{i=1}^k r_ix_i - \sum_{j=k+1}^N s_jx_j,\quad &p=1,\\
1,\quad &p=0.
\end{cases}
\end{equation*}
This example may seem extreme, but it is fair to say that for $p\in\{0,1\}$ uniqueness is the exception rather than the rule.
\end{rem}
For a coupling $\pi\in\mathbb{R}^{N\times N}$ we define its support by the tuple of indices 
\begin{equation}
\text{supp}(\pi)\coloneqq \left\lbrace (i,j)\in [N]^2\, \vert \, \pi_{ij}>0\right\rbrace.
\end{equation}

\begin{defn}
Let $c\in\mathbb{R}^{N\times N}$ be an arbitrary cost matrix. An index set $\Gamma \in [N]\times [N]$ is said to be $c$-cyclically monotone if for any $n\in\mathbb{N}$, and any family $(i_1,j_1),\ldots,(i_n,j_n)$ with $(i_k,j_k)\in \Gamma$ it holds that
\begin{equation*}
\sum_{k=1}^n c_{i_k j_k}\leq \sum_{k=1}^n c_{i_k j_{k-1}}
\end{equation*}
with the convention that $j_0\coloneqq j_n$. Further, the set $\Gamma$ is said to be strictly $c$-cyclically monotone if additionally strict inequality holds for all cases such that the right-hand side contains at least one tuple $(i_k,j_{k-1})\notin \Gamma$ for $1\leq k \leq n$.
\end{defn}
Transport couplings are always supported on a $c$-cyclically monotone set and in fact this yields a characterization of optimality \cite[Theorem 5.10]{villani2008optimal}. We prove that uniqueness is equivalent for the support to be strictly $c$-cyclically monotone. 

\begin{thm}\label{thm:uniquenessOT}
For two marginals $r,s\in\Delta_N$ consider the optimal transport \eqref{eq:OT} with cost $c$ and suppose that $\pi^\star$ is an optimal transport coupling. Then $\pi^\star$ is unique if and only if $\text{supp}(\pi^\star)$ is strictly $c$-cyclically monotone.
\end{thm}

\begin{rem}
If the cost matrix $c$ fulfils the strict Monge condition \eqref{eq:strictmonge} then the optimal transport coupling obtained by applying the Northwest corner algorithm has support that is strictly $c$-cyclically monotone and hence it is unique. However, not every optimal transport coupling whose support is strictly $c$-cyclically monotone can be obtained (after a suitable relabelling of columns and rows) by the Northwest corner rule \cite[Example 1.2]{klinz2011northwest}. In other words, not every optimal transport coupling with strictly $c$-cyclically monotone support is based on an underlying cost $c$ satisfying the strict Monge condition \eqref{eq:strictmonge}.
\end{rem}
Additionally, we see below in Proposition \ref{prop:uniquenessOT} that under suitable assumptions optimal transport couplings are unique almost surely. Before we prove these statements, it will be convenient to discuss assumption \eqref{ass:dualnondegen} for the optimal transport \eqref{eq:OT}. Notice that if each dual basic feasible solution is nondegenerate, then clearly \eqref{ass:dualnondegen} holds. By \cite{klee1968facets} every \emph{primal} basic feasible solution is nondegenerate if for any proper subsets $A,B\subset [N]$ not both empty we have
\begin{equation}\label{eq:summabilityprimal}
\sum_{i\in A} r_i \neq \sum_{j\in B} s_j.
\end{equation} 
We can prove a related condition such that every \emph{dual} basic feasible solution for \eqref{eq:DOT} is nondegenerate (see Appendix \ref{Appendix:OT} for a proof). Further and to the best of our knowledge, the following condition has not yet appeared in the literature.

\begin{thm}\label{thm:dualnondegenerateOT}
Consider the optimal transport \eqref{eq:OT} for marginals $r,s\in\Delta_N$ and given cost $c$. Suppose that for any $n\geq 2$ and any family of indices $\left\lbrace (i_k,j_k)\right\rbrace_{1\leq k \leq n}$ with all $i_k$ pairwise different and all $j_k$ pairwise different it holds that
\begin{equation}\label{eq:summabilitydual}
\sum_{k=1}^n c_{i_k j_k}\neq \sum_{k=1}^n c_{i_k j_{k-1}}, \quad j_0\coloneqq j_n.
\end{equation}
Then all dual basic solutions are nondegenerate and in particular \eqref{ass:dualnondegen} holds.
\end{thm}
Condition \eqref{eq:summabilitydual} can be considered as the dual to \eqref{eq:summabilityprimal} and we refer to it as the \emph{dual summability} condition. An immediate consequence is the uniqueness of the optimal transport coupling.

\begin{cor}
Suppose that the cost $c$ satisfies the \emph{dual summability} condition \eqref{eq:summabilitydual}. Then for any marginals $r,s\in\Delta_N$ the optimal transport \eqref{eq:OT} attains a unique optimal transport coupling.
\end{cor}

\begin{proof}
Simply observe that condition \eqref{eq:summabilitydual} implies each dual basic feasible solution to be nondegenerate. In particular, this holds for the dual optimal basic solution. By Proposition \ref{prop:slackness} we deduce uniqueness of the primal optimal basic solution.
\end{proof}
The dual summability condition \eqref{eq:summabilitydual} only depends on the cost and is independent of the marginal weights. On the real line with cost $c(x,y)=\vert x-y\vert^p$ and $p\ge0$, the condition holds if and only if $p\notin\{0,1\}$ (see Corollary \ref{cor:uniquerealline} and Remark \ref{rem:realline}). If the underlying space involves too many symmetries, such as a regular grid with cost defined by the underlying grid structure, it is usually never satisfied. Nevertheless, for many cost functions the set of all positions of finitely many points such that dual summability \eqref{eq:summabilitydual} does not hold has Lebesgue measure zero. Consequently, the optimal transport coupling $\pi^\star$ is generically unique.

\begin{prop}\label{prop:uniquenessOT}
For any two probability vectors $r,s\in\Delta_N$ define the probability measures $r(\mathbf{X})=\sum_{k=1}^N r_k\delta_{X_k}$ and $s(\mathbf{Y})=\sum_{k=1}^N s_k\delta_{Y_k}$ with random support drawn from two independent collections of $\mathbb{R}^D$-valued random variables $X_1,\ldots,X_N\overset{\text{i.i.d.}}{\sim}\mu$ and $Y_1,\ldots,Y_N\overset{\text{i.i.d.}}{\sim}\nu$ for $D\geq 2$. Suppose that $\mu$ and $\nu$ are absolutely continuous with respect to Lebesgue measure and consider the optimal transport \eqref{eq:OT} between $r(\mathbf{X})$ and $s(\mathbf{Y})$ with cost for $p,q>0$ and $p$ and $q$ not both equal to one between two points $x,y\in \mathbb{R}^D$ defined by
\begin{equation*}
c(x,y)=\Vert x-y\Vert^p_q=\left(\sum_{i=1}^D\vert x_i-y_i\vert^q\right)^{\frac{p}{q}}\, .
\end{equation*}
Then the dual summability condition \eqref{eq:summabilitydual} holds almost surely. In particular, with probability one for any $r,s\in\Delta_N$ and pair of marginals $r(\mathbf{X})$ and $s(\mathbf{Y})$, the corresponding optimal transport coupling is unique.
\end{prop}

\begin{rem}
As the proof shows, the result is valid for more general cost functions.  In particular, $p$ can be strictly negative and the result will hold true.  This includes the Coulomb cost $(p=-1)$ that has applications in physics \citep{cotar2013density}.  If $p\ne 1$, then $q$ can also be infinite, but if $q=\infty$ and $p=1$ then uniqueness fails in a similar fashion as in the case $p=q=1$ (see Remark \ref{rem:nonunique} below).
\end{rem}
\cite{wang2013linear} prove a similar result for the specific case $p=q=2$ and for \emph{fixed} marginal weights $r,s$. In comparison, our result holds uniformly over all marginal weights and for more general cost functions. Moreover, the statement remains correct if we define both measures $r,s$ on the same random locations $X_1,\ldots,X_N\overset{\text{i.i.d.}}{\sim}\mu$ with $\mu$ absolutely continuous with respect to Lebesgue measure, i.e., $r(\mathbf{X})=\sum_{k=1}^N r_k\delta_{X_k}$, $s(\mathbf{X})=\sum_{k=1}^N s_k\delta_{X_k}$ and cost $c(X_i,X_j)=\Vert X_i-X_j\Vert^p$ (see Remark \ref{rem:uniquenessOT} in Appendix \ref{Appendix:OT}).

\begin{rem}[Non-uniqueness in higher dimensions $D$ for $p\in\{0,1\}$]\label{rem:nonunique}
Similar cases as in the one dimensional case (Remark \ref{rem:realline}) can be found in higher dimensions, where we suppose here $q=1$. Suppose that $r$ and $s$ are supported on random locations $X_1,\ldots,X_N\overset{\text{i.i.d.}}{\sim} \mathcal{N}_D(\mu_1,\Sigma_1)$ and independent to that $Y_1,\ldots,Y_N\overset{\text{i.i.d.}}{\sim} \mathcal{N}_D(\mu_2,\Sigma_2)$, where $\mathcal{N}_D(\mu,\Sigma)$ is a $D$-dimensional Gaussian distribution with mean $\mu$ and covariance matrix $\Sigma$. For two vectors $x,y\in\mathbb{R}^D$ we write $x\leq y$ if the order holds coordinatewise. Conditioned on the event $\max X_i<\min Y_j$, any coupling between $r(\mathbf{X})$ and $s(\mathbf{Y})$ is an optimal transport coupling for cost $c(x,y)=\Vert x-y \Vert^p$ for $p\neq 1$ . In particular, this has positive probability and hence we cannot deduce for uniqueness almost surely. 
\end{rem}
We conclude with another sufficient conditions for assumption \eqref{ass:dualnondegen}.

\begin{prop}\label{prop:dualnondegenerateOT}
Consider the optimal transport \eqref{eq:OT} for marginals $r,s\in\Delta_N$ and given cost $c$. Then assumption \eqref{ass:dualnondegen} holds if for any proper subsets $A,B\subset [N]$ not both empty
\begin{equation*}
\sum_{i\in A} r_i \neq \sum_{j\in B} s_j
\end{equation*}
and either of the following two conditions holds:
\begin{itemize}
\item[(i)] The cost $c$ has strict Monge property \eqref{eq:strictmonge}.
\item[(ii)] The support of some optimal transport coupling $\pi^\star$ is strictly $c$-cyclically monotone.
\end{itemize}
\end{prop}

\textbf{Probabilistic assumptions.} Let us focus on the probabilistic assumptions \eqref{ass:cltforb} and \eqref{ass:feasforn}. For this purpose it turns out to be convenient to switch perspective and vectorize the optimal transport problem. Consider the cost $c$ as a vector in $\mathbb{R}^{N^2}$ with entries $c_{(i-1)N+j}\coloneqq c(x_i,x_j)$ and define the coefficient matrix (also known as node-arc incidence matrix) as
\begin{equation}\label{eq:OTmatrix}
A=\begin{pmatrix}
 \mathbf{1}_N^T\\
& \ddots &\\
& & \mathbf{1}_N^T\\
\mathbf{I}_N & \dots & \mathbf{I}_N 
\end{pmatrix}
\, \in \R^{2N\times N^2}.
\end{equation}
Then the set of all couplings $\Pi(r,s)$ can equivalently be defined as
\begin{equation*}
\Pi(r,s)\coloneqq\left\lbrace \pi\in \mathbb{R}^{N^2}\,\vert\, A\pi=\begin{bmatrix}
r\\s
\end{bmatrix},\,\pi\geq 0  \right\rbrace,
\end{equation*}
where the linear equation $A\pi=[r,s]^T$ simply encodes that the probability measure $\pi$ if considered as a matrix has to have marginals equal to $r$ and $s$, respectively. The constraint matrix $A\in \mathbb{R}^{2N\times N^2}$ in \eqref{eq:OTmatrix} has rank $2N-1$ instead of $2N$. In fact, for any coupling $\pi\in\Pi(r,s)$ fixing $N-1$ of its row sums and $N$ of its column sums suffices to fully characterize its marginals. This leads to one degree of freedom in the constraint matrix. To obtain a matrix of full rank we remove the $N$-th row of $A$ in \eqref{eq:OTmatrix} which motivates the following definition.

\begin{defn}[Reduced Optimal Transport]\label{rem:reducedvector}
Describe the set of all couplings between $r,s\in\Delta_N$ by 
\begin{equation}
\Pi_{\dagger}(r,s)\coloneqq \left\lbrace \pi\in \mathbb{R}^{N^2}\,\vert\, A_{\dagger}\pi=\begin{bmatrix}
r_{\dagger}\\s
\end{bmatrix},\,\pi\geq 0  \right\rbrace\, ,
\end{equation}
where  $r_{\dagger}\in\mathbb{R}^{N-1}$ consists of the first $N-1$ entries of $r$. Thus $\Pi_{\dagger}(r,s)$ is now characterized by $A_{\dagger}$ of full rank $2N-1$.
The subscript indicates that we remove the $N$-th row of the matrix $A$ in \eqref{eq:OTmatrix} and the $N$-th entry of the probability measure $r$ such that $A_\dagger\in\mathbb{R}^{(2N-1)\times N^2}$ and $r_{\dagger}\in\mathbb{R}^{N-1}$. 
\end{defn}
According to Remark \ref{rem:reducedvector} the optimal transport problem \eqref{eq:OT} can be stated as the standard linear program
\begin{equation}\tag{$\text{P}_\text{OT}$}\label{eq:OTprimal}
\min_{\pi\in \Pi_{\dagger}(r,s)} c^T\pi.
\end{equation} 
The fact that the set of all couplings between $r$ and $s$ only depends on the reduced vector $r_\dagger$ is reflected in the dual \eqref{eq:DOT} by deleting one dual variable. The corresponding dual linear program is
\begin{equation}\tag{$\text{D}_\text{OT}$}\label{eq:OTdual}
\begin{aligned}
\max_{\alpha \in \mathbb{R}^{N-1},\,\beta\in\mathbb{R}^N} \quad
r_\dagger^T\alpha+s^T\beta \quad
\text{s.t.} \quad
\alpha_i+\beta_j &\leq\, c_{(i-1)N+j}, \, \forall\, i,j\in[N-1],\\
\beta_j&\leq c_{(N-1)N+j},\, \forall\, j\in[N].
\end{aligned}
\end{equation}
To illustrate the upcoming distributional results, we consider the following optimal transport instance with $N=3$ that serves as our guiding example throughout the rest of this section. We also recall the notion of a feasible basis which is crucial for our general distributional theory.
\begin{exm}\label{exm:OT}
Let the ground space $\mathcal{X}=\{x_1<x_2<x_3\}$ consist of $N=3$ points on the real line with cost $c(x_i,x_j)=\vert x_i-x_j\vert^p$ for $p>0$. For two probability vectors $r,s\in\Delta_3$ on $\mathcal{X}$, the optimal transport problem \eqref{eq:OTprimal} reads
\begin{small}
\begin{equation*}
\begin{aligned}
\min_{\pi\in\mathbb{R}^9} \quad
c^T\pi \quad
\text{s.t.} \quad
 A_{\dagger}\pi=\begin{bmatrix}
r_{\dagger}\\s
\end{bmatrix},\,\pi\geq 0 
\end{aligned}
\end{equation*}
\end{small}
with cost $c=\left(0,\vert x_1-x_2\vert^p,\vert x_1-x_3\vert^p,
\vert x_2-x_1\vert^p,0,\vert x_2-x_3\vert^p,
\vert x_3-x_1\vert^p,\vert x_3-x_2\vert^p,0\right)\in \mathbb{R}^9$ and constraint (or node-arc incidence matrix)
\begin{small}
\begin{equation*}
A_\dagger =
\begin{pmatrix}
1&1&1\\
&&&1&1&1\\
1 &&& 1 &&& 1 \\
&1&&& 1 &&& 1\\
&& 1&&& 1 &&& 1\\
\end{pmatrix}\in \mathbb R^{5\times 9}.
\end{equation*}
\end{small} 
A basis $I$ is a subset of cardinality five out of the column index set $\{1,\ldots,9\}$ from $A_\dagger$ such that the sub-matrix ${A_\dagger}_I$ contains five independent columns. For example, the subset $I=\{1,2,3,5,9\}$ constitutes a basis. For optimal transport it is convenient to think of a primal feasible solution in terms of a transport matrix $\pi \in \mathbb{R}^{3\times 3}$ with $\pi_{ij}$ encoding mass transportation from source $i$ to destination $j$. In this way, a basis $I$ can be identified with its \textit{transport scheme} of primal basic transport matrices. More precisely, the basis $I=\{1,2,3,5,9\}$ corresponds to the transport scheme
\begin{small}
\begin{equation*}
TS(I)\coloneqq\begin{pmatrix}
\ast & \ast & \ast\\
 & \ast\\
&&\ast\\
\end{pmatrix}.
\end{equation*}
\end{small}
Each possible nonzero entry is marked by a star and the specific values are defined by $\pi(I,(r_\dagger,s))=({A_\dagger}_I)^{-1}(r_\dagger,s)\in\mathbb{R}^9$. In particular, the above transport scheme yields a primal basic feasible solution if and only if each coordinate is nonnegative. For instance, basis $I=\{1,2,3,5,9\}$ induces the primal basic feasible solution
\begin{small}
\begin{equation*}
\pi\left(I,(r_\dagger,s)\right)=\begin{pmatrix}
s_1 & s_2-r_2 & r_1+r_2-s_1-s_2\\
& r_2\\
&&s_1+s_2+s_3-r_1-r_2
\end{pmatrix}=\begin{pmatrix}
s_1 & s_2-r_2 & s_3-r_3\\
& r_2\\
&&r_3
\end{pmatrix},
\end{equation*}
\end{small}
where for the second equality we specifically use that the probability vectors $r,s$ sum up to one, respectively. By nonnegativity, $\pi\left(I,(r_\dagger,s)\right)$ is feasible if and only if $s_2\geq r_2$ and $s_3\geq r_3$.
\end{exm}

In statistical applications the probability measures $r,s\in\Delta_N$ are usually unknown and instead one has access to $\mathcal{X}$-valued random variables $X_1,\ldots,X_n\overset{i.i.d.}{\sim} r$ and, independently, $Y_1,\ldots,Y_m\overset{i.i.d.}{\sim} s$. We can then estimate $r,s\in\Delta_N$ by empirical probability measures
\begin{equation*}
\hat{r}_n\coloneqq \frac{1}{n}\sum_{i=1}^n \delta_{X_i}, \quad \hat{s}_m\coloneqq \frac{1}{m}\sum_{j=1}^m \delta_{Y_j}.
\end{equation*}
Notice that as $\hat{r}_n$ and $\hat{s}_m$ are again probability measures, the set of all couplings $\Pi_\dagger(\hat{r}_n,\hat{s}_m)$ remains non-empty and compact. In particular, assumption \eqref{ass:feasforn} is always satisfied. Hence, the empirical probability measures give rise to an empirical optimal transport coupling
\begin{equation*}
\hat{\pi}_{n,m} \in \argmin_{\Pi_\dagger(\hat{r}_n,\hat{s}_m)}\, c^T\pi,
\end{equation*}
which is denoted as the \emph{two-sample} case. We additionally denote by $\hat{\pi}_n$ the optimal transport coupling between $\hat{r}_n$ and $s$ which is referred to the \emph{one-sample} case since estimation is restricted to $r$. In view of our distributional limit results, we now characterize the statistical fluctuation of the empirical transport couplings $\hat{\pi}_n$ (and $\hat{\pi}_{n,m}$) around its population version 
\begin{equation*}
\pi^\star\in \argmin_{\Pi_\dagger(r,s)}\, c^T\pi
\end{equation*}
by a distributional limit law. In view of assumption \eqref{ass:cltforb} we require a weak limit for the empirical process for the marginal distributions. For $v\in \Delta_N$, we denote by $G(v)$ the $N$-dimensional central Gaussian distribution with covariance matrix 
\begin{equation}\label{eq:covmatrix}
\Sigma(v)\coloneqq \begin{bmatrix}
v_1(1-v_1) & -v_1v_2 & \ldots & -v_1v_N\\
-v_1v_2 & v_2(1-v_2) & \ldots & -v_2v_N\\
\vdots & & \ddots & \vdots\\
-v_1v_N & -v_2v_N &\ldots & v_N(1-v_N)
\end{bmatrix}\, .
\end{equation}
By the standard multivariate central limit theorem we conclude for sample size $n$ tending to infinity that 
\begin{equation}\label{eq:empprocess1}
\sqrt{n}\left(\hat{r_\dagger}_n-r_\dagger \right)\xrightarrow[]{D} G^1(r_\dagger)
\end{equation}
and similarly for the two sample case with $\frac{m}{n+m}\to \lambda\in (0,1)$ that 
\begin{equation}\label{eq:empprocess2}
\sqrt{\frac{nm}{n+m}}\left( \begin{bmatrix}
\hat{r_\dagger}_n\\ \hat{s}_m
\end{bmatrix}-\begin{bmatrix}
r_\dagger\\ s
\end{bmatrix} \right) \xrightarrow[]{D} \left( \sqrt{\lambda}G^1(r_\dagger),\sqrt{1-\lambda}G^2(s)\right),
\end{equation}
where $G^1$ and $G^2$ are independent. In the following, we assume that $r,s\in \text{ri}(\Delta_N)\coloneqq \{ r\in \mathbb{R}^N\, \vert \, \mathbf{1}_N^Tr=1,\, r_i> 0 \}$ the relative interior of $\Delta_N$. Since $r\in\text{ri}(\Delta_N)$, then the central limit law in \eqref{eq:empprocess1} is absolutely continuous with respect to Lebesgue measure. Nevertheless, the limit law in \eqref{eq:empprocess2} is not absolutely continuous since any realization of $G^2(s)$ sums up to zero. Although $G^2(s)$ is a centred Gaussian distribution in $\mathbb{R}^{N}$ its support is included in the $(N-1)$-dimensional hyperplane of vectors orthogonal to $\mathbf 1_N$. Hence, this case requires some care and we deal with it in Lemma \ref{lem:boundaryzero}.

\subsection{Limit Laws for Optimal Transport Couplings}

In this subsection we derive the distributional limit laws for empirical optimal transport couplings. We distinguish these results according to the one- and the two-sample case and whether the assumption \eqref{ass:dualnondegen} is present.

\subsubsection{One-Sample Case}

The one-sample case is derived by a straightforward application of our general theory. Notice that the limit distribution $G(r_\dagger)$ for the marginal $r_\dagger$ is absolutely continuous. 

\begin{cor}[One-Sample]
Consider the optimal transport problem in \eqref{eq:OTprimal} between two probability measures $r,s\in\text{ri}(\Delta_N)$ and suppose assumption \eqref{ass:uniqueOT} is satisfied. Then as $n$ tends to infinity it holds that  
\begin{equation}\label{eq:onesampleOT}
\sqrt{n}\left( \hat{\pi}_n-\pi^\star\right) \xrightarrow[]{D} \sum_\K \mathbbm{1}_{G^1(r_\dagger)\in H_\K\setminus \cup_{k\notin\K}H_k}\,\alpha^\K \otimes \pi(I_\K,G).
\end{equation}
If further assumption \eqref{ass:dualnondegen} holds then the limit reads as
\begin{align*}
\sqrt{n}\left( \hat{\pi}_n-\pi^\star\right) \xrightarrow[]{D} \sum_{k=1}^K \mathbbm{1}_{G^1(r_\dagger)\in H_k\setminus \cup_{j<k}H_j}\, \pi(I_k,G).
\end{align*}
\end{cor}
%

\subsubsection{Two-sample Case and Extensions}\label{subsec:twosample}

The two-sample case, where both marginals $r,s\in\text{ri}(\Delta_N)$ are estimated, presents an additional challenge in that the underlying empirical process in \eqref{eq:empprocess2} converges in distribution to a degenerate Gaussian distribution (see discussion at the end of Subsection \ref{subsec:assumptionsOT}). Absolute continuity of the limiting random variable is required for the general theory (without assumption \eqref{ass:dualnondegen}) in order to show that the boundaries of the cones $H_k$ have probability zero. Fortunately, the structure of the optimal transport problem allows to reach the same conclusion, as established in Lemma \ref{lem:boundaryzero}. The two-sample case then reads as follows.

\begin{thm}[Two-sample Case]\label{thm:limitforOT}
Consider the optimal transport problem \eqref{eq:OTprimal} between two probability measures $r,s\in\text{ri}(\Delta_N)$ and suppose that assumption \eqref{ass:uniqueOT} holds. Then as $m\wedge n$ tends to infinity such that $\frac{m}{n+m}\to \lambda\in(0,1)$ it holds that 
\begin{equation*}
\sqrt{\frac{nm}{n+m}}\left( \hat{\pi}_{n,m}-\pi^\star\right) \xrightarrow[]{D} \sum_\K \mathbbm{1}_{G\left(r_\dagger,s\right)\in H_\K\setminus \cup_{k\notin\K}H_k}\,\alpha^\K \otimes \pi(I_\K,G\left(r_\dagger,s\right)),
\end{equation*}
where $G(r_\dagger,s)=(\sqrt{\lambda}G^1(r_\dagger),\sqrt{1-\lambda}G^2(s))$ is a centred Gaussian distribution with block diagonal covariance matrix, where the blocks are given in \eqref{eq:covmatrix}.\\
If further assumption \eqref{ass:dualnondegen} holds then the limit law simplifies to
\begin{equation*}
\sqrt{\frac{nm}{n+m}}\left( \hat{\pi}_{n,m}-\pi^\star\right) \xrightarrow[]{D} \sum_{k=1}^K \mathbbm{1}_{G\left(r_\dagger,s\right)\in H_k\setminus \cup_{j<k}H_j}\, \pi(I_k,G\left(r_\dagger,s\right)).
\end{equation*}
\end{thm}

\begin{proof}
The proof follows the same lines as the proof of Theorem \ref{thm:limitGeneral}. The last step of the proof (see Section \ref{subsubsec:proofmain}) relying on the continuous mapping theorem requires the boundary $\partial(H_\mathcal{K}\setminus \cup_{k\notin \mathcal{K}} H_k) \subseteq \mathbb{R}^{N-1}$ for $\mathcal{K}\subseteq [K]$ to have $G$-measure zero. Notice that again by the union bound
\begin{equation*}
\begin{split}
\PR\left(G\in \partial(H_\mathcal{K}\setminus \cup_{k\notin \mathcal{K}} H_k)\right)&\leq \PR\left(G(r_\dagger,s)\in \partial H_{\mathcal{K}}\cup \partial \left(\cup_{k\notin \mathcal{K}}H_k\right)\right)\\
&\leq \sum_{k=1}^K\PR\left( G(r_\dagger,s)\in \partial H_k \right)=0,
\end{split}
\end{equation*}
where the last equality follows by Lemma \ref{lem:boundaryzero}.
\end{proof}

\begin{exm}[Example \ref{exm:OT} continued]
We revisit the optimal transport instance from Example \ref{exm:OT} and restrict to $p=1$. This choice of $p$ causes the nondegeneracy assumption \eqref{ass:dualnondegen} to fail and hence illustrates the most complicated situation of our theory. We further assume the two probability vectors $r$ and $s$ to be equal and strictly positive. The transport problem attains a unique degenerate solution supported on the diagonal, i.e., all the mass remains at its current location. A straightforward computation yields $K=8$ primal and dual optimal bases
\begin{center}
\begin{small}
\begin{tabular}{cccc}
$TS(I_1)=\begin{pmatrix}
\ast & \ast \\
&\ast&\\
&\ast&\ast
\end{pmatrix}$, & $TS(I_2)=\begin{pmatrix}
\ast &  \\
\ast&\ast&\ast\\
&&\ast
\end{pmatrix}$, & $TS(I_3)=\begin{pmatrix}
\ast & & \\
\ast & \ast& \\
\ast & &\ast
\end{pmatrix}$, & $TS(I_4)=\begin{pmatrix}
\ast & &\ast \\
 & \ast&\ast \\
 & &\ast
\end{pmatrix}$, \\ 
$TS(I_5)=\begin{pmatrix}
\ast & \ast & \ast \\
&\ast&\\
&&\ast
\end{pmatrix}$, & $TS(I_6)=\begin{pmatrix}
\ast &&  \\
&\ast&\\
\ast&\ast&\ast
\end{pmatrix}$, & $TS(I_7)=\begin{pmatrix}
\ast &  \\
\ast&\ast&\\
&\ast&\ast
\end{pmatrix}$, & $TS(I_8)=\begin{pmatrix}
\ast & \ast \\
&\ast&\ast\\
&&\ast
\end{pmatrix}$. \\
\end{tabular} 
\end{small}
\end{center}
For example, the transport scheme $TS(I_1)$ corresponds to basis $I_1=\{1,2,5,8,9\}$ and induces an invertible matrix $A_{\dagger I_1}$. The respective closed convex cones $H_k$ for $1\leq k\leq K$ as defined in \eqref{eq:Hk} are
\begin{small}
\begin{align*}
H_{1}&=\left\lbrace v\in \mathbb{R}^5\, \mid \, v_1\geq v_3,\, v_1+v_2\leq v_3+v_4\right\rbrace,\quad H_{2}=\left\lbrace v\in \mathbb{R}^5\, \mid \, v_1\leq v_3,\, v_1+v_2\geq v_3+v_4\right\rbrace,\\
H_{3}&=\left\lbrace v\in \mathbb{R}^5\, \mid \, v_2\geq v_4,\, v_1+v_2\leq v_3+v_4\right\rbrace,\quad H_{4}=\left\lbrace v\in \mathbb{R}^5\, \mid \, v_1\geq v_3,\, v_2\geq v_4\right\rbrace,\\
H_{5}&=\left\lbrace v\in \mathbb{R}^5\, \mid \, v_2\leq v_4,\, v_1+v_2\geq v_3+v_4\right\rbrace,\quad H_{6}=\left\lbrace v\in \mathbb{R}^5\, \mid \, v_1\leq v_3,\, v_2\leq v_4\right\rbrace,\\
H_{7}&=\left\lbrace v\in \mathbb{R}^5\, \mid \, v_1\leq v_3,\, v_1+v_2\leq v_3+v_4\right\rbrace, \quad H_{8}=\left\lbrace v\in \mathbb{R}^5\, \mid \, v_1\geq v_3,\, v_1+v_2\geq v_3+v_4\right\rbrace.
\end{align*}
\end{small}
Notice that each of these cones is an intersection of two proper half-spaces, respectively. Since assumption \eqref{ass:dualnondegen} fails to hold, some of these cones exhibit non-trivial intersections. Such cases arise for the pairs $\{I_3,I_7\}$, $\{I_6,I_7\}$, $\{I_4,I_8\}$ and $\{I_5,I_8\}$. The intersections of the corresponding cones are given by 
\begin{small}
\begin{align*}
H_3\cap H_7&=\left\lbrace v\in \mathbb{R}^5\, \mid \, v_2\geq v_4,\, v_1+v_2\leq v_3+v_4\right\rbrace,
&H_6\cap H_7=\left\lbrace v\in \mathbb{R}^5\, \mid \, v_1\leq v_3,\, v_2\leq v_4\right\rbrace,\\
H_5\cap H_8&=\left\lbrace v\in\mathbb{R}^5\, \mid \, v_2\leq v_4, \, v_1+v_2\geq v_3+v_4 \right\rbrace,
&H_4\cap H_8=\left\lbrace v\in \mathbb{R}^5\,\mid\, v_1\geq v_3,\, v_2\geq v_4 \right\rbrace.
\end{align*}
\end{small}
Together with the marginal weak distributional limit in \eqref{eq:empprocess2} the limit law for $p=1$ and $r=s$ reads 
\begin{equation*}
M(\mathbf{G})=\sum_{\K\in\{\{1\},\{2\}\{3,7\},\{6,7\},\{4,8\},\{5,8\}\}} \mathbbm{1}_{\mathbf{G}\in H_\K\setminus\cup_{k\notin\K}H_k}\,\alpha^\K \otimes x(I_\K,G).
\end{equation*}
A more detailed analysis, also illustrating the Hausdorff distance result in Theorem \ref{thm:Hausdorff}, is given in Appendix \ref{Appendix:FurtherIllustration}.
\end{exm}

The stated limit laws for empirical optimal transport couplings allow to conclude for distributional limit laws for various functionals thereof. A particular example has already been illustrated in Section \ref{subsec:Limitvalue} for the optimal value (see also \cite{sommerfeld2018inference} for its statistical consequences). Here, we give a number of other examples of such functionals in the optimal transport context.

\textbf{Optimal Transport Curve.} For an optimal transport coupling $\pi(r,s)$ the optimal transport curve (OTC) is defined as the function $OTC_{\pi(r,s)}:[0,1]\to[0,1]$ with
\begin{equation}
OTC_{\pi(r,s)}(t)
=\sum_{i=1}^{N^2}\pi(r,s)_i\mathbbm{1}\{c_i\le t\}.
\end{equation}
For $t\geq 0$ the value $OTC_{\pi(r,s)}(t)$ is the amount of mass that needs to travel a distance smaller than $t$ under the optimal transport coupling between the two probability measures $r,s$. Similarly, for entropically regularized transport $\pi_\lambda(r,s)$, this function has recently been introduced by \cite{klatt2020empirical} for colocalization analysis as a measure to quantify spatial proximity of protein interaction networks. 
The corresponding statistical analysis is based on a central limit theorem, and the limit is Gaussian in this case, which is also a consequence of uniqueness of the optimal solution, as the regularization makes the optimization strictly convex for $\lambda>0$. We demonstrate here that the limit law for $OTC$ based on unregularized transport couplings follows from our theory. Notice that since $OTC$ is right-continuous and nondecreasing, it can be embedded in the Banach space $\mathcal B$ of c\`adl\`ag functions on $[0,1]$ endowed with the supremum norm. We consider here the case that both marginals $r,s\in\Delta_N$ are unknown and hence estimated by their empirical counterparts $\hat{r}_n,\hat{s}_m\in \Delta_N$; analogous results for the one-sample case can be obtained similarly under appropriate assumptions.

\begin{thm}
Let $OTC_{\pi(r,s)}$ denote the optimal transport curve based on the optimal transport plan $\pi(r,s)$ and suppose that assumption \eqref{ass:uniqueOT} holds. Then under the same setting as Theorem \ref{thm:limitforOT} it holds that
\begin{equation*}
\sqrt{\frac{nm}{n+m}}\left(OTC_{\pi(\hat{r}_n,\hat{s}_m)} - OTC_{\pi_{(r,s)}}\right)
\xrightarrow[]{D}  OTC_M
\textrm{ in }\mathcal B,
\end{equation*}
where $M$ is given as in Theorem \ref{thm:limitforOT}.
\end{thm}

\begin{proof}
The mapping $OTC:\R^{N^2}\to \mathcal B$ is linear with finite-dimensional domain and hence continuous. The continuous mapping theorem yields
\begin{equation*}
\sqrt{\frac{nm}{n+m}}\left(OTC_{\pi(\hat{r}_n,\hat{s}_m)} - OTC_{\pi_{(r,s)}}\right)
= OTC_{\sqrt{\frac{nm}{n+m}}\left(\pi(\hat{r}_n,\hat{s}_m) - \pi_{(r,s)}\right)}
\xrightarrow[]{D} OTC_M
\end{equation*}
in the space $\mathcal B$.
\end{proof}

\textbf{Trace of a Transport Coupling.} For a matrix $A$ we denote by $tr[A]$ the trace of $A$, i.e., the sum of all diagonal elements. The trace of the transport coupling represents how much mass stays at place.
\begin{thm}
Let $\pi(\hat{r}_n,\hat{s}_m)$ and $\pi_{(r,s)}$ denote the empirical and true optimal transport coupling. Under the conditions of Theorem \ref{thm:limitforOT} it holds that
\begin{equation*}
\sqrt{\frac{nm}{n+m}}(\tr[\pi(\hat{r}_n,\hat{s}_m)] - \tr[\pi(r,s)])
\xrightarrow[]{D}
tr[M].
\end{equation*}
\end{thm}

\begin{proof}
By linearity of the trace.
\end{proof}
\textbf{Geodesic between Probability Measures.} When $\mathcal{X}$ is a subset of a Banach space and the cost function associated with the optimal transport is $c(x,y)=\|x-y\|^p$ for $p>1$, the optimal transport coupling $\pi(r,s)$ defines a geodesic in Wasserstein space between its
marginals \citep{mccann1997convexity}, known as McCann's interpolant. It is defined by
\begin{equation*}
Geod_{\pi(r,s)}(t)
=[(1-t)\proj_1 + t\proj_2]\#\pi(r,s),
\end{equation*}
where $\proj_1(x,y)=x$ and $\proj_2(x,y)=y$ for $x,y\in \mathcal X$ are the respective projection maps and $[f\#\pi(r,s)](A)=\pi(r,s)(f^{-1}(A))$ is the push-forward for $f$ a measurable map and $A$ a measurable subset of the range of $f$ \cite[Theorem 5.27]{santambrogio2015optimal}. The empirical transport coupling $\pi(\hat{r}_n,\hat{s}_m)$ provides an approximate geodesic denoted by $Geod_{\pi(\hat{r}_n,\hat{s}_m)}$. The geodesics are viewed as elements of the Banach space $\mathcal C=C([0,1],\mathcal M(\textrm{Conv}(\mathcal X)))$ of continuous functions from $[0,1]$ to the space of finite signed Borel measures on the convex hull of $\mathcal X$ endowed with the Bounded Lipschitz norm \eqref{eq:BLnorm}.
\begin{thm}
Let $\pi(\hat{r}_n,\hat{s}_m)$ and $\pi_{(r,s)}$ denote the empirical and true optimal transport coupling. Under the conditions of Theorem \ref{thm:limitforOT} it holds that
\begin{equation*}
r_m\left(Geod_{\pi(\hat{r}_n,\hat{s}_m)} - Geod_{\pi(r,s)}\right)
\xrightarrow[]{D}
Geod_M
\textrm{ in }\mathcal C
.
\end{equation*}
\end{thm}
\begin{proof}
For fixed $t$ the function $\pi\mapsto Geod_{\pi(r,s)}(t)$ is linear.  Therefore $\pi\mapsto Geod_{\pi(r,s)}$ is linear from $\R^{N^2}$ to $\mathcal C$ and is thus continuous. The continuous mapping theorem applies.
\end{proof}

%% file: Sections/AppendixA.tex
\section{Linear Programs and Duality}\label{Appendix:LinearProgramProofs}

This section contains the proofs of statements from Section \ref{sec:Preliminaries}. The results are well-known, but their proofs are not always easy to find in a form convenient for the present paper. We therefore provide the proofs for the sake of completeness.

\begin{proof}[Lemma \ref{lem:rulesOnly}]\label{proof:rulesOnly}
Suppose the feasible set $P(\tilde{b})\coloneqq \left\{x\in\R^d \, \mid\, Ax=\tilde{b},\, x\geq 0\right\}$ is non-empty, else there is nothing to prove. According to condition \eqref{ass:optBounded} the set $OPT(b)$ is bounded. In particular, this implies that the only vector $h$ such that $Ah=0$ and $c^Th\leq 0$ is equal the zero vector, i.e., the recession cone of the feasible set is singleton \cite[Section 4.8]{bertsimas1997introduction}. Hence, the primal program $(\text{P}_{\tilde{b}})$ attains a finite optimal solution and so does its dual by Theorem \ref{thm:strongduality}. Therefore, the non-empty set $\{c^Tx\, \mid \,x\in P(\tilde{b})\}$ is bounded below and has a minimum $\tilde{c}=c(\tilde{b})$. For arbitrary vector $a\in \mathbb{R}^d$ consider the following linear program and its dual

\medskip
\begin{minipage}{0.45\textwidth}
\begin{equation}\label{proof:primalp}
\begin{aligned}
\min\limits_{x\geq 0}{} \quad
a^T x& \\
\text{s.t.} \quad
Ax = 0, \, &c^Tx=0, 
\end{aligned}\tag{$\text{P}^1$}
\end{equation}
\end{minipage}
\begin{minipage}{0.45\textwidth}
\begin{equation}\label{proof:dualp}
\begin{aligned}
\max_{\lambda\in\mathbb{R}^m,\alpha\in\mathbb{R}} \quad
0^T\lambda + 0\alpha & \\
\text{s.t.} \quad
A^T\lambda+\alpha c \leq a&.
\end{aligned}\tag{$\text{D}^1$}
\end{equation}
\end{minipage}

As \eqref{proof:primalp} attains a finite optimal value we conclude by strong duality (see Theorem \ref{thm:strongduality}) that the dual program \eqref{proof:dualp} attains an optimal solution. In particular, the set $\{(\lambda,\alpha)\in\R^{m}\times \R:A^T\lambda + \alpha c\leq a\}$ is non-empty for all $a\in\R^d$. Hence, for $j\in [d]$ consider $a=-e_j$, the negative $j$-th unit vector, and the related primal and dual linear programs

\medskip
\begin{minipage}{0.45\textwidth}
\begin{equation}\label{proof:primalpp}
\begin{aligned}
\min\limits_{x\geq 0}{} \quad
-x_j& \\
\text{s.t.} \quad
Ax = \tilde{b}, \, &c^Tx=\tilde{c}, 
\end{aligned}\tag{$\text{P}^2$}
\end{equation}
\end{minipage}
\begin{minipage}{0.45\textwidth}
\begin{equation}\label{proof:dualpp}
\begin{aligned}
\max_{\lambda\in\mathbb{R}^m,\alpha\in\mathbb{R}} \quad
b^T\lambda + \alpha \tilde{c} & \\
\text{s.t.} \quad
A^T\lambda+\alpha c \leq -e_j&.
\end{aligned}\tag{$\text{D}^2$}
\end{equation}
\end{minipage}

Notice that the constraint set for \eqref{proof:primalpp} is equal to $OPT(\tilde{b})$. By feasibility of \eqref{proof:dualpp} we see that the primal \eqref{proof:primalpp} is bounded below for all $j\in [d]$. Hence $x_j$ is bounded above on the set of optimizers $OPT(\tilde{b})$ and clearly bounded below by the nonnegativity constraint. This proves that $OPT(\tilde{b})$ is a bounded polytope and hence the convex hull of its extreme points \cite[Theorem 2.9]{bertsimas1997introduction}. These are precisely the extreme points of $P(\tilde{b})$ that belong to $OPT(\tilde{b})$, and they must take the form $x(I,\tilde{b})$ for some primal feasible basis $I$ \cite[Section 2.5]{luenberger2008linear}.

It remains to show that the convex hull of such extreme points coincides with the convex hull of extreme points induced by considering only dual feasible bases $I$ such that $x(I,\tilde{b})$ is primal optimal. For this, suppose that $\lambda(I)$ is infeasible for the dual program $(\text{P}_{\tilde{b}})$ and $x(I,\tilde{b})\in OPT(\tilde{b})$.  It suffices to show that there exists another basis $J$ such that $\lambda(J)$ is dual feasible and $x(J,\tilde{b})=x(I,\tilde{b})$. Infeasibility of $\lambda(I)$ means that there exists an index $k\notin I$ with negative reduced cost, namely such that
\begin{equation*}
c_k
<[\lambda(I)^TA]_k
=\lambda(I)^TA_k
=c_I^T(A_I)^{-1}A_k.
\end{equation*}
Define the direction $\theta\in\R^d$ that equals $(A_I)^{-1}A_k$ on the coordinates corresponding to $I$ and $\theta_k=1$; in symbols $\theta=\Aug_{\{k\}}1-\Aug_I(A_I)^{-1}A_k$. Then
\begin{equation*}
A\theta=A_k-A_I(A_I)^{-1}A_k=0
\qquad\textrm{ and }\qquad
c^T\theta
=c_k - c_I^T (A_I)^{-1}A_k<0.
\end{equation*}
If we initiate the simplex method with basis $I$ using any anti-cycling rule, such as the one proposed in \cite{bland1977new}, then it will not terminate at $I$.  Since $x(I,\tilde{b})$ is optimal, the simplex method cannot change it (otherwise the cost will decrease strictly), so it must terminate at some other basis $J$ such that $x(I,\tilde{b})=x(J,\tilde{b})$ with the property that $\lambda(J)$ is dual feasible.
\end{proof}

\begin{proof}[Proposition \ref{prop:slackness}]\label{proof:slackness}
In order to prove Proposition \ref{prop:slackness} we require a result known as \emph{strict complementary slackness}. As a rigorous proof is not easy to find, we give one for the sake of completeness.  Our proof follows the idea suggested in \citet[Exercise 4.20]{bertsimas1997introduction}.
\begin{lem}[Strict complementary slackness]\label{lem:strictslackness}
Consider the primal linear program \eqref{eq:standardLP} and its dual \eqref{eq:standardDP}. Assume that both programs have an optimal solution. Then there exist optimal solutions $x^\star$ and $\lambda^\star$ to the primal and dual, respectively, such that for every $j\in [d]$ either $x^\star_j>0$ or $A_j^T\lambda^\star<c_j$.
\end{lem}
\begin{proof}
First, fix some $j$ and suppose that every optimal solution $x$ to the primal satisfies $x_j=0$. Let $c_0$ be the optimal value for the primal program and consider the linear program \eqref{proof:primalppp} and its corresponding dual \eqref{proof:dualppp} given by
\begin{center}
\begin{minipage}{0.45\textwidth}
\begin{equation}\label{proof:primalppp}
\begin{aligned}
\min_{x\geq 0} \quad
-x_j& \\
\text{s.t.} \quad
Ax &= b, \,\,
-c^Tx\geq -c_0
\end{aligned}\tag{$\text{P}^3$}
\end{equation}
\end{minipage}
\begin{minipage}{0.45\textwidth}
\begin{equation}\label{proof:dualppp}
\begin{aligned}
\max_{\mu \geq 0} \quad
b^T\lambda -\mu c_0 & \\
\text{s.t.} \quad
A^T\lambda -\mu c&\leq -e_j
\end{aligned}\tag{$\text{D}^3$}
\end{equation}
\end{minipage}
\end{center}
where $\mu\in \mathbb{R}$ is the Lagrange parameter for $-c^Tx\geq -c_0$ and $e_j$ the $j$-th canonical unit vector in $\mathbb{R}^d$. Notice that the feasible set for \eqref{proof:primalppp} is precisely the set of optimal solutions for \eqref{eq:standardLP}. By assumption the optimal value for \eqref{proof:primalppp} is zero and so is the optimal value of \eqref{proof:dualppp} by strong duality. This implies that any optimal dual solutions $(\lambda,\mu)$ for \eqref{proof:dualppp} is such that $b^T\lambda=\mu c_0$. Suppose first that there exists an optimal solution $(\tilde{\lambda},\mu)$ for \eqref{proof:dualppp} such that $\mu>0$. Then $\lambda\coloneqq \nicefrac{\tilde{\lambda}}{\mu}$ is feasible for \eqref{eq:standardDP} and a straightforward calculation shows that it is optimal. In particular, we find
$A_j^T\lambda\leq c_j - \nicefrac{1}{\mu} < c_j$. For the case that $(\tilde{\lambda},0)$ is optimal for \eqref{proof:dualppp} we take any optimal solution $\lambda^\star$ for \eqref{eq:standardDP} and define $\lambda=\lambda^\star+\tilde{\lambda}$. As before, $\lambda$ is feasible and optimal for \eqref{eq:standardDP} and fulfils $A_j^T\lambda\leq c_j-1<c_j$. We conclude that there exists an optimal solution $\lambda$ for \eqref{eq:standardDP} such that $A_j^T\lambda\leq c_j-1<c_j$.

Now, consider the set of all optimal dual solutions for \eqref{eq:standardLP}. Let $I\subseteq [d]$ be such that all optimal primal solutions satisfy $x_i=0$ for all $i \in I$. According to the previous conclusion, for all $i\in I$ there exists $\lambda_i$ optimal for \eqref{eq:standardDP} such that $A_i^T\lambda_i<c_i$.  For all coordinates with index $k\notin I$ let $x_k^\star$ be optimal for \eqref{proof:primalppp} with $k$-th coordinate nonzero.  By usual complementary slackness $A_k^T\lambda_i=c_k$ for all $k\notin I$ and all $i\in I$.  If $I$ is non-empty then the averages
\[
x^\star
=\frac1{d-|I|}\sum_{k\notin I}x^\star_k
\quad (x^*=0 \text{ if }I=[d]),
\qquad \lambda^\star
=\frac1{|I|}\sum_{i\in I}\lambda_i,
\]
are optimal and have the desired properties.  If $I=\emptyset$ then, since $x^\star$ is strictly positive, the usual complementary slackness yields a $\lambda^\star$ such that $A^T\lambda^\star=c$, and $(x^\star,\lambda^\star)$ satisfy the required conditions.
\end{proof}
We are now able to prove Proposition \ref{prop:slackness}. According to Theorem \ref{thm:strongduality} as soon as one program attains an optimal solution, so does the other. In particular, both programs admit optimal basic solutions. Moreover, by strong duality any pair $(x^\star,\lambda^\star)$ of primal and dual feasible solutions is optimal if and only if \emph{complementary slackness} holds, i.e.\ 
\begin{equation*}
x^\star_i(c_i-A_i^T \lambda^\star)=0, \, \forall i \in [d].
\end{equation*}
Notice that in order to prove the statements it suffices to assume that the primal \eqref{eq:standardLP} fulfils the assumptions and then to conclude for the dual \eqref{eq:standardDP}. This follows as the dual can be transformed into standard form and the roles of primal and dual are interchangeable.
\begin{itemize}
\item[(i)] Let $x^\star$ be a nondegenerate primal optimal solution. In particular, there exists an index set $I\subseteq [d]$ with $\vert I \vert=m$ of the columns of $A$ such that $A_I$ is invertible and $x^\star=x(I,b)$ with $x_i>0$ for all $i \in I$ and $x_i=0$ for all $i\notin I$. Hence, by complementary slackness for any dual optimal solution $\lambda^\star$ it holds that $c_i-A_i^T\lambda^\star=0$ for all $i \in I$. The latter equalities read $A_I^T\lambda^\star=c_I$ which determines $\lambda^{\star}$ uniquely.

\item[(ii)] Uniqueness of the dual solution again follows by (i). Further note that as the primal solution is unique we deduce by strict complementary slackness Lemma \ref{lem:strictslackness} that $c_i-A_i^T\lambda^\star>0$ for all $i \notin I$. Hence, the dual solution is nondegenerate.

\item[(iii)] Assume the primal attains a unique but degenerate optimal basic solution $x^\star$. This means that there exists an index set $I\subseteq [d]$, such that $\vert I \vert<m$, indexing all positive entries of $x^\star$. More precisely, $x^\star_i>0$ for all $i \in I$ and $x^\star_i=0$ for all $i \in I^c$. By strict complementary slackness there is a dual optimal solution $\lambda^\star$ satisfying $A_i^T\lambda^*_i<c_i$ for all $i\in I^c$. As $|I^c|=d-|I|>d-m$ the solution $\lambda^\star$ cannot be basic. Nevertheless, there always exists at least one dual optimal basic solution. This completes the proof.
\end{itemize} 
\end{proof}

\begin{proof}[Lemma \ref{lem:dualnondegenerate}]
Suppose first that $\lambda(I_j)=\lambda(I_k)$ for some $j\neq k$, i.e., assumption \eqref{ass:dualnondegen} does not hold. Then for all indices $i\in I_j\cup I_k$ we have equalities in $[A^T\lambda(I_k)]_i= c_i$. As $j\neq k$ the union $I_j\cup I_k$ has cardinality greater than $m$. Consequently, $\lambda(I_j)$ is degenerate.\\
For the converse, let $\lambda(I_1),\ldots,\lambda(I_K)$ be all dual optimal basic solutions induced by basis index sets $I_1,\ldots,I_K$ each of cardinality $m$ and recall that by definition $x(I_j,b)$ for $1\leq j \leq K$ is also primal optimal. Suppose there exists an index set $I_j$ with $1\leq j \leq K$ such that $\lambda(I_j)$ is a degenerate dual optimal basic solution. By definition of degeneracy there exists an index set $L$ of active constraints in the dual, i.e., for each $l\in L$ it holds that 
\begin{equation*}
\lambda^T(I_j)\cdot a_l=c_l
\end{equation*}
with the property $\vert L \vert>m$ and $I_j\subseteq L$. Denote by \textit{Pos} the positive entries for the optimal primal basic solution $x(I_j,b)$. Then we have $Pos \subseteq I_j$ and can partition 
\begin{equation*}
L=I_j \, \sqcup \, L\setminus I_j = I_j \,\sqcup\, I_j\setminus Pos \,\sqcup\, L\setminus I_j\, .
\end{equation*}
Notice that by definition the columns of the constraint matrix $A_{I_j}$ form a basis of $\mathbb{R}^m$. Hence, any column of the constraint matrix $A_{L\setminus I_j}$ can be written as
\begin{equation}\label{eq:representation1}
a_z=\sum_{i\in Pos} y_i^z a_i + \sum_{s\in I_j\setminus Pos} y_s^za_s\, , \, z\in L\setminus I_j\, .
\end{equation}
Suppose there exists some index $z\in L\setminus I_j$ and $s\in I_j\setminus Pos$ such that $y_s^z\neq 0$. Then we can define a new basis $\tilde{I}\coloneqq I_j\setminus\{s\}\cup \{z\}$ such that the columns of $A_{\tilde{I}}$ define a basis of $\mathbb{R}^m$. Further, by degeneracy $\lambda(\tilde{I})=\lambda(I_j)$ and as $Pos\subseteq \tilde{I}$ we conclude that $x(\tilde{I},b)=x(I_j,b)$. Consequently, there exists an index $i\neq j$ with $1\leq i\leq K$ such that $\tilde{I}=I_i$ contradicting assumption \eqref{ass:dualnondegen}. Hence, for the representation \eqref{eq:representation1} we conclude $y_s^z=0$ for all $s\in I_j\setminus Pos$ and find that for any $z\in L\setminus I_j$ it holds that 
\begin{equation}\label{eq:representation2}
a_z=\sum_{i\in Pos} y_i^z a_i \, .
\end{equation}
Suppose there exists an index 
\begin{equation*}
i_0\in \argmin_{i\vert y_i^z > 0} \frac{x_i}{y_i^z}\, ,
\end{equation*}
where $x_i>0$ for $i\in Pos$ are the positive entries of $x(I_j,b)$. Then we can rewrite 
\begin{equation*}
a_{i_0}=\frac{1}{y_{i_0}^z} a_z - \sum_{\substack{i \in Pos\\ i\neq i_0}} \frac{y_i^z}{y_{i_0}^z} a_i
\end{equation*}
and further find that
\begin{equation*}
\begin{split}
b=\sum_{i\in Pos} x_i a_i = \sum_{\substack{i \in Pos \\ i\neq i_0}} x_i a_i + x_{i_0}\left( \frac{1}{y_{i_0}^z}a_z - \sum_{\substack{i \in Pos \\ i\neq i_0}}\frac{y_i^z}{y_{i_0}^z}a_i\right)
=\frac{x_{i_0}}{y_{i_0}^z} a_z +\sum_{\substack{i \in Pos \\ i\neq i_0}}\left( x_i-\frac{x_{i_0}y_i^z}{y_{i_0}^z}\right)a_i
=\sum_{i\in \tilde{I}} \tilde{x_i}a_i
\end{split}
\end{equation*}
for some proper choice of $\tilde{x_i}$. By choice of the index $i_0$ we find that $\tilde{x_i}$ are nonnegative, i.e., $\tilde{I}$ is a primal and dual optimal basis. Moreover, $\lambda(\tilde{I})=\lambda(I_j)$ that again contradicts \eqref{ass:dualnondegen}. We deduce for the representation in \eqref{eq:representation2} that $y_i^z\leq 0$. Consider the vector 
\begin{equation*}
w\coloneqq \text{Aug}_{\{z\}}(1)-\text{Aug}_{Pos}(y^z)\, .
\end{equation*}
By definition $w\geq 0$ and $Aw=0$. Further, by dual degeneracy and the representation of $a_z$ it holds that
\begin{equation}
c_z=\sum_{i\in Pos} y_i^z c_i\, .
\end{equation}
We deduce that $c^Tw=0$ and hence $w$ is a primal optimal ray, i.e., for each $t\geq 0$ the vector $\bar{x}(t)=x(I_j,b)+tw$ is primal optimal. This contradicts assumption \eqref{ass:optBounded}. In total we see that if any basis $I_j$ for $1\leq j \leq K$ yields a degenerate dual basic solution we can modify basis $I_j$ to some $I_i$ with $i\neq j$ and $1\leq i \leq K$ such that $\lambda(I_j)=\lambda(I_j)$.

It remains to prove that also $\lambda(I_l)$ for $K+1\leq l \leq N$ are nondegenerate. Recall that any basis $I_l$ for $K+1\leq l \leq N$ yields an optimal dual basic solution but $x(I_l,b)$ is not primal optimal. We prove that such bases do not exist under assumption \eqref{ass:dualnondegen}. Consider any optimal primal basic solution $x(I_j,b)$ for $1\leq j \leq K$ and denote by $Pos_j$ its positivity set. Then by strong duality 
\begin{equation*}
0=c^Tx(I_j,b)-b^T\lambda(I_l)=\sum_{i\in Pos_j} x_i\left( c_i-a_i^T\lambda(I_l)\right)\, .
\end{equation*}
In particular, as $I_l$ is not a primal optimal basis it holds that $Pos\nsubseteq I_l$ and above equality implies that $\lambda(I_l)$ is necessarily degenerate with active constraint set $L$ including $I_l\cup Pos_j$. But then we can modify basis $I_l$ to some primal and dual optimal basis $I_i$ for $1\leq i \leq K$ such that $\lambda(I_i)=\lambda(I_l)$ is degenerate. Hence, such basis $I_l$ cannot exist under assumption \eqref{ass:dualnondegen}. Hence, any optimal dual basic solution is nondegenerate and induced by some basis primal and dual optimal basis $I$.
\end{proof}

%% file: Sections/AppendixB.tex
\section{On the Assumptions}\label{Appendix:Assumptions}

\begin{proof}[Lemma \ref{lem:boundedness1}]
By assumption $P(b_0)$ is feasible and bounded. Boundedness is equivalent to the set $\{x\in\mathbb{R}^d\,\vert\,Ax=0,x\geq0\}$ containing only the zero vector. This set is known as the \emph{recession cone} of $A$ \cite[Section 4.8]{bertsimas1997introduction} and does not depend on $b_0$. In particular, this implies that for any $b\in\mathbb{R}^m$ the set $P(b)$ is either feasible and bounded or empty (hence bounded).
\end{proof}

\begin{proof}[Lemma \ref{lem:boundedness2}]
By assumption, for each $i\in[d]$ there exists $j\in[m]$ such that $A_{ji}>0$. Since $A$ has nonnegative entries, the $j$-th constraint enforces that $x_i\le b_j/A_{ji}$. Since $i$ is arbitrary, this implies that $P(b)=\{Ax=b,x\ge0\}$ is bounded. (It will be empty if $\min_jb_j<0$.)
\end{proof}

\begin{proof}[Lemma \ref{lem:feasibility}]
\begin{itemize}
\item[(i)] The first condition states existence of basis $I$ such that $x(I,b)=\text{Aug}_I[A_I^{-1}b_0]$ is feasible and the coordinates $i\in I$ contain positive values. Now, notice that $b\mapsto \text{Aug}_I[A_I^{-1}b]$ is a continuous function in $b$ and hence for any $b$ sufficiently close to $b_0$ we have that $x(I,b)$ remains positive on the coordinates $i\in I$ and by definition of the operator $Ax(I,b)=b$. Hence, we conclude that $P(b)\neq \emptyset$ for all $b$ sufficiently close to $b_0$.
\item[(ii)] According to \citet[Theorem 4.15]{bertsimas1997introduction} there exist bases $I_1,\ldots,I_K$ and elements $w_1,\ldots,w_R$ in the recession cone of $A$ (so that $w_i\in\R^d_+$ and $Aw_i=0$) such that 
\begin{equation*}
0<\bar{x}=\sum_{k=1}^K \lambda_k x(I_k,b_0)+\sum_{r=1}^R \beta_r w_r,\quad \sum_{k=1}^K\lambda_k=1,\, \lambda_k\geq 0,\, \beta_r\geq 0.
\end{equation*}
Notice that, as before, $b\mapsto x(I_k,b)$ is continuous in $b$, and the elements $w_r$ do not depend on $b$. In particular, replacing $b_0$ by some $b$ sufficiently close to $b_0$, we find that 
\begin{equation*}
0<x=\sum_{k=1}^K \lambda_k x(I_k,b)+\sum_{r=1}^R \beta_r w_r,\quad \sum_{k=1}^K\lambda_k=1,\, \lambda_k\geq 0,\, \beta_r\geq 0
\end{equation*}
is nonnegative and feasible by definition. This shows that $P(b)\neq \emptyset$.
\end{itemize}
\end{proof}

\begin{proof}[Lemma \ref{lem:uniqueness}]
Consider the dual linear program $\max_{A^T\lambda\le C}b^T\lambda$. We may assume without loss of generality that the dual attains an optimal solution (otherwise the primal is either unbounded or infeasible). Then there exists a dual optimal basic solution $\lambda(I)$ for some basis $I$ and it holds that $A^T\lambda(I)=C_{I}$. By definition, the dual is degenerate if there exists an index $j\notin I$ such that $A^T\lambda(I)-C_j=0$. As $C_j$ has a density, this event occurs with probability zero. Taking the (finite) union over all indices $l\notin I$ proves that with probability zero $\lambda(I)$ is degenerate. By Proposition \ref{prop:slackness} the primal is unique almost surely.
\end{proof}

\begin{proof}[Lemma \ref{lem:primalNotUnique}]
Suppose that $\lambda(I_1)=\lambda^{\star}=\lambda(I_2)$, i.e., the dual feasible solution $\lambda^{\star}$ is degenerate.  Consider the set $H$ of all $v\in\R^m$ such that $\lambda^\star$ uniquely maximizes $v^T\lambda$ over the dual feasible region and $v$ is in the open cone spanned by $A$, namely
\[
H=\bigcap_{I:\lambda(I)\ne \lambda^\star,A^T\lambda(I)\le c}\{v\in\mathbb R^m:v^T\lambda^\star>v^T\lambda(I)\}\bigcap \{Ax:x\in (0,\infty)^d\}.
\]
Then $H$ is an open non-empty (convex) cone. Moreover, there exists $\epsilon>0$ such that if $\|b-b_0\|<\epsilon$, none of $\lambda(I_k)$ for $k>K$ can be optimal for (\hyperref[eq:standardDP]{$\text{D}_b$}). Let $H_\epsilon=H\cap\{x\in\R^m:\|x\|<\epsilon\}$, which is again open and non-empty (since $H$ is a cone). By assumption we have $0<\PR (G\in H_\epsilon)\le \liminf_{n\to\infty} \PR (G_n\in H_\epsilon)$ and, invoking \eqref{ass:feasforn}, we see that $\liminf_{n\to\infty} \PR(G_n\in H_\epsilon,x^{\star}(b_n)\textrm{ exists})>0$. But the latter event entails that the dual problem (\hyperref[eq:standardDP]{$\text{D}_{b_n}$}) is bounded and has a unique degenerate solution $\lambda^{\star}$. Proposition \ref{prop:slackness} implies that in this case $x^{\star}(b_n)$ is not unique.
\end{proof}

\begin{proof}[Lemma \ref{lem:jointConv}]
We need to show that for any $\varphi \in \BL(\Z)$ we have that
\begin{equation*}
\begin{split}
\lvert \E[\varphi(\alpha_n^\K,G_n)]- \E[\varphi(\alpha^\K,G)] \rvert &= \lvert \E\psi_n(G_n)- \E\psi(G) \rvert \\
&\leq \lvert\E [\psi_n(G_n) - \psi(G_n)]\rvert  + \lvert \E[\psi(G_n) - \psi(G)] \rvert
\end{split}
\end{equation*}
vanishes as $n\to\infty$. To bound the first term notice that in particular for any fixed $g$ it holds that $\|\alpha\mapsto\varphi(\alpha,g)\|_{\BL(\Delta_{|\K|})}\le \|\varphi\|_{\BL(\Z)}\le1$, so
\begin{equation*}
|\psi_n(g) - \psi(g)|
\le \int_{\Delta_{|\K|}}|\varphi(\alpha,g)|d[\mu_{n,g}^\K-\mu_g^\K](\alpha)
\le \BL(\mu_{n,g}^\K,\mu_g^\K).
\end{equation*}
Hence, we find $\E |\psi_n(G_n) - \psi(G_n)|\le \sup_g \BL(\mu^\K_{n,g},\mu^\K_g)$ that tends to zero by the uniform assumption. Notice that the supremum can be an essential supremum, i.e., taken on set of full measure with respect to both $G_n$ and $G$ instead of the whole of $\R^m$. For the second term observe that $\|\psi\|_\infty\le\|\varphi\|_\infty$ and that
\begin{align*}
|\psi(g_1)-\psi(g_2)|
&=\Bigg|\int_{\Delta_{|\K|}}\varphi(\alpha,g_1)-\varphi(\alpha,g_2)d\mu_{g_1}^\K(\alpha)
+\int_{\Delta_{|\K|}}\varphi(\alpha,g_2)d[\mu_{g_2}^\K-\mu_{g_1}^\K](\alpha)\Bigg|\\
&\le\|\varphi\|_{\Lip}\|g_1-g_2\|
+\BL(\mu_{g_2}^\K,\mu_{g_1}^\K).
\end{align*}
Hence, for $L\coloneqq\sup_{g_1\ne g_2}\BL(\mu^\K_{g_1},\mu^\K_{g_2})/\|g_1-g_2\|\in[0,\infty]$ we conclude that 
\begin{equation*}
\|\psi\|_{\BL(\R^m)}\le \|\varphi\|_{\BL(\Z)}+L\le 1+L.
\end{equation*}
Dividing $\psi$ by its bounded Lipschitz norm, we find
\begin{equation*}
\E|\psi(G_n)-\psi(G)|\le \|\psi\|_{\BL(\mathbb{R}^m)}\BL(\mathcal L(G_n),\mathcal L(G))\le (1+L)\BL(\mathcal L(G_n),\mathcal L(G)).
\end{equation*}
This completes the proof for the quantitative statement. Joint convergence still follows if $g\mapsto \mu_g^\K$ is only continuous $G$-almost surely (but not Lipschitz). In fact, $\psi$ is still continuous and bounded $G$-almost surely so that $\E\psi(G_n)\to\E\psi(G)$. Therefore, $\E\varphi(\alpha_n^\K,G_n)\to\varphi(\alpha^\K,G)$ for all $\varphi\in \BL(\Z)$, which implies that 
$(\alpha_n^\K,G_n)\to(\alpha^\K,G)$ in distribution.
\end{proof}

%% file: Sections/AppendixC.tex
\section{Optimal Transport}\label{Appendix:OT}

\begin{proof}[Corollary \ref{cor:uniquerealline}]\label{proof:cor:uniquerealline}
We prove that on the finite discrete space $\mathcal{X}\coloneqq\{x_1<\ldots<x_N\}$ the cost $c(x_i,x_j)=f(\vert x_i-x_j\vert)$ for $f$ strictly convex satisfies the strict Monge condition \eqref{eq:strictmonge}, i.e., we show that for $x_i<x_{i^\prime}$ and $x_j<x_{j^\prime}$ it holds that
\begin{equation*}
f(\vert x_i-x_j\vert)+f(\vert x_{i^\prime}-x_{j^\prime}\vert) < f(\vert x_i-x_{j^\prime}\vert)+ f(\vert x_{i^\prime}-x_j\vert).
\end{equation*}
By symmetry it suffices to consider the following cases:
\begin{itemize}
\item[(i)] $x_i<x_{i^\prime}\leq x_j<x_{j^\prime}$: Define $d_1\coloneqq \vert x_i-y_{j^\prime}\vert$, $d_2\coloneqq \vert x_i-y_j\vert$, $d_3\coloneqq \vert x_{i^\prime}-y_{j^\prime}\vert$ and $d_4\coloneqq \vert x_{i^\prime}-y_j\vert$. By definition we have that $d_1>d_4$ and $d_2+d_3=d_1+d_4$. Further, define $t_i\coloneqq \frac{d_1-d_4}{d_1-d_4}$ for $i=2,3$. By strict convexity it follows that 
\begin{equation*}
f(d_i)=f(t_id_1+(1-t_i)d_4)<t_if(d_1)+(1-t_i)f(d_4), \, i=2,3.
\end{equation*}
Adding both inequalities and the fact that $t_2+t_3=1$ yields that $f(d_2)+f(d_3)<f(d_1)+f(d_4)$ and hence the Monge condition.
\item[(ii)] $x_i\leq x_j < x_{i^\prime}\leq x_{j^\prime}$: For this case, define $d_1\coloneqq \vert x_i-x_j\vert$, $d_2\coloneqq \vert x_{i^\prime}-x_{j^\prime}\vert$, $d_3\coloneqq \vert x_i-x_{j^\prime}\vert$ and $d_4\coloneqq \vert x_j-x_{i^\prime}\vert$. Moreover, the Monge condition is trivally satisfied if either $x_i=x_j$ or $x_{i^\prime}=x_{j^\prime}$ since  $d_1+d_2<d_3$. Hence, suppose that $d_1,d_2>0$ and notice that by strict convexity of $f$ and the fact that $f(0)=0$ we deduce that $f$ is strictly super-additive, i.e., $f(x)+f(y)< f(x+y)$ for $x,y>0$. In particular, this implies that
\begin{equation*}
f(d_1)+f(d_2)< f(d_1+d_2)<f(d_3)<f(d_3)+f(d_4)
\end{equation*}
and hence the Monge condition follows.
\item[(iii)] $x_i< x_j \leq x_{i^\prime}< x_{j^\prime}$: If $x_j<x_{i^\prime}$ then case (ii) applies, else we have $x_j=x_{i^\prime}$ and the Monge condition is fulfilled by the strict super additive condition.
\item[(iv)] $x_i\leq x_j<x_{j^\prime}\leq x_{i^\prime}$: For this case, define $d_1\coloneqq \vert x_i-y_j\vert$, $d_2\coloneqq \vert x_{i^\prime}-x_{j^\prime}\vert$, $d_3\coloneqq \vert x_i-x_{j^\prime}\vert$ and $d_4\coloneqq \vert x_{i^\prime}-x_j \vert$. We notice that $d_1<d_3$ and $d_2<d_4$ which yields by strict monotonicity of $f$ that $f(d_1)+f(d_2)<f(d_3)+f(d_4)$ and hence the Monge condition.
\end{itemize}
For the concave case we argue as follows.  Let $\mu$ and $\nu$ be probability measures on $\mathbb R$ and let the cost function take the form $c(x,y)=f(|x-y|)$ where $f:\R_+\to\R_+$ is nondecreasing, $f(0)=0$ and $f$ is strictly concave and the optimal cost between $\mu$ and $\nu$ is finite (thus $f$ is stricly increasing and continuous). Firstly, according to \citet[Proposition 2.9]{gangbo1996geometry}, all the common mass must stay in place.  Hence, we may assume that $\mu$ and $\nu$ are mutually singular.
\begin{lem}
Let $\mu$ and $\nu$ be mutually singular and both supported on a finite union of intervals.  Then under the above conditions on $c$, the optimal transport plan between $\mu$ and $\nu$ is unique.
\end{lem}
\begin{rem}
If $\mu$ and $\nu$ have finite support, the assumption is satisfied.  We believe that the statement is true for an arbitrary pair of measures $\mu$ and $\nu$, but the above formulation is sufficient as in the context of the present $\mu$ and $\nu$ are anyway finitely supported.  For example, the support could contain countably many intervals as long as there is ``clear" starting point $a_0$ below; but $M$ could be infinite.
\end{rem}
For the proof, there is nothing to prove if $\mu=\nu=0$, so we assume $\mu\ne\nu$. It follows from the assumptions that there exists a finite sequence of $M+1\ge3$ real numbers
\begin{equation*}
-\infty\le a_0 < a_1 < a_2 < a_3 < .... < a_M\le \infty
\end{equation*}
such that (interchanging $\mu$ and $\nu$ if necessary)
\begin{align*}
\mu([a_0,a_1]\cup[a_2,a_3]\cup [a_4,a_5]\cup\dots)&=1;\\
\mu([a_1,a_2]\cup[a_3,a_4]\cup [a_5,a_6]\cup\dots)&=1.
\end{align*}
Let $m_0=\mu([a_0,a_1])$ and suppose that $m_0 \le \nu([a_1,a_2])$. Define the quantile
\begin{equation*}
a^*=\inf\{a : \nu[a_1,a]\ge m_0\}\in[a_1,a_2].
\end{equation*}
We now claim that in any optimal transport plan $\pi$ between $\mu$ and $\nu$, the $\mu$-mass of $[a_0,a_1]$ must go to $[a_1,a^*]$.  Indeed, suppose that a positive $\mu$-mass from $[a_0,a_1]$ goes strictly beyond $a^*$.  Then some mass from the support of $\mu$ but not in $[a_0,a_1]$ has to go to $[a_1,a^*]$.  Such a transport plan gives positive measure to the set
\begin{equation*}
[a_0,a_1]\times [a^*+\epsilon,\infty) \bigcap [a_2,\infty]\times [a_1,a^*]
\end{equation*}
for some $\epsilon>0$.
This contradicts cyclical monotonicity: indeed, if $\mu$ and $\nu$ are discrete measures, this entails sending mass from $x_1$ to $y_1$ and from $x_2$ to $y_2$ with $x_1<y_2<\min(x_2,y_1)$, which is clearly suboptimal.  For general measures (not necessarily finitely supported), see \citet[Theorem 2.3]{gangbo1996geometry}.  Hence the claim is proved.

Let $\mu_1$ be the restriction of $\mu$ to $[a_0,a_1]$ and $\nu_1$ be the restriction of $\nu$ to $[a_1,a^*]$ with mass $m_0$, namely $\nu_1(B)=\nu(B)$ if $B\subseteq[a_1,a^*)$, $\nu_1(\{a^*\}) = m_0 - \nu([a_1,a^*))$ and $\nu(B)=0$ if $B\cap[a_1,a^*]=\emptyset$.  By definition of $a^*$, $\nu_1$ is a measure (i.e., $\nu_1(\{a^*\})\ge0$) and $\nu_1$ and $\mu_1$ have the same total mass $m_0$.  Each of these measures is supported on an interval and these intervals are (almost) disjoint.  Cyclical monotonicity and strict concavity of the cost function entail that any optimal transport plan between $\mu_1$ and $\nu_1$ must be non-increasing (in a set-valued sense).  Since there is only one such plan, the transport plan is unique.

By the preceding paragraph and the above claim, we know that $\pi$ must be non-increasing from $[a_0,a_1]$ to $[a_1,a^*]$, which determines $\pi$ uniquely on that part.  After this transport is carried out, we are left with the measures $\nu-\nu_1$ and $\mu-\mu_1$, where the latter is supported on one less interval, namely the interval $[a_0,a_1]$ disappears.\\
If instead $\mu_0([a_0,a_1])>\nu([a_1,a_2])$, we can use the same construction with
\begin{equation*}
a^*=\inf\{a : \mu([a_0,a]\ge \nu([a_1,a_2])\}\in[a_0,a_1],
\end{equation*}
and the interval $[a_1,a_2]$ will disappear.  We then merge $[a^*,a_1]$ with $[a_2,a_3]$, that is
\begin{align*}
\mu-\mu_1 \textrm{ is supported on }&[a^*,a_3]\cup [a_4,a_5]\cup\dots,\\
\nu-\nu_1 \textrm{ is supported on }&[a_3,a_4]\cup [a_5,a_6]\cup\dots.
\end{align*}
If $\mu([a_0,a_1])=\nu([a_1,a_2])$ then both the intervals $[a_0,a_1]$ and $[a_1,a_2]$ disappear when considering $\mu-\mu_1$ and $\nu-\nu_1$.  In all three cases we can continue inductively and construct $\pi$ in a unique way.  Since there are finitely many intervals, the procedure is guaranteed to terminate.  Thus $\pi$ is unique.
\end{proof}

\begin{proof}[Theorem \ref{thm:uniquenessOT}]\label{proof:uniquenessOT}
For the if direction, recall the fact that even for more general probability measures the support of any optimal transport plan is indeed $c$-cyclically monotone \cite[Proposition 2.3]{gangbo1996geometry}. In fact, we can simply adapt their proof strategy for the discrete case considered here. For this, suppose $\pi^\star$ is a unique optimal transport coupling such that $\text{supp}(\pi^\star)$ is not strictly $c$-cyclically monotone. Then we can find $n\in\mathbb{N}$, $n\geq 2$ and a family $(i_1,j_1),\ldots,(i_n,j_n)$ with $(i_k,j_k)\in \text{supp}(\pi^\star)$ for all $1\leq k \leq n$ such that 
\begin{equation}\label{eq:cyclicproof}
\sum_{k=1}^n c_{i_kj_k}=\sum_{k=1}^n c_{i_kj_{k-1}}, \quad j_0 \coloneqq j_n.
\end{equation} 
Notice that by definition at least one tuple $(i_k,j_{k-1})$ for $1\leq k\leq n$ is not contained in the support of $\pi^\star$. We now create a different feasible plan with same optimal objective value. For this, set $\tau\coloneqq \min_{k\in [n]} \pi^\star_{i_kj_k}>0$ and define a new transport coupling $\tilde{\pi}$ by subtracting $\tau$ on $\pi^\star_{i_kj_k}$ and adding $\tau$ to $\pi^\star_{i_kj_{k-1}}$ for all $1\leq k\leq n$. Clearly, $\tilde{\pi}$ is feasible and different to $\pi^\star$. Moreover, by \eqref{eq:cyclicproof} it has the same overall cost as $\pi^\star$. Hence, $\pi^\star$ cannot be unique.

For the converse, we prove that strict $c$-cyclically monotonicity for $\text{supp}(\pi^\star)$ implies the existence of nondegenerate optimal dual solutions $(\alpha^\star,\beta^\star)$. By Proposition \ref{prop:slackness} this proves that $\pi^\star$ is unique. Suppose $\pi^\star$ is an optimal transport plan whose support is strictly $c$-cyclically monotone. Without loss of generality we assume that $\pi^\star$ is an optimal primal basic solution since by Lemma \ref{thm:optFeas} there always exists an optimal primal basic $\pi$ with support included in that $\pi^\star$, so $\pi$ will have strict $c$-cyclically monotone support. By complementary slackness we have for corresponding dual optimal solutions $(\alpha^\star,\beta^\star)$ that if $(i,j)$ belongs to a basic variable $\pi_{ij}$ then
\begin{equation*}
\alpha^\star_{i}+\beta^\star_{j}=c_{ij}.
\end{equation*}
Nondegeneracy for $(\alpha^\star,\beta^\star)$ means that for every tuple $(i,j)$ that does not belong to a basic variable we find that
\begin{equation*}
\alpha^\star_{i}+\beta^\star_{j}<c_{ij}.
\end{equation*}
To prove this strict inequality recall that the basic variables of $\pi^\star$ induce a tree (bipartite graph) between the support of the measure $r$ and the measure $s$ \cite[Section 3.4]{peyre2019computational}. Hence, for every index tuple $(i,j)$ that does not belong to a basic variable we can find a unique path 
\begin{align*}
(i=i_1,j_1,\ldots,j_{n-1},i_{n-1},j_n=j)
\end{align*}
on the induced tree such that each tuple $(i_k,j_k)$ for all $1\leq k\leq n$ belongs to $\text{supp}(\pi^\star)$ and $(i_k,j_{k-1})$ for all $1\leq k \leq n$ with $j_0\coloneqq j_n$ belong to basic variables in $\pi^\star$. In particular, by strong duality we have that $c_{i_kj_k}=\alpha^\star_{i_k}+\beta^\star_{j_k}$ and $c_{i_k j_{k-1}}=\alpha^\star_{i_k}+\beta^\star_{j_{k-1}}$. Suppose that also $c_{i_1j_n}=\alpha^\star_{i_1}+\beta^\star_{j_n}$, i.e., $(\alpha^\star,\beta^\star)$ are degenerate. Then the index set 
\begin{equation*}
\Gamma\coloneqq \left\lbrace (i_k,j_k)\,\mid\, 1\leq k\leq n \right\rbrace\subseteq \text{supp}(\pi^\star)
\end{equation*}
fulfils
\begin{equation*}
\sum_{k=1}^n c_{i_kj_k}=\sum_{k=1}^n c_{i_k j_{k-1}}, \quad j_0\coloneqq j_n
\end{equation*}
which contradicts the strict $c$-cyclically monotonicity. Consequently, the optimal dual variables $(\alpha^\star,\beta^\star)$ are nondegenerate and hence the optimal transport coupling $\pi^\star$ is unique.
\end{proof}

\begin{proof}[Theorem \ref{thm:dualnondegenerateOT}]\label{proof:thm:dualnondegenerateOT}
Let $I_k$ be a dual feasible basis inducing a dual feasible solution $(\alpha,\beta)$. Every such basis induces a graph $G(I_k)$ in the sense that if $(i,j)\in I_k$ then the $i$-th support point of the measure $r$ is connected to the $j$-th support point of measure $s$, i.e., $(i,j)\in G(I_k)$. By definition of dual feasible basis it holds that $\alpha_i+\beta_j=c_{ij}$. In fact, such a basis induces a tree structure between all support points of $r$ and all support points of $s$ \cite[Section 3.4]{peyre2019computational}.

In order to exclude that $\lambda(I_k)\neq \lambda(I_l)$ for $k\neq l$ we proceed as follows. Let $G(I_k)$ be the tree induced by basis $I_k$. For $k\neq l$ we clearly have that $G(I_k)\neq G(I_l)$ and consequently there exists at least one edge $(i,j)$ in $G(I_l)$ such that $(i,j)\not\in G(I_k)$. By definition if $(\tilde{\alpha},\tilde{\beta})$ are the feasible dual solutions induced by $I_l$ then we have $\tilde{\alpha}_i+\tilde{\beta}_j=c_{i,j}$. We now need to prove that $\alpha_i+\beta_j\neq c_{i,j}$ and hence $\lambda(I_k)=(\alpha,\beta)\neq (\tilde{\alpha},\tilde{\beta})=\lambda(I_l)$. To see this, notice that adding edge $(i,j)$ to $G(I_k)$ creates a cycle. In particular, after proper relabelling there exists a path of the form $$(i= i_1,j_1,i_2,j_2,\ldots,i_n,j_n= j)$$ such that $(i_l,j_l)\in G(I_k)$ as well as $(i_{l+1},j_{l})\in G(I_k)$ for all $1\leq l \leq n-1$. Recall further that by definition if edge $(i_k,j_k)\in G(I_k)$ then $\alpha_{i_k}+\beta_{j_k}=c_{i_k,j_k}$. Suppose now that $\alpha_i+\beta_j= c_{i,j}$. Then similar as in the proof of Theorem \ref{thm:uniquenessOT} the set $\left\lbrace (i_k,j_k)\right\rbrace_{1\leq k\leq n}$ contradicts the summability assumption \eqref{eq:summabilitydual}. Consequently, $(\alpha,\beta)\neq (\tilde{\alpha},\tilde{\beta})$ and further as $I_k$ and $I_l$ are arbitrarily chosen, we have that assumption \eqref{ass:dualnondegen} holds. 
\end{proof}

\begin{proof}[Proposition \ref{prop:uniquenessOT}]\label{proof:propuniquenessOT}
According to Theorem \ref{thm:dualnondegenerateOT} all dual feasible basic solutions for \eqref{eq:DOT} are nondegenerate if there exists no family of indices $\left\lbrace (i_k,j_k)\right\rbrace$ for $n\geq 2$ with all $i_k$ pairwise different and all $j_k$ pairwise different such that 
\begin{equation}\label{eq:proofsummability}
\sum_{k=1}^n \Vert X_{i_k}-Y_{j_k}\Vert_q^p=\sum_{k=1}^n \Vert X_{i_k}-Y_{j_{k-1}}\Vert_q^p, \quad Y_{j_0}\coloneqq Y_{j_{n}}\, .
\end{equation}
Notice, that this condition only involves finitely many families where equalities of the form \eqref{eq:proofsummability} have to be checked. Further, the union of finitely many null-sets is again a null-set. Hence, it suffices to prove that \eqref{eq:proofsummability} holds with probability zero for fixed $n$. For the sake of notational simplicity, we choose the first $n\leq N$ random locations $(\mathbf{X},\mathbf{Y})\coloneqq (X_{1},\ldots,X_{n},Y_{1},\ldots,Y_{n})$. We denote by $(\mathbf{x},\mathbf{y})=(x_1,\ldots,x_n,y_1,\ldots,y_n)\in \left( \mathbb{R}^D\right)^{2n}$ and define the set
\begin{equation*}
\begin{split}
A\coloneqq \left\lbrace  (\mathbf{x},\mathbf{y}) \in \left( \mathbb{R}^D\right)^{2n}\, \vert \, \sum_{k=1}^n \Vert x_k-y_{k-1}\Vert_q^p-\Vert x_k-y_k\Vert_q^p=0  \right\rbrace\, .
\end{split}
\end{equation*}
Then $\PR ((\mathbf{X},\mathbf{Y})\in A)$ is the probability that $(\mathbf{X},\mathbf{Y})$ fulfils \eqref{eq:proofsummability} and the goal is to show that this equals zero. Set $e_i\in\mathbb{R}^D$ to be the $i$th unit vector and consider the closed set
\begin{equation*}
\begin{split}
B\coloneqq \{  (\mathbf{x},\mathbf{y}) \in \left( \mathbb{R}^D\right)^{2n}\, \vert \, \langle x_k,e_i\rangle \in \{\langle y_{k-1},e_i\rangle,\langle y_k,e_i\rangle \},\, \forall 1\leq i\leq D,\, 1\leq k \leq n, \, y_0\coloneqq y_n \}\, .
\end{split}
\end{equation*}
Define the function $f\colon (\mathbb{R}^D)^{2n}\setminus B \to \mathbb{R}$ with $f(\mathbf{x},\mathbf{y})=\sum_{k=1}^n \Vert x_k-y_{k-1}\Vert_q^p-\Vert x_k-y_k\Vert_q^p$. We can rewrite 
\begin{equation*}
\PR \left((\mathbf{X},\mathbf{Y})\in A\right)\leq \PR \left((\mathbf{X},\mathbf{Y}) \in f^{-1}(0)\right) + \PR \left((\mathbf{X},\mathbf{Y})\in B\right)\, .
\end{equation*}
The second term on the right-hand side is zero since by independence and absolute continuity the high-dimensional vector $(\mathbf{X},\mathbf{Y})$ has a Lebesgue density and the set $B$ lives in dimension less than $2Dn$. It remains to discuss $\PR \left((\mathbf{X},\mathbf{Y}) \in f^{-1}(0)\right)$. The open set $(\mathbb{R}^D)^{2n}\setminus B$ on which $f$ is defined can be partitioned into finitely many\footnote{less than $6^{nD}$} open connectedness components $U_1,\ldots,U_L$ according to the signs of $\langle x_k-y_k,e_i\rangle$ and $\langle x_k-y_{k-1},e_i\rangle$. On each such component $f_{\vert U_i}$ is analytic. Further, $f_{\vert U_l}$ is not identically zero function on $U_l$. For this, consider for any point $(\mathbf{x},\mathbf{y})\in U_l$ and $\epsilon\in\mathbb{R}$ the function 
\begin{equation*}
f_{\vert U_l}(\epsilon)=f_{\vert U_l}(x_1+\epsilon e_i,x_2,\ldots,x_n,y_1,\ldots,y_n)
\end{equation*}
with derivative at $\epsilon=0$ given by
\begin{equation}\label{eq:derf}
\frac{\partial f_{\vert U_l}}{\partial \epsilon}_{\vert \epsilon =0}=p\left(\Vert x_1-y_1\Vert_q^{p-q} \frac{\vert x_{1_i}-y_{1_i}\vert^q}{(x_{1_i}-y_{1_i})}-\Vert x_1-y_n\Vert_q^{p-q}\frac{\vert x_{1_i}-y_{n_i}\vert^q}{(x_{1_i}-y_{n_i})} \right)\, ,
\end{equation}
where $x_{i_j}$ denotes the $j$th entry of the $i$th vector. If this derivative is nonzero, then clearly $f$ is not identically zero.  If the derivative is zero then we shall show that there exists another point in $U_l$ for which this derivative is nonzero.  Since $U_l$ is open, we can add $\delta e_j$ to $y_n$ for small $\delta$ and any $1\le j\le D$.  If $p\ne q$ then, taking $j\ne i$ (which is possible because $D\ge2$) only modifies the term $\|x_1-y_n\|$ in \eqref{eq:derf}, and for small $\delta$ the derivative will not be zero.  If $p=q\ne 1$ then the norms do not appear in \eqref{eq:derf} and taking $j=i$ would yield a nonzero derivative.  Hence, if $p$ and $q$ are not both equal to one
, the function $f$ is not identically zero on each piece $U_l$. Finally, we deduce by \citet[Lemma 1.2]{dang2015complex} that $\PR\left((\mathbf{X},\mathbf{Y})\in f^{-1}_{\vert U_l}(\{0\})\right)=0$ for any $1\leq l\leq L$. An application of the union bound finishes the proof of the main statement since 
\begin{equation*}
\begin{split}
\PR \left((\mathbf{X},\mathbf{Y})\in A\right)&\leq \PR \left((\mathbf{X},\mathbf{Y}) \in f^{-1}(0)\right) + \PR \left((\mathbf{X},\mathbf{Y})\in B\right)\\
&\leq \sum_{l=1}^L \PR\left( (\mathbf{X},\mathbf{Y})\in f^{-1}_{\vert U_l}(\{0\}))\right)+ \PR \left((\mathbf{X},\mathbf{Y})\in B\right)=0\, .
\end{split}
\end{equation*}
The argument only depends on the position of the random support points of the probability measures $r=\sum_{k=1}^n r_k\delta_{X_k}$ and $s=\sum_{k=1}^n s_k\delta_{Y_k}$ and hence is uniform in their probability weights. Recall further Proposition \ref{prop:slackness} that if the dual problem admits a nondegenerate optimal solution the primal optimal solution is unique. We conclude that almost surely the optimal transport coupling is unique.
\end{proof}

\begin{rem}\label{rem:uniquenessOT}
The proof remains correct if both measures $r,s$ are based on the same random locations $X_1,\ldots,X_N\overset{\text{i.i.d.}}{\sim}\mu$ with $\mu$ absolutely continuous with respect to Lebesgue measure, i.e., $r=\sum_{k=1}^n r_k\delta_{X_k}$, $s=\sum_{k=1}^n s_k\delta_{X_k}$ and cost $c(X_i,X_k)=\Vert X_i-X_j\Vert^p_q$. Notice that this only involves a different definition of the function $f$ in the proof of Proposition \ref{prop:uniquenessOT} now based on all $X_k$ involved. Still, on a proper open subset $f$ remains analytic and is not equal the zero function. The rest of the arguments are then analogously.
\end{rem}

\begin{proof}[Proposition \ref{prop:dualnondegenerateOT}]\label{proof:dualnondegenerateOT}
The summability condition \eqref{eq:summabilityprimal} implies all primal feasible basic solutions for OT to be nondegenerate \cite[Corollary 3]{klee1968facets}. In particular, any optimal basic solution is nondegenerate. Now, if either of the two conditions in Proposition \ref{prop:dualnondegenerateOT} hold then the transport coupling is unique and given by one of the basic solutions. This implies the underlying basis to be unique and therefore assumption \eqref{ass:dualnondegen} is trivially fulfilled.
\end{proof}

\begin{lem}\label{lem:boundaryzero}
Consider the optimal transport problem \eqref{eq:OTprimal} between two probability measures $r,s\in\text{ri}(\Delta_N)$. Recall the definition of the closed convex cone from Section \ref{subsubsec:convexcones},
\begin{equation*}
H\coloneqq \bigcap_{j\in J} \left\lbrace v\in \mathbb{R}^{2N-1} \, \vert \, \left[{A_{\dagger\, I}^{-1}}v\right]_j\geq 0\right\rbrace,
\end{equation*}
where $I$ is a feasible dual basis and  $J$ is the set of indices corresponding to degenerate zeroes in the optimal transport coupling $\pi^\star$ induced by $I$. Then it holds that 
\begin{equation*}
\PR(G\in \partial H_k)=0,
\end{equation*}
where $G=(G^1(r_\dagger),G^2(s))$ is a centred Gaussian distribution on $\mathbb{R}^{2N-1}$ with block diagonal covariance matrix, where the blocks are given in \eqref{eq:covmatrix}.
\end{lem}

\begin{proof}
We first require a result essentially relying on an observation in \citet[Corollary 8.1.4]{brualdi2006combinatorial} characterizing basic variables in optimal transport couplings. In fact, this can be used to obtain a proper subset of the boundaries for our generic hyperplanes
\begin{equation*}
H\coloneqq \bigcap_{j\in J} \left\lbrace v\in \mathbb{R}^{2N-1} \, \vert \, \left[{A_{\dagger\, I}^{-1}}v\right]_j\geq 0\right\rbrace,
\end{equation*}
where $I$ is a dual feasible basis and $J$ is the set of indices corresponding to degenerate zeroes in the optimal transport coupling $\pi^\star$.

\begin{lem}\label{lem:boundary}
Consider the optimal transport problem \eqref{eq:OTprimal} between two probability measures $r,s\in\text{ri}(\Delta_N)$. Then
\begin{equation}
\partial H\subseteq \bigcup_{i=1}^{2N-1} \left\lbrace v\in\mathbb{R}^{2N-1}\,\vert\,\pm\left\lbrace\sum_{j\in \mathsf{R}_i} v_j - \sum_{k\in \mathsf{S}_i}v_k\right\rbrace\in \{0,1\}\right\rbrace,
\end{equation}
where $\mathsf{R}_i\subseteq \{1,\ldots,N-1\}$ and $\mathsf{S}_i\subseteq \{N,\ldots,2N-1\}$ are not both equal to the empty set.
\end{lem}

\begin{proof}
Consider first the optimal transport problem with feasible set $\Pi(r,s)$, i.e., we do not delete the last entry of $r$. According to \citet[Corollary 8.1.4]{brualdi2006combinatorial} the nonnegative entries of any extreme point $\pi$ for the polytope $\Pi(r,s)$ are of the form 
\begin{equation*}
\pi_i=\pm\left\lbrace \sum_{j\in \mathsf{R}_i}r_j-\sum_{k\in \mathsf{S}_i}s_k\right\rbrace
\end{equation*}
for some subsets $\mathsf{R}_i\subseteq\{1,\ldots,N\}$ and $\mathsf{S}_i\subseteq\{N+1,\ldots,2N\}$, where not both sets are simultaneously equal to the empty set. Notice that the set of all extreme points for the polytope $\Pi(r,s)$ is equal to $\text{Aug}(A_{I_k}^{-1}[r,s]^T)$ for all primal feasible bases $I_1,\ldots,I_k$. We conclude that for each primal feasible basis $I_k$ and each coordinate $i\in I_k$ there exists subsets $\mathsf{R}_i\subseteq \{1,\ldots,N\}$ and $\mathsf{S}_i\subseteq \{N+1,\ldots,2N\}$ not both equal to the empty set such that 
\begin{equation}\label{eq:proofsum}
\left[A_{I_k}^{-1}[r,s]^T\right]_i= \pm\left\lbrace\sum_{j\in \mathsf{R}_i} r_j - \sum_{k\in \mathsf{S}_i}s_k\right\rbrace.
\end{equation}
Recall the notion of $r_\dagger$, where we delete the last entry of the probability measure $r$ in order to guarantee full rank of $A_\dagger$. In view of \eqref{eq:proofsum} a similar statement is true if we only consider right-hand side vector $[r_\dagger,s]^T$ that, however, requires a careful case distinction. In fact, \eqref{eq:proofsum} remains true for $[{A_{\dagger\, I_k}^{-1}}[r_\dagger,s]^T]_i$ if $\mathsf{R}_i$ does not contain the index $N$. Moreover, if $N\in \mathsf{R}_i$ then we can replace $r_N=1-\sum_{i\in\{1,\ldots,N-1\}}r_i$. Notice that this only depends on the first $1,\ldots,N-1$ coordinates. In total, we obtain the following modified version of \eqref{eq:proofsum}: For each primal feasible basis $I_k$ and $i\in I_k$ it holds that
\begin{equation}\label{eq:proofsummodi}
\left[{A_{\dagger\, I_k}^{-1}}[r_\dagger,s]^T\right]_i=
\begin{cases}
\pm\left\lbrace\sum_{j\in {\mathsf{R}_\dagger}_i} r_j - \sum_{k\in \mathsf{S}_i}s_k\right\rbrace, \text{ for } N\notin \mathsf{R},\\
\pm\left\lbrace 1-\sum_{j\in {\mathsf{R}_\dagger}_i} r_j - \sum_{k\in \mathsf{S}_i}s_k\right\rbrace, \text{ for } N\in \mathsf{R}.
\end{cases}
\end{equation}
Now, recall that for a feasible dual basis $I_k$ we define the set 
\begin{equation*}
H_k\coloneqq \bigcap_{j\in J_k\subseteq I_k} \left\lbrace v\in \mathbb{R}^{2N-1} \, \vert \, \left[{A_{\dagger\, I_k}^{-1}}v\right]_j\geq 0\right\rbrace,
\end{equation*}
where $J_k$ is the set of indices corresponding to degenerate zeroes in the optimal transport coupling $\pi^\star$. Together with our modified version \eqref{eq:proofsummodi} we conclude the statement by the chain of inclusions
\begin{equation*}
\begin{split}
\partial H_k \subseteq \bigcup_{j \in J_k} \partial \left\lbrace v\in \mathbb{R}^{2N-1}\, \vert \, \left[{A_{\dagger\, I_k}^{-1}}v\right]_j\geq 0 \right\rbrace
&=\bigcup_{j \in J_k} \left\lbrace v\in \mathbb{R}^{2N-1}\, \vert \, \left[{A_{\dagger\, I_k}^{-1}}v\right]_j= 0 \right\rbrace\\
&\subseteq \bigcup_{i=1}^{2N-1} \left\lbrace v\in\mathbb{R}^{2N-1}\,\vert\,\pm\left\lbrace\sum_{j\in \mathsf{R}_i} v_j - \sum_{k\in \mathsf{S}_i}v_k\right\rbrace\in \{0,1\}\right\rbrace.
\end{split}
\end{equation*}
\end{proof}

Continuing the proof of Lemma \ref{lem:boundaryzero}, according to Lemma \ref{lem:boundary} and the union bound we have that 
\begin{equation*}
\PR(G\in \partial H_k)\leq \sum_{i=1}^{2N-1} \PR\left(\pm\left( \sum_{j\in\mathsf{R}_i} G^1(r_\dagger)_j-\sum_{k\in\mathsf{S}_i} G^2(s)_k\right)\in\{0,1\}\right)
\end{equation*}
for pairs of proper subsets $(\mathsf{R}_i,\mathsf{S}_i)\subseteq \{1,\ldots,N-1\}\times\{N,\ldots,2N-1\}$, $i=1,\ldots,2N-1$ not both equal to the empty set. Recall that $G^1(r_\dagger)$ is independent of $G^2(s)$ and admits a density on $\R^{N-1}$. Hence, each coordinate $G^1(r_\dagger)_j$ has a density. In particular, if the set $\mathsf{R}_i$ is non empty then the probability that the random entries $G^1(r_\dagger)$ fulfil either one of finitely many equality constraints is zero. More precisely,  the event on the right-hand side of the last display has probability zero. 
If $\mathsf{R}_i$ is empty, then $\mathsf{S}_i$ contains at least one element, but not all of them as it is proper. If $s>0$, then the only eigenvector in the kernel of $\Sigma(s)$ is a vector of ones. Hence the distribution of $\sum_{k\in\mathsf{S}_i} G^2(s)_k$ is absolutely continuous, and therefore almost surely this random variable is not in $\{-1,0,1\}$. This completes the proof.
If $\mathsf{R}_i$ is the empty set, notice that $\mathsf{S}_i$ contains at least one element and hence the sum $\sum_{k\in\mathsf{S}_i} G^2(s)_k\neq 0$. Furthermore, the sum $\sum_{k\in\mathsf{S}_i} G^2(s)_k$ can also never attain the value $\pm 1$ as this would lead to an infeasible coupling (a matrix of dim $N\times N$ with $N\geq 2$ containing only zeroes except that one coordinate is equal to one). This finishes the proof.
\end{proof}

%% file: Sections/AppendixD.tex
\section{Further Illustrations}\label{Appendix:FurtherIllustration}
\textbf{The Limit Law.} Consider Example \ref{exm:OT}, where we assume both probability vectors $r$, $s$ to be equal and strictly positive. We here allow for general cost exponent $p\in (0,\infty)$. In any case, the corresponding optimal solution is unique and supported on the diagonal, i.e., all the mass remains at its current location. In particular, the optimal solution is degenerate. For our case of 9 variables and 5 constraints, there are at most $\binom 95=126$ candidates for bases $I$. In this one-dimensional optimal transport problem, however, only 81 are such that the corresponding matrix $A_{\dagger I}$ is invertible. Moreover and since the optimal solution is supported on the diagonal, each primal optimal basis necessarily includes the diagonal. In terms of the bases set that means that for each primal and dual optimal basis $\{1,5,9\}\subset I$, and only such bases could appear in the limit law. There are at most twelve such bases, corresponding to the transport schemes

\begin{small}
\begin{tabular}{cccc}
$TS(I_1)=\begin{pmatrix}
\ast & \ast \\
&\ast&\\
&\ast&\ast
\end{pmatrix}$ & $TS(I_2)=\begin{pmatrix}
\ast &  \\
\ast&\ast&\ast\\
&&\ast
\end{pmatrix}$ & $TS(I_3)=\begin{pmatrix}
\ast & & \\
\ast & \ast& \\
\ast & &\ast
\end{pmatrix}$ & $TS(I_4)=\begin{pmatrix}
\ast & &\ast \\
 & \ast&\ast \\
 & &\ast
\end{pmatrix}$ \\ 
$TS(I_5)=\begin{pmatrix}
\ast & \ast & \ast \\
&\ast&\\
&&\ast
\end{pmatrix}$ & $TS(I_6)=\begin{pmatrix}
\ast &&  \\
&\ast&\\
\ast&\ast&\ast
\end{pmatrix}$ & $TS(I_7)=\begin{pmatrix}
\ast &  \\
\ast&\ast&\\
&\ast&\ast
\end{pmatrix}$ & $TS(I_8)=\begin{pmatrix}
\ast & \ast \\
&\ast&\ast\\
&&\ast
\end{pmatrix}$ \\
$TS(I_{9})=\begin{pmatrix}
\ast & &\ast \\
& \ast &\\
&\ast&\ast
\end{pmatrix}$ & $TS(I_{10})=\begin{pmatrix}
\ast && \ast \\
\ast&\ast\\
&&\ast
\end{pmatrix}$ & $TS(I_{11})
=\begin{pmatrix}
\ast & & \\
 &\ast &\ast \\
\ast & &\ast
\end{pmatrix}$ & $TS(I_{12})
=\begin{pmatrix}
\ast &\ast & \\
 &\ast & \\
\ast & &\ast
\end{pmatrix}$,
\end{tabular} 
\end{small}

\noindent
all of which are optimal for $r=s$. However, only some of these bases also induce dual optimal basic solutions. This depends on the cost vector $c$ through the parameter $p>0$. The following table illustrates this dependence. Recall that $K$ is the number of primal and dual feasible bases.

\begin{table}
\renewcommand{\arraystretch}{1.5}
\begin{center}
\begin{tabular}{cccc}
Cost exponent $p$ & primal \& dual optimal bases & $K$ & dual nondegeneracy \eqref{ass:dualnondegen}\\
\hline
$(0,1)$ & $I_1$, $I_2$, $I_3$, $I_4$, $I_5$, $I_6$ & 6 & $\surd$\\
1 & $I_1$, $I_2$, $I_3$, $I_4$, $I_5$, $I_6$, $I_7$, $I_8$ & 8 & $X$\\
$(1,\infty)$ & $I_1$, $I_2$, $I_7$, $I_8$ & 4 & $\surd$\\
\end{tabular}
\caption{Primal and dual optimal bases for optimal transport between three ordered points on the real line. Cost given by $\vert x-y\vert^p$ depending on $p$.}
\end{center}
\end{table}
To each basis $I$ we associate a cone by \eqref{eq:Hk} that have already been computed in Example \ref{exm:OT}. Notice that the failure of \eqref{ass:dualnondegen} for $p=1$ causes the limit law to be of a more complicated nature as some of these cones have nontrivial intersections. For $p\neq 1$, this is not the case and given the marginal limit law from \eqref{eq:empprocess2} we deduce according to Theorem \ref{thm:limitforOT} that, e.g., for $p\in (1,\infty)$ the limit law for the transport coupling reads as 
\begin{small}
\begin{equation*}
\begin{split}
M(\mathbf{G})=\sum_{k\in\{1,2,7,8\}}\mathbbm{1}_{\mathbf{G}\in H_{k}}\pi(I_k,\mathbf{G})=
&\mathbbm{1}_{\left\lbrace\substack{\mathbf{G_1}\geq \mathbf{G_3}\\ \mathbf{G_1}+\mathbf{G_2}\leq\mathbf{G_3}+\mathbf{G_4} }\right\rbrace}\pi(I_1,\mathbf{G})+
\mathbbm{1}_{\left\lbrace\substack{\mathbf{G_1}\leq \mathbf{G_3}\\ \mathbf{G_1}+\mathbf{G_2}\geq \mathbf{G_3}+\mathbf{G_4}}\right\rbrace}\pi(I_2,\mathbf{G})\\
&+\mathbbm{1}_{\left\lbrace\substack{\mathbf{G_1}\leq \mathbf{G_3}\\  \mathbf{G_1}+\mathbf{G_2}\leq \mathbf{G_3}+\mathbf{G_4}}\right\rbrace}\pi(I_7,\mathbf{G})+
\mathbbm{1}_{\left\lbrace\substack{\mathbf{G}_1\geq \mathbf{G_3}\\  \mathbf{G_1}+\mathbf{G_2}\geq \mathbf{G_3}+\mathbf{G_4}}\right\rbrace}\pi(I_8,\mathbf{G}),
\end{split}
\end{equation*}
\end{small}
and similarly for $p\in(0,1)$.\\
\textbf{The Hausdorff Distance.} Suppose now that $p=1$ but the probability vectors are $r=\left(\nicefrac{1}{4},\nicefrac{1}{4},\nicefrac{1}{2}\right)$ and $s=\left(\nicefrac{1}{2},\nicefrac{1}{4},\nicefrac{1}{4}\right)$. In this case, there are four primal and dual optimal transport schemes $I_1,I_2,I_3,I_4$, namely

\begin{small}
\begin{center}
\begin{tabular}{cccc}
$TS(I_1)=\begin{pmatrix}
\ast & & \\
\ast& &\\
\ast&\ast&\ast
\end{pmatrix}$, & $TS(I_2)=\begin{pmatrix}
\ast & & \\
\ast&\ast&\\
&\ast&\ast
\end{pmatrix}$, &
$TS(I_3)=\begin{pmatrix}
\ast & & \\
&\ast&\\
\ast&\ast&\ast
\end{pmatrix}$, & $TS(I_4)=\begin{pmatrix}
\ast &  \\
\ast&\ast&\\
\ast& &\ast
\end{pmatrix},$ \\ 
\end{tabular} 
\end{center}
\end{small}
that induce the following two primal optimal basic solutions
\begin{small}
\begin{equation*}
\begin{split}
\pi(I_1,(r_\dagger,s))=\pi(I_2,(r_\dagger,s))=\begin{pmatrix}
\nicefrac{1}{4} & & \\
\nicefrac{1}{4}&0 &\\
0&\nicefrac{1}{4}&\nicefrac{1}{4}
\end{pmatrix},\quad 
\pi(I_3,(r_\dagger,s))=\pi(I_4,(r_\dagger,s))=\begin{pmatrix}
\nicefrac{1}{4} &  \\
0&\nicefrac{1}{4}&\\
\nicefrac{1}{4}&0 &\nicefrac{1}{4}
\end{pmatrix},
\end{split}
\end{equation*}
\end{small}
respectively. Notice that each convex combination of both primal optimal basic solutions is also primal optimal. According to our equivalence classes as defined in the proof of Theorem \ref{thm:Hausdorff} in Proof \ref{proof:Hausdorff}, we find that $\mathcal{B}_1=\{1,2\}$ and $\mathcal{B}_2=\{3,4\}$. Suppose for $\varepsilon>0$ sufficiently small we perturb the probability vector $r$ to obtain $\tilde{r}=\left(\nicefrac{1}{4}+\varepsilon,\nicefrac{1}{4}-\varepsilon,\nicefrac{1}{2}\right)$. By the nonnegativity constraint the transport scheme $I_4$ becomes infeasible
\begin{small}
\begin{equation*}
\pi(I_4,(\tilde{r}_\dagger,s))=\begin{pmatrix}
\nicefrac{1}{4}+\varepsilon &  \\
-\varepsilon & \nicefrac{1}{4} &\\
\nicefrac{1}{4}&0 &\nicefrac{1}{4}
\end{pmatrix}.
\end{equation*}
\end{small}
However, for $\varepsilon\in (0,\nicefrac{1}{4})$ the transport schemes $I_1,I_2,I_3$ remain feasible
\begin{small}
\begin{equation*}
\begin{split}
\pi(I_1,(\tilde{r}_\dagger,s))=\pi(I_2,(\tilde{r}_\dagger,s))=\begin{pmatrix}
\nicefrac{1}{4}+\varepsilon & & \\
\nicefrac{1}{4}-\varepsilon &0 &\\
0&\nicefrac{1}{4}&\nicefrac{1}{4}
\end{pmatrix},\quad
\pi(I_3,(\tilde{r}_\dagger,s))=\begin{pmatrix}
\nicefrac{1}{4}+\varepsilon &  \\
0&\nicefrac{1}{4}-\varepsilon&\\
\nicefrac{1}{4}-\varepsilon& \varepsilon &\nicefrac{1}{4}
\end{pmatrix}
\end{split}
\end{equation*}
\end{small}
and hence optimal. As expected from Lemma \ref{lem:basisequivalenceclasses}, for small perturbations at least one transport scheme from each equivalence class $\mathcal{B}_1$, $\mathcal{B}_2$ remains optimal for the perturbed problem. In addition to optimality of $I_1$, $I_2$ and $I_3$, it is clear that
\begin{equation*}
\Vert \pi(I_j,(r_\dagger,s))-\pi(I_j,(\tilde{r}_\dagger,s))\Vert=O(\varepsilon)
\end{equation*}
for $j=1,2,3$. Although for the perturbed problem $I_4$ is no longer feasible, we still have $\Vert \pi(I_4,(r_\dagger,s))-\pi(I_3,(\tilde{r}_\dagger,s))\Vert=O(\varepsilon)$. In particular, we conclude for the Hausdorff distance $d_H\left(Opt\left(r_\dagger,s\right),Opt\left(\tilde{r}_\dagger,s\right)\right)=O(\varepsilon)$.
For small perturbations with $\epsilon<0$, bases $I_1$, $I_2$ and $I_4$ are optimal and $I_3$ becomes infeasible.  In a similar fashion to the aforementioned case, it holds that $\Vert \pi(I_3,(r_\dagger,s))-\pi(I_4,(\tilde{r}_\dagger,s))\Vert=O(\varepsilon)$ and the Hausdorff distance between the optimality sets is $O(\varepsilon)$. Thus the Hausdorff distance is of the same magnitude as the perturbations, which is $O_\mathbb P(r_n^{-1})$ in view of \eqref{ass:cltforb}.

%% file: main.bbl
\begin{thebibliography}{63}
\providecommand{\natexlab}[1]{#1}
\providecommand{\url}[1]{\texttt{#1}}
\expandafter\ifx\csname urlstyle\endcsname\relax
  \providecommand{\doi}[1]{doi: #1}\else
  \providecommand{\doi}{doi: \begingroup \urlstyle{rm}\Url}\fi

\bibitem[Arjovsky et~al.(2017)Arjovsky, Chintalah, and
  Bottou]{arjovsky2017wasserstein}
M~Arjovsky, S~Chintalah, and L~Bottou.
\newblock Wasserstein generative adversarial networks.
\newblock \emph{Proceedings of Machine Learning Research}, 70:\penalty0
  214--223, 2017.

\bibitem[Beale(1955)]{beale1955minimizing}
Evelyn~ML Beale.
\newblock On minimizing a convex function subject to linear inequalities.
\newblock \emph{Journal of the Royal Statistical Society: Series B
  (Methodological)}, 17\penalty0 (2):\penalty0 173--184, 1955.

\bibitem[Bertsimas and Tsitsiklis(1997)]{bertsimas1997introduction}
Dimitris Bertsimas and John~N Tsitsiklis.
\newblock \emph{Introduction to {L}inear {O}ptimization}, volume~6.
\newblock Athena Scientific Belmont, MA, 1997.

\bibitem[Billingsley(1999)]{billingsley1999convergence}
Patrick Billingsley.
\newblock \emph{Convergence of {P}robability {M}easures}.
\newblock Wiley, New York, 2nd edition, 1999.

\bibitem[Bland(1977)]{bland1977new}
Robert~G Bland.
\newblock New finite pivoting rules for the simplex method.
\newblock \emph{Mathematics of Operations Research}, 2\penalty0 (2):\penalty0
  103--107, 1977.

\bibitem[Boyd and Vandenberghe(2004)]{boyd2004convex}
Stephen Boyd and Lieven Vandenberghe.
\newblock \emph{Convex optimization}.
\newblock Cambridge university press, 2004.

\bibitem[Brualdi(2006)]{brualdi2006combinatorial}
Richard~A Brualdi.
\newblock \emph{Combinatorial {M}atrix {C}lasses}, volume~13.
\newblock Cambridge University Press, 2006.

\bibitem[Chang and Pollard(1997)]{chang1997conditioning}
Joseph~T Chang and David Pollard.
\newblock Conditioning as disintegration.
\newblock \emph{Statistica Neerlandica}, 51\penalty0 (3):\penalty0 287--317,
  1997.

\bibitem[Chernozhukov et~al.(2017)Chernozhukov, Galichon, Hallin, and
  Henry]{chernozhukov2017monge}
Victor Chernozhukov, Alfred Galichon, Marc Hallin, and Marc Henry.
\newblock Monge--{K}antorovich depth, quantiles, ranks and signs.
\newblock \emph{Ann. Stat.}, 45\penalty0 (1):\penalty0 223--256, 2017.

\bibitem[Cotar et~al.(2013)Cotar, Friesecke, and
  Kl{\"u}ppelberg]{cotar2013density}
Codina Cotar, Gero Friesecke, and Claudia Kl{\"u}ppelberg.
\newblock Density functional theory and optimal transportation with {C}oulomb
  cost.
\newblock \emph{Communications on Pure and Applied Mathematics}, 66\penalty0
  (4):\penalty0 548--599, 2013.

\bibitem[Dang(2015)]{dang2015complex}
Nguyen~Viet Dang.
\newblock Complex powers of analytic functions and meromorphic renormalization
  in {QFT}.
\newblock \emph{arXiv:1503.00995}, 2015.

\bibitem[Dantzig(1948)]{dantzig1948programming}
George~B Dantzig.
\newblock Programming in a linear structure.
\newblock In \emph{Bulletin of the American Mathematical Society}, volume~54,
  pages 1074--1074, 1948.

\bibitem[Dantzig(1951)]{dantzig1951maximization}
George~B Dantzig.
\newblock Maximization of a linear function of variables subject to linear
  inequalities.
\newblock \emph{Activity Analysis of Production and Allocation}, 13:\penalty0
  339--347, 1951.

\bibitem[Dantzig(1955)]{dantzig1955uncertainty}
George~B Dantzig.
\newblock Linear programming under uncertainty.
\newblock \emph{Management Science}, 1:\penalty0 197 -- 206, 1955.

\bibitem[De~Loera et~al.(2010)De~Loera, Rambau, and
  Santos]{de2010triangulations}
Jes{\'u}s~A De~Loera, J{\"o}rg Rambau, and Francisco Santos.
\newblock \emph{Triangulations Structures for Algorithms and Applications}.
\newblock Springer, 2010.

\bibitem[del Barrio and Loubes(2019)]{del2019central}
Eustasio del Barrio and Jean-Michel Loubes.
\newblock Central limit theorems for empirical transportation cost in general
  dimension.
\newblock \emph{The Annals of Probability}, 47\penalty0 (2):\penalty0 926--951,
  2019.

\bibitem[del Barrio et~al.(1999)del Barrio, Cuesta-Albertos, Matr{\'a}n, and
  Rodr{\'\i}guez-Rodr{\'\i}guez]{delBarrio1999tests}
Eustasio del Barrio, Juan~A Cuesta-Albertos, Carlos Matr{\'a}n, and Jes{\'u}s~M
  Rodr{\'\i}guez-Rodr{\'\i}guez.
\newblock Tests of goodness of fit based on the {$L_2$}-{W}asserstein distance.
\newblock \emph{The Annals of Statistics}, 27\penalty0 (4):\penalty0
  1230--1239, 1999.

\bibitem[del Barrio et~al.(2019)del Barrio, Cuesta-Albertos, Matr{\'a}n, and
  Mayo-\'Iscar]{delBarrio2019robust}
Eustasio del Barrio, Juan~A Cuesta-Albertos, Carlos Matr{\'a}n, and
  A~Mayo-\'Iscar.
\newblock Robust clustering tools based on optimal transportation.
\newblock \emph{Statistics and Computing}, 29:\penalty0 139--160, 2019.

\bibitem[Dubuc et~al.(1999)Dubuc, Kagabo, and Marcotte]{dubuc1999note}
Serge Dubuc, Issa Kagabo, and Patrice Marcotte.
\newblock A note on the uniqueness of solutions to the transportation problem.
\newblock \emph{INFOR: Information Systems and Operational Research},
  37\penalty0 (2):\penalty0 141--148, 1999.

\bibitem[Dudley(2002)]{dudley2002real}
Richard~M Dudley.
\newblock \emph{Real {A}nalysis and {P}robability}, volume~74.
\newblock Cambridge University Press, 2002.

\bibitem[Dupa{\v{c}}ov{\'a}(1987)]{dupavcova1987stochastic}
Jitka Dupa{\v{c}}ov{\'a}.
\newblock Stochastic programming with incomplete information: a surrey of
  results on postoptimization and sensitivity analysis.
\newblock \emph{Optimization}, 18\penalty0 (4):\penalty0 507--532, 1987.

\bibitem[Dupa{\v{c}}ov{\'a} and Wets(1988)]{dupacova1988asymptotic}
Jitka Dupa{\v{c}}ov{\'a} and Roger Wets.
\newblock Asymptotic behavior of statistical estimators and of optimal
  solutions of stochastic optimization problems.
\newblock \emph{The Annals of Statistics}, pages 1517--1549, 1988.

\bibitem[Eichhorn and R{\"o}misch(2007)]{eichhorn2007stochastic}
Andreas Eichhorn and Werner R{\"o}misch.
\newblock Stochastic integer programming: Limit theorems and confidence
  intervals.
\newblock \emph{Mathematics of Operations Research}, 32\penalty0 (1):\penalty0
  118--135, 2007.

\bibitem[Ferguson and Dantzig(1956)]{ferguson1956allocation}
Allen~R Ferguson and George~B Dantzig.
\newblock The allocation of aircraft to routes: An example of linear
  programming under uncertain demand.
\newblock \emph{Management Science}, 3\penalty0 (1):\penalty0 45--73, 1956.

\bibitem[Frogner et~al.(2015)Frogner, Zhang, Mobahi, Araya, and
  Poggio]{frogner2015learning}
Charlie Frogner, Chiyuan Zhang, Hossein Mobahi, Mauricio Araya, and Tomaso~A
  Poggio.
\newblock Learning with a {W}asserstein loss.
\newblock In C.~Cortes, N.~D. Lawrence, D.~D. Lee, M.~Sugiyama, and R.~Garnett,
  editors, \emph{Advances in Neural Information Processing Systems 28}, pages
  2053--2061. Curran Associates, Inc., 2015.

\bibitem[Galichon(2018)]{galichon2018optimal}
Alfred Galichon.
\newblock \emph{Optimal transport methods in economics}.
\newblock Princeton University Press, 2018.

\bibitem[Gangbo and McCann(1996)]{gangbo1996geometry}
Wilfrid Gangbo and Robert~J McCann.
\newblock The geometry of optimal transportation.
\newblock \emph{Acta Mathematica}, 177\penalty0 (2):\penalty0 113--161, 1996.

\bibitem[Greenberg(1986)]{greenberg1986analysis}
Harvey~J Greenberg.
\newblock An analysis of degeneracy.
\newblock \emph{Naval Research Logistics Quarterly}, 33\penalty0 (4):\penalty0
  635--655, 1986.

\bibitem[Hadigheh and Terlaky(2006)]{hadigheh2006sensitivity}
Alireza~Ghaffari Hadigheh and Tam{\'a}s Terlaky.
\newblock Sensitivity analysis in linear optimization: Invariant support set
  intervals.
\newblock \emph{European Journal of Operational Research}, 169\penalty0
  (3):\penalty0 1158--1175, 2006.

\bibitem[Hitchcock(1941)]{hitchcock1941distribution}
Frank~L Hitchcock.
\newblock The distribution of a product from several sources to numerous
  localities.
\newblock \emph{Journal of Mathematics and Physics}, 20\penalty0
  (1-4):\penalty0 224--230, 1941.

\bibitem[Hoffman(1963)]{hoffman1963simple}
Alan~J Hoffman.
\newblock On simple linear programming problems.
\newblock In \emph{Proceedings of Symposia in Pure Mathematics}, volume~7,
  pages 317--327, 1963.

\bibitem[Kall and Mayer(1976)]{kall1976stochastic}
Peter Kall and Janos Mayer.
\newblock \emph{Stochastic linear programming}, volume~7.
\newblock Springer, 1976.

\bibitem[Kallenberg(1997)]{kallenberg1997foundations}
Olav Kallenberg.
\newblock \emph{Foundations of {M}odern {P}robability}.
\newblock Springer-Verlag, 2nd edition, 1997.

\bibitem[Kantorovich(1939)]{kantorovich1960mathematical}
Leonid~V Kantorovich.
\newblock Mathematical methods in the organization and planning of production.
\newblock \emph{Publication House of the Leningrad State University},
  6:\penalty0 336--422, 1939.

\bibitem[King and Rockafellar(1993)]{king1993asymptotic}
Alan~J King and R~Tyrrell Rockafellar.
\newblock Asymptotic theory for solutions in statistical estimation and
  stochastic programming.
\newblock \emph{Mathematics of Operations Research}, 18\penalty0 (1):\penalty0
  148--162, 1993.

\bibitem[Klatt et~al.(2020)Klatt, Tameling, and Munk]{klatt2020empirical}
Marcel Klatt, Carla Tameling, and Axel Munk.
\newblock Empirical regularized optimal transport: Statistical theory and
  applications.
\newblock \emph{SIAM Journal on Mathematics of Data Science}, 2\penalty0
  (2):\penalty0 419--443, 2020.

\bibitem[Klee and Witzgall(1968)]{klee1968facets}
Victor Klee and Christoph Witzgall.
\newblock Facets and vertices of transportation polytopes.
\newblock \emph{Mathematics of the Decision Sciences}, 1:\penalty0 257--282,
  1968.

\bibitem[Klinz and Woeginger(2011)]{klinz2011northwest}
Bettina Klinz and Gerhard~J Woeginger.
\newblock The northwest corner rule revisited.
\newblock \emph{Discrete Applied Mathematics}, 159\penalty0 (12):\penalty0
  1284--1289, 2011.

\bibitem[Luenberger and Ye(2008)]{luenberger2008linear}
David~G Luenberger and Yinyu Ye.
\newblock \emph{Linear and {N}onlinear {P}rogramming}.
\newblock Springer, New York, 2008.

\bibitem[McCann(1997)]{mccann1997convexity}
Robert~J McCann.
\newblock A convexity principle for interacting gases.
\newblock \emph{Advances in Mathematics}, 128\penalty0 (1):\penalty0 153--179,
  1997.

\bibitem[McCann(1999)]{mccann1999exact}
Robert~J McCann.
\newblock Exact solutions to the transportation problem on the line.
\newblock \emph{Proceedings of the Royal Society of London. Series A:
  Mathematical, Physical and Engineering Sciences}, 455\penalty0
  (1984):\penalty0 1341--1380, 1999.

\bibitem[Panaretos and Zemel(2019)]{panaretos2019statistical}
Victor~M Panaretos and Yoav Zemel.
\newblock {Statistical Aspects of {W}asserstein Distances}.
\newblock \emph{Annual Review of Statistics and Its Applications}, 6:\penalty0
  405--431, 2019.

\bibitem[Panaretos and Zemel(2020)]{panaretos2020invitation}
VM~Panaretos and Y~Zemel.
\newblock \emph{An Invitation to Statistics in {W}asserstein Space}.
\newblock Springer, Berlin, 2020.

\bibitem[Peyr{\'e} and Cuturi(2019)]{peyre2019computational}
Gabriel Peyr{\'e} and Marco Cuturi.
\newblock Computational optimal transport.
\newblock \emph{Foundations and Trends in Machine Learning}, 11\penalty0
  (5-6):\penalty0 355--607, 2019.

\bibitem[R{\"o}misch(2003)]{romisch2003stability}
Werner R{\"o}misch.
\newblock Stability of stochastic programming problems.
\newblock \emph{Handbooks in operations research and management science},
  10:\penalty0 483--554, 2003.

\bibitem[Rubner et~al.(2000)Rubner, Tomasi, and Guibas]{rubner2000earth}
Yossi Rubner, Carlo Tomasi, and Leonidas~J Guibas.
\newblock The earth mover's distance as a metric for image retrieval.
\newblock \emph{International Journal of Computer Vision}, 40\penalty0
  (2):\penalty0 99--121, 2000.

\bibitem[Ruszczy{\'n}ski and Shapiro(2003)]{ruszczynski2003stochastic}
Andrzej Ruszczy{\'n}ski and Alexander Shapiro.
\newblock Stochastic programming models.
\newblock \emph{Handbooks in Operations Research and Management Science},
  10:\penalty0 1--64, 2003.

\bibitem[Santambrogio(2015)]{santambrogio2015optimal}
Filippo Santambrogio.
\newblock \emph{Optimal {T}ransport for {A}pplied {M}athematicians}.
\newblock Birkh{\"a}user, Basel, 2015.

\bibitem[Shapiro(1989)]{shapiro1989asymptotic}
Alexander Shapiro.
\newblock Asymptotic properties of statistical estimators in stochastic
  programming.
\newblock \emph{The Annals of Statistics}, pages 841--858, 1989.

\bibitem[Shapiro(1991)]{shapiro1991asymptotic}
Alexander Shapiro.
\newblock Asymptotic analysis of stochastic programs.
\newblock \emph{Annals of Operations Research}, 30\penalty0 (1):\penalty0
  169--186, 1991.

\bibitem[Shapiro(1993)]{shapiro1993asymptotic}
Alexander Shapiro.
\newblock Asymptotic behavior of optimal solutions in stochastic programming.
\newblock \emph{Mathematics of Operations Research}, 18\penalty0 (4):\penalty0
  829--845, 1993.

\bibitem[Sierksma(2001)]{sierksma2001linear}
Gerard Sierksma.
\newblock \emph{{Linear and Integer Programming: Theory and Practice}}.
\newblock CRC Press, 2001.

\bibitem[Solomon et~al.(2015)Solomon, De~Goes, Peyr{\'e}, Cuturi, Butscher,
  Nguyen, Du, and Guibas]{solomon2015convolutional}
Justin Solomon, Fernando De~Goes, Gabriel Peyr{\'e}, Marco Cuturi, Adrian
  Butscher, Andy Nguyen, Tao Du, and Leonidas Guibas.
\newblock Convolutional {W}asserstein distances: Efficient optimal
  transportation on geometric domains.
\newblock \emph{ACM Transactions on Graphics (TOG)}, 34\penalty0 (4):\penalty0
  1--11, 2015.

\bibitem[Sommerfeld and Munk(2018)]{sommerfeld2018inference}
Max Sommerfeld and Axel Munk.
\newblock Inference for empirical {W}asserstein distances on finite spaces.
\newblock \emph{Journal of the Royal Statistical Society: Series B
  (Methodological)}, 80\penalty0 (1):\penalty0 219--238, 2018.

\bibitem[Sturmfels and Thomas(1997)]{sturmfels1997variation}
Bernd Sturmfels and Rekha~R Thomas.
\newblock Variation of cost functions in integer programming.
\newblock \emph{Mathematical Programming}, 77\penalty0 (2):\penalty0 357--387,
  1997.

\bibitem[Tameling et~al.(2019)Tameling, Sommerfeld, and
  Munk]{tameling2019empirical}
Carla Tameling, Max Sommerfeld, and Axel Munk.
\newblock Empirical optimal transport on countable metric spaces:
  Distributional limits and statistical applications.
\newblock \emph{The Annals of Applied Probability}, 29\penalty0 (5):\penalty0
  2744--2781, 2019.

\bibitem[Terlaky and Zhang(1993)]{terlaky1993pivot}
Tam{\'a}s Terlaky and Shuzhong Zhang.
\newblock Pivot rules for linear programming: a survey on recent theoretical
  developments.
\newblock \emph{Annals of Operations Research}, 46\penalty0 (1):\penalty0
  203--233, 1993.

\bibitem[Vershik(2013)]{vershik2013long}
Anatoly~Moiseevich Vershik.
\newblock Long history of the {M}onge--{K}antorovich transportation problem.
\newblock \emph{The Mathematical Intelligencer}, 35\penalty0 (4):\penalty0
  1--9, 2013.

\bibitem[Villani(2008)]{villani2008optimal}
C{\'e}dric Villani.
\newblock \emph{Optimal {T}ransport: {O}ld and {N}ew}.
\newblock Springer, Berlin, 2008.

\bibitem[Walkup and Wets(1969)]{walkup1969lifting}
David Walkup and Roger Wets.
\newblock Lifting projections of convex polyhedra.
\newblock \emph{Pacific Journal of Mathematics}, 28\penalty0 (2):\penalty0
  465--475, 1969.

\bibitem[Wang et~al.(2013)Wang, Slep{\v{c}}ev, Basu, Ozolek, and
  Rohde]{wang2013linear}
Wei Wang, Dejan Slep{\v{c}}ev, Saurav Basu, John~A Ozolek, and Gustavo~K Rohde.
\newblock A linear optimal transportation framework for quantifying and
  visualizing variations in sets of images.
\newblock \emph{International Journal of Computer Vision}, 101\penalty0
  (2):\penalty0 254--269, 2013.

\bibitem[Ward and Wendell(1990)]{ward1990approaches}
James~E Ward and Richard~E Wendell.
\newblock Approaches to sensitivity analysis in linear programming.
\newblock \emph{Annals of Operations Research}, 27\penalty0 (1):\penalty0
  3--38, 1990.

\bibitem[Wets(1966)]{wets1966programming}
Roger J-B Wets.
\newblock Programming under uncertainty: the equivalent convex program.
\newblock \emph{SIAM Journal on Applied Mathematics}, 14\penalty0 (1):\penalty0
  89--105, 1966.

\end{thebibliography}
